\documentclass[11pt]{amsart}
\usepackage{amsfonts}
\usepackage{fullpage, amsmath, amsthm,amsfonts,amssymb,stmaryrd, mathrsfs,amscd}
\usepackage{hyperref}
\usepackage[all]{xy}

\newtheorem{thm}{Theorem}[subsection]
\newtheorem{prop}[thm]{Proposition}
\newtheorem{lem}[thm]{Lemma}

\newtheorem{lem-def}[thm]{Lemma-Definition}
\newtheorem{cor}[thm]{Corollary}

\newtheorem{proposition}[thm]{Proposition}

\newtheorem{conjecture}[thm]{Conjecture}

\theoremstyle{definition}

\newtheorem{remark}[thm]{Remark}
\theoremstyle{definition}

\newtheorem{definition}[thm]{Definition}

\newtheorem{ex}[thm]{Example}

\newtheorem{construction}[thm]{Construction}

\newtheorem{rmk}[thm]{Remark}
\theoremstyle{definition}
\newtheorem{dfn}[thm]{Definition}

\numberwithin{equation}{subsection}

\input xy
\xyoption{all}

\newcommand{\quash}[1]{}  
\newcommand{\nc}{\newcommand}
\nc{\on}{\operatorname}

\DeclareFontEncoding{OT2}{}{} 

\DeclareMathOperator{\Ker}{Ker}

\def\bba{\mathbf{a}}

\def\FF{\mathbb{F}}

\def\ZZ{\mathbb{Z}}

\def\calE{\mathcal{E}}
\def\calF{\mathcal{F}}

\def\calH{\mathcal{H}}

\newcommand{\out}{\mathrm{out}}

\newcommand{\frakg}{{\mathfrak g}}

\newcommand{\frakp}{{\mathfrak p}}

\newcommand{\scrS}{\mathscr{S}}

\newcommand{\bA}{{\mathbb A}}

\newcommand{\bC}{{\mathbb C}}
\newcommand{\bD}{{\mathbb D}}

\newcommand{\bF}{{\mathbb F}}
\newcommand{\bG}{{\mathbb G}}
\newcommand{\bH}{{\mathbb H}}

\newcommand{\bM}{{\mathbb M}}
\newcommand{\bN}{{\mathbb N}}

\newcommand{\bP}{{\mathbb P}}
\newcommand{\bQ}{{\mathbb Q}}
\newcommand{\bR}{{\mathbb R}}

\newcommand{\bV}{{\mathbb V}}

\newcommand{\bX}{{\mathbb X}}

\newcommand{\bZ}{{\mathbb Z}}

\newcommand{\mA}{{\mathcal A}}

\newcommand{\mC}{{\mathcal C}}
\newcommand{\mD}{{\mathcal D}}
\newcommand{\mE}{{\mathcal E}}
\newcommand{\mF}{{\mathcal F}}
\newcommand{\mG}{{\mathcal G}}
\newcommand{\mH}{{\mathcal H}}

\newcommand{\mL}{{\mathcal L}}

\newcommand{\mO}{{\mathcal O}}

\newcommand{\mW}{{\mathcal W}}
\newcommand{\mX}{{\mathcal X}}

\newcommand{\mZ}{{\mathcal Z}}

\nc{\Tate}{\mathrm{Tate}}

\nc{\al}{{\alpha}} \nc{\be}{{\beta}} \nc{\ga}{{\gamma}}
\nc{\ve}{{\varepsilon}} \nc{\Ga}{{\Gamma}} \nc{\la}{{\lambda}}
\nc{\La}{{\Lambda}}

\nc{\ad}{{\on{ad}}}
\newcommand{\Ad}{{\on{Ad}}}
\nc{\aff}{{\on{aff}}}
\nc{\Aff}{{\mathbf{Aff}}}
\newcommand{\Aut}{{\on{Aut}}}
\nc{\Bun}{{\on{Bun}}}
\newcommand{\cha}{{\on{char}}}
\nc{\Coh}{{\on{Coh}}}
\nc{\der}{{\on{der}}}

\nc{\diag}{{\on{diag}}}
\nc{\dR}{{\on{dR}}}
\newcommand{\End}{{\on{End}}}
\nc{\Fl}{{\calF\ell}}
\newcommand{\Gal}{{\on{Gal}}}
\newcommand{\Gr}{{\on{Gr}}}
\newcommand{\Hom}{{\on{Hom}}}
\newcommand{\id}{{\on{id}}}
\nc{\Id}{{\on{Id}}}

\nc{\Ind}{{\on{Ind}}}
\nc{\inv}{{\on{Inv}}}
\nc{\Iso}{{\on{Isom}}}

\nc{\Nm}{{\on{Nm}}}

\nc{\pf}{{\on{pf}}}
\newcommand{\pr}{{\on{pr}}}
\nc{\rec}{{\on{rec}}}
\nc{\res}{{\on{res}}}
\nc{\Iw}{\on{Iw}}
\newcommand{\Res}{{\on{Res}}}

\DeclareMathOperator{\Spec}{Spec}

\nc{\tr}{{\on{tr}}}
\newcommand{\loc}{\mathrm{loc}}
\newcommand{\Tr}{{\on{Tr}}}

\newcommand{\Shv}{\mathrm{Shv}}

\nc{\ev}{\on{ev}}
\nc{\coev}{\on{coev}}
\nc{\un}{\on{un}}
\nc{\coun}{\on{coun}}

\newcommand{\GL}{{\on{GL}}}
\nc{\GSp}{{\on{GSp}}} \nc{\GU}{{\on{GU}}} \nc{\SL}{{\on{SL}}}
\nc{\SU}{{\on{SU}}} \nc{\SO}{{\on{SO}}}

\nc{\Ql}{{\overline{\bQ}_\ell}}

\nc{\fg}{\frakg}
\nc{\fp}{\frakp}
\nc{\pFr}{{_\mathrm{p}\sigma}}

\nc{\rat}{\overline{\bQ}}
\nc{\triv}{{\bf{1}}}

\newcommand{\longto}{\longrightarrow}

\newcommand{\wGr}{{\widetilde{\Gr}}}

\newcommand{\et}{\mathrm{et}}

\nc{\dm}{/\!\!/}

\def\xcoch{\mathbb{X}_\bullet}
\def\xch{\mathbb{X}^\bullet}

\nc{\wt}{\mathrm{wt}}
\nc{\Sat}{\on{Sat}}
\nc{\Sh}{\on{Sh}}
\nc{\Sht}{\on{Sht}}
\nc{\wSh}{\widetilde{\Sht}}
\nc{\Sph}{\on{Sph}}
\nc{\Fr}{\on{Frob}}
\nc{\Fp}{{^\sigma\mE}}
\nc{\und}{\underline}
\nc{\mmu}{{\mu_\bullet}}
\nc{\nnu}{{\nu_\bullet}}
\nc{\Hk}{{\on{Hk}}}
\nc{\lhk}{\Hk^{\on{loc}}}
\newcommand{\IC}{\on{IC}}
\nc{\bb}{{\mathbf{b}}}
\nc{\bp}{{\mathbf{p}}}

\nc{\MV}{{\bM\bV}}
\nc{\Mp}{{\on{Mp}}}

\newcommand{\leftone}{\leftarrow}
\newcommand{\rightone}{\to}
\nc{\bfB}{{\mathbf{B}}}
\newcommand{\rmP}{\mathrm{P}}

\begin{document}
\title{Geometric Satake, categorical traces, and arithmetic of Shimura varieties}
\author{Xinwen Zhu}
\address{Xinwen Zhu, Department of Mathematics, Caltech, 1200 East California Boulevard, Pasadena, CA 91125.}
\email{xzhu@caltech.edu}\date{\today}
\thanks{
The author was partially supported by NSF grant DMS-1602092 and the Sloan Fellowship.}
\begin{abstract}
We survey some recent work on the geometric Satake of $p$-adic groups and its applications to some arithmetic problems of Shimura varieties. We reformulate a few constructions appeared in the previous works more conceptually. 
\end{abstract}
\date{}
\maketitle

\tableofcontents

\section{Introduction}
The geometric Langlands program, initiated by Drinfeld and Laumon, arose as a continuation and generalization of  Drinfeld's approach to the Langlands correspondence for $\GL_2$ over a global function field. In the geometric theory, the fundamental object to study shifts from the space of automorphic forms of a reductive group $G$ to the category of sheaves on the moduli space of $G$-bundles on an algebraic curve. 

In recent years, the geometric Langlands program has found fruitful applications to the classical Langlands program and some related arithmetic problems. Traditionally, one applies Grothendieck's sheaf-to-function dictionary to
``decategorify" the category of sheaves studied in geometric theory to obtain the space of functions studied in arithmetic theory. This was used in Drinfeld's approach to the Langlands correspondence for $\GL_2$ (\cite{Dr}), as mentioned above. Another celebrated example is Ng\^o's proof the the fundamental lemma (\cite{Ng}). 
In this expository article, we explain \emph{another} passage from the geometric theory to the arithmetic theory, again via a trace construction, but of different nature. The abstract version of this method is discussed in \S \ref{S: cat trace}. We should mention that this idea already appeared  implicitly in V. Lafforgue's work on Langlands correspondence over global function fields (\cite{La}), and has also been explained in an informal article by Gaitsgory (\cite{Ga}).

One of the cornerstones of the geometric Langlands program is the geometric Satake equivalence, which establishes a tensor equivalence between the category of certain perverse sheaves on the affine Grassmannian of $G$ and the category of finite dimensional algebraic representations of the Langlands dual group $\hat G$. This is a vast generalization of the classical Satake isomorphism (via the sheaf-to-function dictionary), and can be regarded as a conceptual construction of the Langlands dual group.
For reductive groups over equal characteristic local fields, the geometric Satake is a result of works of Lusztig, Ginzburg, Beilinson-Drinfeld and Mirkovi\'{c}-Vilonen (\cite{Lu,Gi1,BD,MV}), which has been found many applications in representation theory, mathematical physics, and (arithmetic) algebraic geometry. The recent work \cite{Zh1} also established this equivalence for $p$-adic groups,  which will be  reviewed in \S \ref{S: GSat}. We will take the opportunity to discuss a motivic version of the geometric Satake in \S \ref{SS: MotSat}, which in some sense is a toy model for our arithmetic applications. Then  in \S \ref{S: cattrSat}, we apply the above mentioned abstract trace construction to the geometric Satake of $p$-adic groups, which leads to some applications to the study of the cohomology and cycles of Shimura varieties, as discussed in \S \ref{S:app to Sh}. These last two sections are based on the work \cite{XZ}.

\medskip

\noindent\bf Acknowledgement. \rm The article originates as a survey for the author's talk in the Current Developments in Mathematics 2016 conference. The author would like to thank organizers for the invitation. Section \ref{S: cattrSat} and \ref{S:app to Sh} are based on a joint work with Liang Xiao. The author would like to thank him for the collaboration. In addition, the author would like to thank Pavel Etingof, Dennis Gaitsgory and Yifeng Liu for useful discussions.

\section{Geometric Satake for $p$-adic groups}
\label{S: GSat}
In this section, we review the work \cite{Zh1}. We also take the opportunity to reformulate the geometric Satake in a motivic way. This makes some ideas behind the constructions in \cite{Zh1} more transparent and also gives a toy model of the arithmetic applications in \cite{XZ}, which will be discussed in later sections.
\subsection{Review of the geometric Satake for $p$-adic groups}
\label{S: affGrass}
We start with a brief review the classical Satake isomorphism. 

Let $F$ be a non-archimidean local field with $\mO$ its ring of integers and $k=\bF_q$ its residue field. I.e. $F$ is a finite extension of $\bQ_p$ or is isomorphic to $\bF_q((\varpi))$.
Let $\sigma$ be the geometric $q$-Frobenius of $k$.

We will assume that $G$ is a connected reductive group over $\mO$. 
Then $G(F)$ is topological group, with a basis of open neighborhoods of the unit element given by subgroups of $K:=G(\mO)$ of finite index. 
For example, if $G=\GL_2$, $F=\bQ_p$ and $\mO=\bZ_p$, let 
$$\Ga(p^n)=\big\{\begin{pmatrix}a& b\\ c& d\end{pmatrix}\mid a=d=1 \mod p^n,\  b=c=0 \mod p^n\big\}$$ 
denote the $n$th principal congruence subgroup. Then 
$$\GL_2(\bZ_p)=\Ga(1)\supset\Ga(p)\supset\Ga(p^2)\supset\cdots$$ form a basis of open neighbourhood of the identity matrix in $\GL_2(\bQ_p)$.

With this topology, $G(F)$ is locally compact, and $K$ is an open compact subgroup. Therefore, there is a unique Haar measure such that the volume of $K$ is $1$.
The classical spherical Hecke algebra is the space of compactly supported $G(\mO)$-bi-invariant $\bC$-valued functions on $G(F)$, equipped with the convolution product
\begin{equation}
\label{E:conv prod alge}
(f*g)(x)=\int_{G(F)}f(y)g(y^{-1}x)dy,
\end{equation}
Note that if both $f$ and $g$ are $\bZ$-valued, so is $f*g$. Therefore, the subset of $\bZ$-valued functions form a $\bZ$-subalgebra, which we denote by $H_G$.

It follows from the definition that the characteristic functions of the double cosets $\{KgK\mid g\in G\}$ form a $\bZ$-basis of $H_G$. We recall an explicit parameterization of these double cosets.

Let $T\subset G$ be a maximal torus defined over $\mO$. Let 
$$\xcoch(T)=\Hom(\bG_m,T)$$ denote the group of one-parameter subgroups of $T$, on which $\sigma$ acts.
Elements in $\xcoch(T)$ are usually called cocharacters. A $\sigma$-invariant cocharacter $\la$ is defined over $F$ and induces $\la: F^\times\to T(F)\subset G(F)$, and therefore for a choice of uniformizer $\varpi\in F$, we obtain a map
\begin{equation*}\label{cartan1}
\xcoch(T)^\sigma\to G(F), \quad \la\mapsto \la(\varpi).
\end{equation*}
The Cartan decomposition asserts that this map induces a canonical bijection
\begin{equation}\label{cartan2}
\inv:  K\backslash G(F)/K\cong \xcoch(T)^{\sigma}/W_0 
\end{equation}
which is independent of the choice of the uniformizer. Here $W=N_G(T)/T$ is the Weyl group of $T$, on which $\sigma$ acts, and $W_0=W^{\sigma}$.

\begin{ex}\label{coweightGLn}
If $G=\GL_n$ over $\bQ_p$, we can choose $T\subset G$ be the group of diagonal matrices. Then there is a canonical identification $\xcoch(T)\simeq \bZ^n$, on which $\sigma$ acts trivially. The above map $\xcoch(T)\to G(F)$ is then given by 
\[\mu=(m_1,m_2,\ldots,m_n)\mapsto \mu(p)=\begin{pmatrix}p^{m_1}&&&\\&p^{m_2}&&\\ &&\ddots&\\ &&& p^{m_n}\end{pmatrix}.\]
The Weyl group in this case is the permutation group $S_n$ of $n$ letters acting on $\bZ^n$ by permuting direct factors. Then every element in $\xcoch(T)/W$ admits a unique representative in
$$\xcoch(T)^+=\{\mu=(m_1,\ldots,m_n)\in \bZ^n\mid m_1\geq m_2\geq\cdots\geq m_n\}.$$
\end{ex}

We fix a coefficient ring $E$, which for simplicity is assumed to be an algebraically closed field of characteristic zero.
From $G$, one can construct another split connected reductive group $\hat G$ over $E$, called the Langlands dual group. The original definition of $\hat G$ is combinatoric, and relies on the the classification of connected reductive groups. A more conceptual construction is via the geometric Satake presented below (see Remark \ref{R: birth of the dual group}) so we do not recall the original definition of $\hat G$ here. For the moment, it is enough to know that $\hat G$ is equipped with a Borel subgroup $\hat B$, a maximal torus $\hat T\subset \hat B$ whose character group $\Hom(\hat T,\bG_m)$ is equal to $\xcoch(T)$, and that and the set of isomorphism classes of irreducible representations of $\hat G$ are parameterized by $\xcoch(T)/W$. In addition, the action of $\sigma$ induces a canonical finite order automorphism of $(\hat G,\hat B, \hat T)$.
Here is a list of examples to keep in mind (ignoring $(\hat B,\hat T)$ and the action of $\sigma$).

\medskip

\begin{center}
\begin{tabular}{| c | c | c | c | c | c |}
\hline 
$G$ & $\GL_n$ & $\on{SL}_n$ & $\on{SO}_{2n+1}$ & $\on{SO}_{2n}$ & $E_8$  \\ \hline 
$\hat G$ & $\GL_n$ & $\on{PGL}_{n}$ & $\on{Sp}_{2n}$ & $\on{SO}_{2n}$ & $E_8$ \\ \hline 
\end{tabular}
\end{center}

\medskip

Now consider the $\sigma$-twisted conjugation of $\hat G$ on itself given by
\begin{equation}
\label{E:twist conj}
c_\sigma:\hat G\times \hat G\to\hat G, \quad (g,h)\mapsto c_\sigma(g)(h):=gh\sigma(g)^{-1},
\end{equation}
which is equivalent to the usual conjugation action of $\hat G$ on the coset $\hat G\sigma\subset \hat G\rtimes\langle\sigma\rangle=:{^L}G$. Let $R({^L}G)$ denote the Grothendieck $K$-ring of the category $\on{Rep}({^L}G)$ of finite dimensional algebraic representations of ${^L}G$.
Let $\mathbf J= E[\hat G\sigma]^{\hat G}$ denote the space of $\sigma$-twisted conjugate invariant functions on $\hat G$. For a representation $V$ of ${^L}G$, let $\chi_V$ be the restriction of the character of $V$ to $\hat G\sigma$.
Then the map $R({^L}G)\otimes E\to \mathbf J$ sending $[V]$ to $\chi_V$ is surjective.

We fix a choice of $q^{1/2}\in E$.
The classical Satake isomorphism (or rather, Langlands' reinterpretation) establishes a canonical isomorphism
\begin{equation}\label{Sat isom}
\Sat^{cl}: H_G\otimes E\cong \mathbf J.
\end{equation}
This can be deduced via Grothendieck's sheaf-to-function dictionary from the geometric Satake which will be discussed below. So we do not discuss the proof here (but see \cite[\S 5.6]{Zh2} and \cite[\S 3.5]{XZ} for details). Instead, we give an example.

\begin{ex}
Let $G=\GL_2$. Then $\hat{G}=\GL_2$ on which $\sigma$ acts trivially so ${^L}G=\hat{G}$. As an algebra 
$$H_{\GL_2}=\bZ[T_p,S_p^{\pm 1}],$$
where $T_p$ is the characteristic function of $K\begin{pmatrix}p&\\&1\end{pmatrix}K$, and $S_p$ is the characteristic function of $\begin{pmatrix}p&\\&p\end{pmatrix}K$.
On the other hand, we know that 
$$R(\GL_2)=\bZ[\tr,\det{^{\pm 1}}],$$
where $\tr$ is the usual trace function (the character of the standard representation $\on{Std}$), and $\det$ is the usual determinant function (the character of $\wedge^2\on{Std}$). Then after choosing $p^{1/2}\in E$, the Satake isomorphism is given by
$$\Sat^{cl}(T_p)=p^{\frac{1}{2}}\cdot \tr,\quad \Sat^{cl}(S_p)=\det.$$
Note that 
\[T_p* T_p= (T_{p^2}+S_p)+pS_p,\quad \tr\cdot \tr= \chi_{\on{sym}^2\on{Std}}+\det,\]
where $T_{p^2}$ is the characteristic function of $K\begin{pmatrix}p^2&\\&1\end{pmatrix}K$, and $\chi_{\on{sym}^2\on{Std}}$ is the trace function of the symmetric square of the standard representation.
It follows that
\begin{equation}\label{upper tr}
\Sat^{cl}(T_{p^2}+S_p)=p\cdot\chi_{\on{sym}^2\on{Std}}.
\end{equation}
\end{ex}
\begin{rmk}
Note that it follows from the definition that both rings in \eqref{Sat isom} have a natural basis labelled by $\xcoch(T)^\sigma/W_0$. I.e. every $\mu\in \xcoch(T)^\sigma/W_0$ gives the characteristic function $1_{K\mu(\varpi)K}\in H_G$, and the function $\chi_{_{V_\mu}}\in \mathbf J$, where $V_\mu$ is the irreducible representation of ${^L}G$ of ``highest weight $\mu$".
However, the above example shows that the isomorphism $\Sat^{cl}$ does not send $1_{K\mu(\varpi)K}$ to $\chi_{_{V_\mu}}$ in general.
\end{rmk}

\begin{rmk}
\label{R: norm Sat}
Of course, it is possible to define an isomorphism $\Sat'^{cl}: H_{\GL_2}\to R(\GL_2)$ sending $T_p$ to $\tr$ (which is defined over $\bZ$ and is independent of the choice of $p^{1/2}$). This is related to the question of normalization of the Satake isomorphism (see \cite[\S 8]{Gr}). More fundamentally, it is related to the \emph{correct} definition of the Langlands dual group. We will not discuss this issue in this article (except Remark \ref{R:Tate twist}). A detailed discussion, from the point of view of the geometric Satake, can be found in \cite[\S 5.5]{Zh2}. 
\end{rmk}

\medskip

Now we explain the geometric approach, where instead of thinking $G(F)$ as a topological group and considering the space of $K$-bi-invariant compactly supported functions on it, we regard $G(F)$ as certain algebro-geometric object and study the category of $K$-bi-equivariant sheaves on it. It is equivalent, but more convenient to regard $G(F)/K$ as an algebro-geometric object and to study the category of $K$-equivariant sheaves on it. More precisely, we will construct an infinite dimensional algebraic variety $\Gr_G$ over $k$, called the affine Grassmannian of $G$, such that $G(F)/K$ is the set of its $k$-points. 
To see why this is possible, let us first analyze some subsets of $G(F)/K$, which can naturally be identified with the set of $k$-points of some algebraic varieties.

\begin{ex}\label{minu Sch}
Assume $G=\GL_n$ over $F$. Then the quotient $\GL_n(F)/K$ can be identified with the set of lattices in $F^n$. Here by a lattice, we mean a finite free $\mO$-submodule $\Lambda$ of $F^n$ such that $\Lambda\otimes_{\mO}F=F^n$. For example, $\La_0:=\mO^n\subset F^n$ is a lattice, usually called the standard lattice. Then the bijection is given by
\begin{equation}\label{coset vs lattice}
\GL_n(F)/K\simeq \big\{\mbox{Lattices in } F^n\big\},\quad gK\mapsto \Lambda:=g\La_0.
\end{equation}

We consider $\omega_i=1^i0^{n-i}\in \bZ^{n}$, regarded as a coweight of $\GL_n$ as in Example \ref{coweightGLn}. Then it is easy to see that 
under the identification \eqref{coset vs lattice}, the set $K\omega_i(\varpi)K/K$ can be identified with the set of lattices $\Lambda\subset F^n$ satisfying
\[\big\{\varpi\mO^n\subset \Lambda \subset \mO^n,\quad  \dim_{k}\mO^n/\Lambda=i\big\}.\]
The map 
$$\Lambda\mapsto (k^n=\mO^n/\varpi\mO^n\to \mO^n/\La)$$
identifies this set with the set of $i$-dimensional quotient spaces of $k^n$ (or equivalently, $(n-i)$-dimensional subspaces of $k^n$). To summarize, we have a canonical bijection between $K\omega_i(\varpi)K/K$ and the set of $k$-points of the Grassmannian variety $\Gr(n-i,n)$ of $(n-i)$-planes in a fixed $n$-dimensional space. It is natural to expect that $\Gr(n-i,n)$ is a subvariety of  $\Gr$.
\end{ex}

\begin{ex}\label{quasiminuSch}
We assume $G=\GL_2$ over $F$. According to the Satake isomorphism (in particular see \eqref{upper tr}), it is reasonable to consider the set 
\begin{equation}\label{Gr2}
K\begin{pmatrix}\varpi^2 & \\ &1 \end{pmatrix}K/K\sqcup \begin{pmatrix}\varpi & \\ &\varpi \end{pmatrix} K/K.
\end{equation} Under the identification \eqref{coset vs lattice}, it can be identified with the set of lattices $\La\subset F^2$ such that 
\begin{equation}\label{Gr2latt}
\{\La\subset \mO^2\mid \on{length}\mO^2/\La=2\}.
\end{equation}
If $F=k((\varpi))$, we may regard $\mO^2/\La$ as a $2$-dimensional quotient space of $k^4=\mO^2/\varpi^2\mO^2$. In addition, multiplication by $\varpi$ induces a nilpotent endomorphism on $\mO^2/\varpi^2\mO^2$, which stabilizes $\mO^2/\La$. It is easy to see that every $2$-dimensional quotient space of $\mO^2/\varpi^2\mO^2$ stable under this nilpotent endomorphism arises in this way. Therefore, \eqref{Gr2} can be identified with the set of $\bF_q$-points of a closed subvariety of $\Gr(2,4)$.

If $F=\bQ_p$, the situation is more complicated since we cannot regard $\mO^2/\La$ as a vector space over $k=\bF_p$. To overcome this difficulty, it is natural to consider another set
\begin{equation}\label{wGr2latt}
\{\La\subset\La'\subset \mO^2\mid \dim\mO^2/\La'=\La/\La'=1\}.
\end{equation}
There is a natural surjective map $m$ from \eqref{wGr2latt} to \eqref{Gr2latt} by forgetting $\La'$. It is an easy exercise to see that $m^{-1}(\La)$ consists of one element unless $\La=p\La_0$.
 
Using a reasoning similar as before, one shows that \eqref{wGr2latt} is identified with the set of $\bF_p$-points of the Hirzebruch surface $\bP(\mO_{\bP^1}(-1)+\mO_{\bP^1}(1))$. In addition, the set $p^{-1}(\La_0)$ form the set of $\bF_p$-points of a $(-2)$-curve on the surface. Blowing down this curve, one obtains a projective variety whose set of $\bF_p$-points is naturally identified with \eqref{Gr2latt}.
\end{ex}

Now we give a precise definition of the affine Grassmannian of $G$.  In the rest of the section, we allow $F$ to be slightly more general. Namely, we will assume that
$F$ is a local field complete with respect to a discrete valuation, with ring of integer $\mO$ and a \emph{perfect} residue field $k$ of characteristic $p>0$ \footnote{In equal characteristic situation, i.e. $\cha F=\cha k$, this assumption is not really necessary. We impose here to have a uniform treatment.}. Let $\varpi\in\mO$ be a uniformizer. 

Recall that a $k$-algebra $R$ is called perfect if the Frobenius endomorphism $R\to R, \ r\mapsto r^p$ is a bijection.
For a perfect $k$-algebra $R$, let $W(R)$ be the ring of ($p$-typical) Witt vectors of $R$. This is a ring whose elements are sequences $(r_0,r_1,r_2,\ldots)\in R^{\bN}$, with the addition and the multiplication given by certain (complicated) polynomials. The projection to the first component $W(R)\to R, \ (r_0,r_1,\ldots)\mapsto r_0$ is a surjective ring homomorphism, with a multiplicative (but \emph{non-additive}) section $R\to W(R), \ r\mapsto [r]=(r,0,0,\ldots)$, called the Teichm\"uller lifting of $r$. If $R$ is a perfect ring, every element in $W(R)$ can be uniquely written as $\sum_{i\geq 0} [a_i]p^i$ so $W(R)$ can be regarded as a ``power series ring in variable $p$ and with coefficients in $R$''.
For example, $W(\bF_p)=\bZ_p$. We define the ring of Witt vectors in $R$ with coefficient in $\mO$ as
\begin{equation}\label{ramified Witt vector}
W_\mO(R):= W(R)\hat\otimes_{W(k)}\mO:= \underleftarrow\lim_n W_{\mO,n}(R),\quad W_{\mO,n}(R):=W(R)\otimes_{W(k)} \mO/\varpi^n.
\end{equation}
In particular, if $R=\bar k$, we sometimes write $W_\mO(\bar k)$ by $\mO_L$ and $W_{\mO}(\bar k)[1/\varpi]$ by $L$.
Note that if $\cha F=\cha k$, then $W_\mO(R)\simeq R[[\varpi]]$, and $r\mapsto [r]$ a ring homomorphism. We sometimes write
\begin{equation}\label{discs}
D_{n,R}=\Spec W_{\mO,n}(R),\quad D_R=\Spec W_\mO(R),\quad D_R^*=\Spec W_\mO(R)[1/\varpi],
\end{equation}
which are thought as families of (punctured) discs parameterized by $\Spec R$.
Given two $G$-torsors $\mE_1$ and $\mE_2$ on $D_R$, a modification from $\mE_1$ to $\mE_2$ is an isomorphism $\mE_1|_{D_R^*}\simeq \mE_2|_{D_R^*}$. We usually denote it by $\mE_1\dashrightarrow\mE_2$ to indicate that the map is not defined over the whole $D_R$. 

We need to work with perfect algebraic geometry. Some foundations can be found in \cite[Appendix A]{Zh1}, \cite{BS} and \cite[\S A.1]{XZ}.
Let $\Aff_k^\pf$ denote the category of perfect $k$-algebras. We define the affine Grassmannian $\Gr_G$ of $G$ over $k$ as a presheaf over $\Aff_k^\pf$,
\begin{equation}
\label{E:affGrass def}
\Gr_G(R)= \left\{(\mE,\beta)\left | \begin{array}{l}\mE \mbox{ is a } G\mbox{-torsor on } D_R\textrm{ and }  \\
 \beta: \mE\dashrightarrow \mE^0 \mbox{ is a modification }\end{array}\right.\right\},
\end{equation}
where $\mE^0$ denotes the trivial $G$-torsor. If the group $G$ is clear, we write $\Gr_G$ by $\Gr$ for simplicity.
\begin{thm}\label{affGrass}
The affine Grassmannian $\Gr$ is represented as the inductive limit of subfunctors $\Gr=\underrightarrow\lim X_i$, with $X_i\to X_{i+1}$ closed embedding, and $X_i$ being perfections of projective varieties. 
\end{thm}
In the case when $\cha F=\cha k$, this is a classical result, due to Beauville-Laszlo, Faltings, etc (cf. \cite{BL, Fa}). In the mixed characteristic situation (i.e. $\cha F\neq \cha k$), this was conjectured in \cite[Appendix B]{Zh1}, and proved by Bhatt-Scholze (\cite{BS}). Previously a weaker statement that $X_i$ are perfections of proper algebraic spaces was proved in \cite{Zh1}, which allows one to formulate the geometric Satake in this setting.

\begin{rmk}
The category of perfect $k$-schemes is a full subcategory of the category of presheaves on $\Aff_k^\pf$, see \cite[Lemma A.10]{Zh1}. So the above theorem makes sense. 
\end{rmk}
\begin{rmk}One may wonder why we just consider the affine Grassmannian $\Gr$ as a presheaf on $\Aff_k^\pf$, rather than on the whole category $\Aff_k$ of $k$-algebras as usual. In equal characteristic, this is indeed possible and is the situation considered in \cite{BL, Fa}. But
in mixed characteristic, we do not know a correction definition of  $\Gr$ as a presheaf on $\Aff_k$, due to some pathology of taking ring of Witt vectors for non-perfect characteristic $p$ rings. On the other hand, passing to the perfection does not change the topology of a scheme so we do not loss any information if we are just interested in topologically problems related $\Gr$ (i.e. perverse sheaves on $\Gr$).
\end{rmk}
\begin{rmk}
\label{R:scholzediamond}
If $F=k((t))$, there exists the Beilinson-Drinfeld Grassmannian living over the disc $D$ (in fact, the Beilinson-Drinfeld Grassmannians are living over all Cartesian powers of $D$ over $k$), such that the above defined $\Gr$ appears as the fiber over the closed point $s\in D$. In mixed characteristic, the analogous geometric objects were recently constructed by Scholze, after he invented some new geometric structure called diamonds (cf. \cite{SW}). Using these objects, it literally makes sense to define the fiber product like $\Spec \bQ_p\times \Spec \bQ_p$ and it is possible to define the analogous Beilinson-Drinfeld Grassmannians over these fiber products.
\end{rmk}

Here is another useful interpretation of the affine Grassmannian as a homogeneous space. 
Let $H$ be an affine group scheme of finite type defined over $\mO$. We denote by $L^+H$ (resp. $LH$) the \emph{jet group} (resp. \emph{loop group}) of $H$. As presheaves,
\begin{equation*}\label{E:jet and loop}
L^+H(R)=H(W_\mO(R)),\quad LH(R)=H(W_\mO(R)[1/\varpi]).
\end{equation*}
It is represented by an affine group scheme (resp. ind-scheme). For $r\geq 0$, let $L^rH$ be the $r$th jet group, i.e. 
\begin{equation*}\label{E:rjet}
L^rH(R)=H(W_{\mO,r}(R)).
\end{equation*} 
Then $L^rH$ is represented by the perfection of an algebraic $k$-group (the usual Greenberg realization), and $L^+H=\underleftarrow\lim L^rH$. 

Now we can define a space $\Gr^{(\infty)}$ over $\Gr$ consisting of $(\mE,\beta)\in \Gr(R)$ and a trivialization $\epsilon:\mE\simeq \mE^0$. It turns out that this is an $L^+G$-torsor over $\Gr$ and there is a canonical isomorphism $\Gr^{(\infty)}\cong LG$ (see \cite[Lemma 1.3]{Zh1} or \cite[Proposition 1.3.6]{Zh2}). Therefore, one can realize $\Gr$ as the fpqc quotient
\begin{equation}
\label{E:gr as quotient}
\Gr\cong LG/L^+G.
\end{equation}
It follows that $L^+G$ acts naturally on $\Gr$ by left multiplication. In terms of the moduli problem \eqref{E:affGrass def}, it is the natural composition of $\beta$ with automorphisms $L^+G=\Aut(\mE^0)$.
\medskip

Next, we define some closed subvarieties of $\Gr$, including those studied in Example \ref{minu Sch} and Example \ref{quasiminuSch} as special cases.
First let $\mE_1$ and $\mE_2$ be two $G$-torsors over $D_{\bar k}=\Spec \mO_L$, and let $\beta: \mE_1\dashrightarrow\mE_2$ be a modification.
We attach to $\beta$ an element
$G(\mO_L)\backslash G(L)/G(\mO_L)$ as follows: by choosing isomorphisms $\epsilon_1:\mE_1\simeq \mE^0$ and $\epsilon_2:\mE_2\simeq \mE^0$, one obtains an automorphism of the trivial $G$-torsor $\epsilon_2\beta\epsilon_1^{-1}\in \Aut(\mE^0|_{D_{\bar k}^*})$ and therefore an element in $G(L)$. Different choices of $\epsilon_1$ and $\epsilon_2$ will modify this element by left and right multiplication by elements from $G(\mO_L)$. Therefore,
$\beta$ gives rise to a well-defined element in $G(\mO_L)\backslash G(L)/G(\mO_L)$. Via the bijection \eqref{cartan2}, we attach $\beta$ an element
\[\inv(\beta)\in \xcoch(T)/W.\]
We call $\inv(\beta)$ the relative position of $\beta$. 
\begin{ex}\label{invbetaGLn}
Let $G=\GL_n$. We identify $\GL_n$-torsors on $D$ with finite projective $\mO$-modules of rank $n$ in the usual way. Then $\inv(\beta)$ can be describe as follows.
We identify $\xcoch(T)\simeq \bZ^n$ as in Example \ref{coweightGLn}. 
Given two finite free $\mO$-modules of rank $n$ $\La_1$ and $\La_2$, and $\beta: \La_1[1/\varpi]\simeq \La_2[1/\varpi]$, there always exist a basis $(e_1,\ldots,e_n)$ of $\La_1$ and a basis $(f_1,\ldots,f_n)$ of $\La_2$ such that $\beta$ is given by 
$$\beta(e_i)=p^{m_i}f_i$$ 
and $m_1\geq m_2\geq\cdots\geq m_n$. In addition, this sequence $(m_1,\ldots,m_n)$ is independent of the choice of the basis. Then we define $\inv(\beta)=(m_1,\ldots,m_n)$.
\end{ex}

Now let $\mE_1$ and $\mE_2$ be two $G$-torsors over $D_R$, and let
$\beta: \mE_1\dashrightarrow \mE_2$ be a modification.
Then for each geometric point $x\in \Spec R$, by base change we obtain a triple $(\mE_1|_{D_x},\mE_2|_{D_x},\beta_x: \mE_1|_{D_x}\dashrightarrow \mE_2|_{D_x})$. Let $\inv_x(\beta):=\inv(\beta_x)$.

We define the (spherical) Schubert cell $\mathring{\Gr}_{\mu}\subset \Gr\otimes \bar k$ as
\[\mathring{\Gr}_\mu:=\left\{(\mE,\beta)\in \Gr\otimes \bar k\mid \inv(\beta)=\mu\right\}.\]
One can show that it is locally closed in $\Gr\otimes \bar k$ and forms an $L^+G\otimes \bar k$-orbit through $\mu(\varpi)$. Its closure, denoted by $\Gr_\mu$ is called a (spherical) Schubert variety\footnote{Note that in literature sometimes Schubert varieties are denoted by $\Gr_{\leq\mu}$ or $\overline{\Gr}_\mu$ while Schubert cells are denoted by $\Gr_\mu$. We hope our notation is not too confusing.}. It is a union of Schubert cells. Then one can define a partial order (usually called the Bruhat order) on $\xcoch(T)/W$ as follows:
\[\la\leq \mu \quad \mbox{ iff } \quad \Gr_{\la}\subset \Gr_\mu.\]
\begin{rmk}
Usually, the Bruhat order is defined combinatorially in terms of the root datum associated to $G$. For our discussion in the sequel, the above definition will suffice.
\end{rmk}
\begin{ex}
Let $G=\GL_n$. We identify $\xcoch(T)/W$ with $\xcoch(T)^+$ as in Example \ref{invbetaGLn}. 
Then Bruhat order can be explicitly described as follows: $\mu_1=(m_1,\ldots,m_n)\leq \mu_2=(l_1,\ldots,l_n)$ if 
$$m_1+\cdots+m_j\leq l_1+\cdots+l_j, \ \ \ j=1,\ldots, n, \quad \mbox{ and }\ \ m_1+\cdots+m_n=l_1+\cdots+l_n.$$ 
\end{ex}

The Schubert variety $\Gr_\mu$ is a projective variety in general defined over $\bar k$\footnote{It is in fact defined over the field of definition of $\mu$.}. It is perfectly smooth if and only if $\Gr_\mu=\mathring{\Gr}_\mu$, in which case $\mu$ is a minimal element in $\xcoch(T)/W$ with respect to the Bruhat order. Such $\mu$ is called \emph{minuscule}.

\begin{ex}
Let $G=\GL_n$. Let $\omega_i=1^i0^{n-i}$ as in Example \ref{minu Sch}, and let $\omega_i^*=\omega_{n-i}-\omega_{n}$. 
Note that  $\inv(\beta)=\omega_i$ if and only if $\beta$ extends to a genuine map $\La_1\to \La_2$  such that $\varpi \La_2\subset \La_1$ and $\La_2/\La_1$ is a $k$-vector space of dimension $i$. Similarly, $\inv(\beta)=\omega_i^*$ if and only if $\beta^{-1}$ induces the inclusions $\varpi \La_1\subset \La_2\subset \La_1$ such that $\La_1/\La_2$ is of dimension $i$.
It is not hard to see that $\omega_i,\omega_i^*$ are minuscule, and $\Gr_{\omega_i}\cong \Gr_{\omega_{n-i}^*}\cong \Gr(n-i,n)$. Taking $k$-points recovers Example \ref{minu Sch}. 
\end{ex}

Next we briefly discuss the usual formulation of the geometric Satake and refer to \cite{Zh2} for a much more detailed discussion (at least in the equal characteristic case). 
A motivic version will be discussed in the next subsection.
For simplicity, we assume that $k=\bar{k}$ in the rest of this subsection, and denote $L^+G$ by $K$ and $L^nG$ by $K_n$ (so contrary to the previous notations, in this rest of this subsection $K$ is a pro-algebraic group rather than a pro-finite group). Let $K^{(n)}=\ker(K\to K_n)$ denote the $n$th congruence subgroup. We fix a coefficient ring $E$ in which $\on{char} k$ is invertible. Sheaves will mean $E$-sheaves.

We first briefly recall the definition of the category $\on{P}_{K}(\Gr)$ of perverse sheaves on the affine Grassmannian, usually called the \emph{Satake category}, and denoted by $\Sat_G$. 
First, we can choose a presentation $\Gr=\underrightarrow\lim X_i$ such that each $X_i$ is a $K$-invariant closed subscheme (e.g. we can choose $X_i$ to be a finite union of Schubert varieties). In addition, the action of $K$ on $X_i$ factors through some $K_m$.
Then one can define 
\[\Sat_G=\on{P}_{K}(\Gr)=\underrightarrow\lim\on{P}_{K}(X_i),\]
where $\on{P}_{K}(X_i)$ is defined as the category of $K_m$-equivariant perverse sheaves on $X_i$ (which is independent of the choice of $m$ up to a canonical equivalence) and the connecting functor $\on{P}_{K}(X_i)\to \on{P}_{K}(X_{i'})$ is the pushforward along the closed embedding $X_i\to X_{i'}$. We refer to \cite[\S A.1.17]{Zh2} for a detailed discussion of the definition of this category.

It turns out that this is a semisimple abelian monoidal category with the monoidal structure given by Lusztig's convolution product of sheaves. We recall the definition. Let $\Gr\tilde\times \Gr$ be the twisted product of affine Grassmannians (also known as the convolution Grassmannian). In terms moduli problem, it is defined as the presheaf
\begin{equation*}\label{I:conv}
(\Gr\tilde\times\Gr)(R)=\left\{(\mE_1,\mE_2,\beta_1,\beta_2)\ \left|\ \begin{split}&\mE_1,\mE_2 \mbox{ are }  G\mbox{-torsors on } D_R, \\
&  \beta_1:\mE_1\dashrightarrow \mE^0, \beta_2:\mE_2\dashrightarrow  \mE_1\end{split}\right.\right\}.
\end{equation*}
There exist a natural map
\begin{equation}\label{conv m0}
m:\Gr\tilde\times\Gr\to \Gr,\quad (\mE_1,\mE_2,\beta_1,\beta_2)\mapsto (\mE_2, \beta_1\beta_2).
\end{equation}
and a natural projection
$$\pr_1:\Gr\tilde\times\Gr\to \Gr, \quad (\mE_1,\mE_2,\beta_1,\beta_2)\mapsto (\mE_1,\beta_1),$$
which together induce an isomorphism
$(\pr_1,m):\Gr\tilde\times\Gr\cong \Gr\times \Gr$.
In particular, the convolution Grassmannian is representable.  The map $m$ is usually called the convolution map. 

More generally, there exists the $r$-fold convolution Grassmannian $\Gr\tilde\times\cdots\tilde\times\Gr$, classifying a chain of modifications of $G$-torsors
\begin{equation}
\label{E:chain mod}
\mE_r\stackrel{\beta_r}{\dashrightarrow}\mE_{r-1}\stackrel{\beta_{r-1}}{\dashrightarrow}\cdots\stackrel{\beta_0}{\dashrightarrow}\mE_0=\mE^0.
\end{equation}
There exist a natural projection $\pr_j$ to the $j$-fold convolution Grassmannian for each $j=1,\ldots,r-1$ by remembering $\mE_j\stackrel{\beta_j}{\dashrightarrow}\cdots\stackrel{\beta_0}{\dashrightarrow}\mE_0=\mE^0$, and
an $r$-fold convolution map
\begin{equation}\label{conv m}
m:\Gr\tilde\times\cdots\tilde\times\Gr\to \Gr,
\end{equation}
sending \eqref{E:chain mod} to $(\mE_r, \beta_1\cdots\beta_r)$. As before, $K$ acts naturally on $\Gr\tilde\times\cdots\tilde\times\Gr$ as automorphisms of $\mE_0=\mE^0$, and there is a $K$-torsor $(\Gr\tilde\times\cdots\tilde\times\Gr)^{(\infty)}\to\Gr\tilde\times\cdots\tilde\times\Gr$, which classifies \eqref{E:chain mod} together with an isomorphism $\epsilon:\mE_r\simeq \mE^0$.

Using the isomorphism \eqref{E:gr as quotient}, it is easy to see that $\Gr\tilde\times\Gr\cong \Gr^{(\infty)}\times^K\Gr=LG\times^K\Gr$
with the convolution map induced by the action map $LG\times\Gr\to \Gr$. 
This motivates us to define the twisted product of $K$-invariant closed subsets of $\Gr$. Namely,
for a closed subset $X\subset \Gr$, let $X^{(\infty)}$ denote its preimage in $\Gr^{(\infty)}$. So $X^{(\infty)}\to X$ is a $K$-torsor, and for every $n$, let $X^{(n)}=X^{(\infty)}/K^{(n)}\subset LG/K^{(n)}$, which is a $K_n$-torsor over $X$.
Then given two closed subsets $X_1, X_2$ with $X_2$ being $K$-invariant, we denote their twisted product as
\begin{equation}
\label{E-conv Sch}
X_1\tilde\times X_2:= X_1^{(\infty)}\times^K X_2,
\end{equation}
which is a closed subsets of $\Gr\tilde\times\Gr$ (and therefore is representable). If $X_2$ is quasi-compact, there is an integer $m$ sufficiently large so that the action of $K$ on $X_2$ factors through $K_m$. 
Then $X_1^{(\infty)}\times^K X_2=X_1^{(m)}\times^{K_m}X_2$.  Note that
if $X_i=\Gr_{\mu_i}$ for $\mu_i\in \xcoch(T)/W$, we can alternatively describe $\Gr_{\mu_1}\tilde{\times}\Gr_{\mu_2}$ as
\[\Gr_{\mu_1}\tilde{\times}\Gr_{\mu_2}\cong \left\{(\mE_1,\mE_2,\beta_1,\beta_2)\in\Gr\tilde\times\Gr\mid \inv(\beta_1)\leq \mu_1, \inv(\beta_2)\leq \mu_2\right\}.\]

More generally, we can define the twisted product of a closed subset $X_1$ in the $r$-fold convolution Grassmannian and a $K$-invariant closed subset $X_2$ in the $s$-fold convolution Grassmannian by the same formula in \eqref{E-conv Sch}, with $X_1^{(\infty)}$ understood as the preimage of $X_1$ in $(\Gr\tilde\times\cdots\tilde\times\Gr)^{(\infty)}$.
In particular,  if $\mmu=(\mu_1,\ldots,\mu_r)$ is a sequence of dominant coweights of $G$, one can define the twisted product of Schubert varieties
$$\Gr_{\mmu}:=\Gr_{\mu_1}\tilde\times\cdots\tilde\times\Gr_{\mu_r}\subset\Gr\tilde\times\cdots\tilde\times\Gr,$$
and if $\nu_\bullet=(\nu_1,\ldots,\nu_s)$ is another sequence, then $\Gr_{\mmu}\tilde\times\Gr_{\nu_\bullet}\cong\Gr_{\mmu,\nu_\bullet}$.
Let $|\mu_\bullet|=\sum \mu_i$. Then the convolution map \eqref{conv m} induces
\begin{equation}\label{conv m2}
m: \Gr_{\mu_\bullet}\to \Gr_{|\mu_\bullet|}, \quad \mE_r\stackrel{\beta_r}{\dashrightarrow}\cdots\stackrel{\beta_0}{\dashrightarrow}\mE_0=\mE^0\mapsto (\mE_r,\beta_1\cdots\beta_r),
\end{equation}
which is known to be a semismall map. This implies that
\begin{equation}
\label{E:dim est}
\dim (\Gr_{\mmu}\times_\Gr\Gr_{\nu_\bullet})\leq \frac{1}{2}(d_{\mmu}+d_{\nu_\bullet}),
\end{equation}
where we denote $d_{\mmu}=\dim \Gr_{\mmu}$. (See \cite[Proposition 3.1.10]{XZ}.) 

Now, for $\mA_1,\mA_2\in \on{P}_{K}(\Gr)$,
we denote by $\mA_1\tilde\boxtimes\mA_2$ the ``external twisted product'' of $\mA_1$ and $\mA_2$ on $\Gr\tilde\times\Gr$. It is the unique perverse sheaf on $\Gr\tilde\times \Gr$ whose pullback to $\on{Supp}(\mA_1)^{(m)}\times \on{Supp}(\mA_2)$ is isomorphic to the external product of the pullback of $\mA_1$ to $\on{Supp}(\mA_1)^{(m)}$ and $\mA_2$. Here $\on{Supp}(\mA_i)$ is the support of $\mA_i$, which is $K$-invariant, and $m$ is sufficiently large so that the action of $K$ on $\on{Supp}(\mA_2)$ factors through $K_m$.
For example, if $\mA_i=\IC_{\mu_i}$, then $\IC_{\mu_1}\tilde\boxtimes\IC_{\mu_2}$ is canonically isomorphic to the intersection cohomology sheaf of $\Gr_{\mu_1}\tilde\times\Gr_{\mu_2}$.
Then the convolution product of $\mA_1$ and $\mA_2$ is defined as
\begin{equation}\label{Lus conv}
\mA_1\star\mA_2:=m_!(\mA_1\tilde\boxtimes\mA_2),
\end{equation}
where $m:\Gr\tilde\times\Gr\to \Gr$ is the convolution map (defined by \eqref{conv m0}).
A priori, this is a $K$-equivariant complex on $\Gr$. But because of the semismallness of the convolution map \eqref{conv m2}, it is a perverse sheaf.

The twisted product $\Gr\tilde\times\Gr$ is ``topologically isomorphic to" the product $\Gr\times\Gr$. If $F=\bC((\varpi))$, one can make this precise using the analytic topology on $\Gr$. Namely, under the analytic topology, the $K$-torsor $\Gr^{(\infty)}\to \Gr$ is trivial.  In general, this can be understood motivically. 
Then when $E=\Ql$ ($\ell\neq \cha k$), applying the K\"unneth formula, one can endow the hypercohomology functor 
$$\on{H}^*(\Gr_G,-) : \on{P}_{K}(\Gr_G,\Ql)\to \on{Vect}_{\Ql}$$ with a canonical monoidal structure (see \cite[\S 2.3]{Zh1} and \cite[\S 5.2]{Zh2}). 

This following theorem is usually referred as the geometric Satake equivalence. We assume that $E=\Ql$.

\begin{thm}\label{intro:geomSat}
The monoidal functor $\on{H}^*$ factors as the composition of an equivalence of monoidal categories from $\Sat_G$ to the category $\on{Rep}_{\Ql}(\hat{G})$ of finite dimensional representations of the Langlands dual group $\hat{G}$ over  $\Ql$ and the forgetful functor from $\on{Rep}_{\Ql}(\hat{G})$ to the category $\on{Vect}_{\Ql}$ of finite dimensional $\Ql$-vector spaces. Under the equivalence, the intersection cohomology sheaf $\IC_\mu$ of the Schubert variety $\Gr_\mu$ corresponds to the irreducible representation $V_\mu$ of $\hat G$ of ``highest weight" $\mu$.
\end{thm}

\begin{rmk}
\label{R: birth of the dual group}
(1) Indeed, it is more canonical to \emph{define} the dual group $\hat G$ of $G$ as the Tannakian group of the Tannakian category $(\Sat_G,\on{H}^*)$, and \emph{define} $V_\mu$ as $\on{H}^*(\Gr,\IC_\mu)$, equipped with the tautological action $\hat G=\Aut^\otimes(\on{H}^*)$. In the rest of the article, we will take this point of view.

(2) As explained in \cite[\S 5.3]{Zh2}, the Tannakian group is canonically equipped with a pinning $(\hat G, \hat B, \hat T, \hat X)$. We briefly recall the construction of $(\hat B, \hat T)$ and refer to \emph{loc. cit.} for more details.
The grading on the cohomology $\on{H}^*$ defines a cocharacter $\bG_m\to \hat G$, which can be shown to be regular. Then its centralizer gives a maximal torus $\hat T$, and $\hat B\supset \hat T$ is the unique Borel in $\hat G$ with respect to which this cocharacter is dominant.
\end{rmk}

When $F=k((t))$ is an equal characteristic local field, this theorem is a result of works of Lusztig, Ginzburg, Beilinson-Drinfeld and Mirkovi\'{c}-Vilonen (cf. \cite{Lu,Gi1,BD,MV}). When $F$ is a mixed characteristic local field, this theorem was proved in \cite{Zh1}. 

The most difficult part of theorem is the construction of a commutativity constraint for the monoidal structure on $\on{P}_{K}(\Gr_G)$ such that $\on{H}^*$ becomes a tensor functor. In equal characteristic, the construction relies on an interpretation of the convolution product as fusion product, which crucially uses the existence of Beilinson-Drinfeld Grassmannians over Cartesian powers of an algebraic curve over $k$. As mentioned in Remark \ref{R:scholzediamond}, the analogous geometric objects in mixed characteristic have been constructed by Scholze recently. Scholze also announced recently that it is possible to define the fusion product and establish the geometric Satake in mixed characteristic in a fashion similar to the equal characteristic case. 

The approach used in \cite{Zh1} is different. We constructed a commutativity constraint using a categorical version of the classical Gelfand's trick. This idea is not new, and already appeared in \cite{Gi1}. Therefore, we do have a candidate of the commutativity constraint even in mixed characteristic. The problem is that it is not clear how to verify the properties it supposes to satisfy (e.g. the hexagon axiom), without using the fusion interpretation.

Our new observation is that the validity of these properties is equivalent to a numerical result for the affine Hecke algebra. Namely, in \cite{LV,Lu2} Lusztig and Vogan introduced, for a Coxeter system $(W,S)$ with an involution, certain polynomials $P^\sigma_{y,w}(q)$ similar to the usual Kazhdan-Lusztig polynomials $P_{y,w}(q)$. Then it was conjectured in \cite{Lu2} that if $(W,S)$ is an affine Weyl group and $y,w$ are certain elements in $W$,
\[P_{y,w}^\sigma(q)=P_{y,w}(-q).\]
See \emph{loc. cit.} for more details. This conjecture is purely combinatoric, but its proof by Lusztig and Yun \cite{LY} is geometric, which in fact uses the equal characteristic geometric Satake! We then go in the opposite direction by showing that this formula implies that the above mentioned commutativity constraint is the correct one.  

So our proof of Theorem \ref{intro:geomSat} in mixed characteristic uses the geometric Satake in equal characteristic. It is an interesting question to find a direct proof of the above combinatoric formula, which will yield a purely local proof of the geometric Satake, in both equal and mixed characteristic. In addition, a similar strategy should be applicable to proving the geometric Satake for ramified $p$-adic group (whose equal characteristic counterparts were established in \cite{Zh0, Ri}).

\subsection{The motivic Satake category}
\label{SS: MotSat}
To explain some ideas behind the later application of the geometric Satake to arithmetic geometry of Shimura varieties, we would like to reformulate the geometric Satake motivically. 
We have to assume that readers are familiar with basic theory of (pure) motives in this subsection. A classical reference is \cite{Kl}. 

We assume that $k$ is a perfect field in this subsection.
We will construct a $\bQ$-linear semisimple super Tannakian subcategory $\Sat^m$ inside the category of (numerical) motives. One can recover (a variant of) the usual Satake category by tensoring $\Sat^m$ with $\Ql$.
To avoid some complicated geometry, we will mainly focus on the $\GL_n$ case and briefly mention modifications needed to deal with general reductive groups at the end of this subsection.

Let $\on{Var}_k$ denote the category of smooth projective varieties over $k$.
Let $\on{Mot}^{\on{rat}}_k$ (resp. $\on{Mot}_k^{\on{num}}$) denote the $\bQ$-linear category of pure motives over $k$ with respect to rational (resp. numerical) equivalence relations. Then $\on{Mot}_k^*$ (for $*=\on{rat},\on{num}$) is a symmetric monoidal $\bQ$-linear pseudo-abelian category, and $\on{Mot}_k^{\on{num}}$ is a super Tannakian semisimple $\bQ$-linear abelian category (see \cite{Ja}). Following \cite{Ja}, objects in $\on{Mot}_k^{*}$ are denoted by triples $(X,p,m)$, where $X$ is a smooth projective variety over $k$, $p$ is an idempotent in the ring of degree zero self-correspondences of $X$ (defined via one of above adequate relations), and $m\in\bZ$. Then there is the natural  functor
\[h:\on{Var}_k^{\on{op}}\to \on{Mot}^*_k,\quad X\mapsto h(X):=(X,\id,0).\]
As usual, we write $h(X)(m)$ for $(X,\id,m)$.
The functor $h$ factors through the localization of $\on{Var}_k$ with respect to universal homeomorphisms.
Note that $\cha k=p>0$, the perfection functor from $\on{Var}_k$ to the category $\on{Sch}_k^\pf$ of perfect $k$-schemes factors through this localization and in fact induces a full embedding of this localized category to $\on{Sch}_k^\pf$ (by \cite[Corollary A.16, Proposition A.17]{Zh1}, see also \cite[Lemma 3.8, Proposition 3.11]{BS}). The essential image is denoted by $\on{Var}^\pf_k$, consisting of perfect schemes over $k$ as the perfection of smooth projective schemes.
Therefore, we obtain a functor
$$h:(\on{Var}^\pf_k)^{\on{op}}\to\on{Mot}^*_k.$$

Recall that given a sequence $\mmu=(\mu_1,\ldots,\mu_r)$, there is the twisted product $\Gr_{\mmu}=\Gr_{\mu_1}\tilde\times\cdots\tilde\times\Gr_{\mu_r}$ of Schubert varieties, whose dimension is denoted by $d_\mmu$. If (and only if) each $\mu_i$ is minuscule, $\Gr_\mmu$ is the perfection of a smooth projective variety and therefore defines an object in $\on{Var}^\pf_k$. In the case of $\GL_n$, $\Gr_\mmu$ is an iterated fibration by perfect Grassmannian varieties, 
classifying chains of lattices $\{\La_i, i=0,\ldots, r\}$ satisfying
$$\La_0=\mO^n, \quad  \varpi^{d_i+1}\La_{i-1}\subset \La_i\subset \varpi^{d_i}\La_{i-1} \mbox{ for some } d_i\in\bZ.$$
(Here and in the rest of this subsection, we only describe $k$-points of $\Gr_\mmu$ for simplicity.) 
Let $\on{H}(\Gr_\mmu, \Gr_{\nu_\bullet})$ be the $\bQ$-vector space with a basis given by irreducible components of $\Gr_{\mmu}\times_\Gr\Gr_{\nu_\bullet}$ of dimension \emph{exactly} $\frac{1}{2}(d_{\mmu}+d_{\nu_\bullet})$. We need the following geometric fact.
\begin{lem}
\label{L:faithful realization}
Let $m,n,r$ be three integers satisfying $2(m-n)=d_{\mmu}-d_{\nu_\bullet}$ and $2(n-r)=d_{\nu_\bullet}-d_{\la_\bullet}$. 
Then the natural map
\[\on{H}(\Gr_{\mmu},\Gr_{\nu_\bullet})\to \Hom_{\on{Mot}^*_k}(h(\Gr_{\mmu})(m),h(\Gr_{\nu_\bullet})(n))\]
is injective. In addition, there is a unique $\bQ$-linear map 
$$\on{H}(\Gr_{\mmu},\Gr_{\nu_\bullet})\otimes \on{H}(\Gr_{\nu_\bullet},\Gr_{\la_\bullet})\to \on{H}(\Gr_{\mmu},\Gr_{\la_\bullet})$$ 
making the following diagram commutative
\begin{small}
\[
\begin{CD}
\on{H}(\Gr_{\mmu},\Gr_{\nu_\bullet})@.\  \otimes \ @.\on{H}(\Gr_{\nu_\bullet},\Gr_{\la_\bullet})@>>> \on{H}(\Gr_{\mmu},\Gr_{\la_\bullet})\\
@VVV@.@VVV@VVV\\
\Hom_{\on{Mot}^*_k}(h(\Gr_{\mmu})(m),h(\Gr_{\nu_\bullet})(n))@.\otimes @. \Hom_{\on{Mot}^*_k}(h(\Gr_{\nu_\bullet})(n),h(\Gr_{\la_\bullet})(r))@>>>\Hom_{\on{Mot}^*_k}(h(\Gr_{\mmu})(m),h(\Gr_{\la_\bullet})(r)).
\end{CD}
\]
\end{small}
\end{lem} 
\begin{proof}
We give a sketch.
Since $\Gr_\mmu\times\Gr_{\nu_\bullet}$ admits a cellular decomposition (by the convolution product of semi-infinite orbits, e.g. using \cite[\S 3.2.5, \S 3.2.6]{XZ}), rational equivalence, homological equivalence and numerical equivalence. Then it is enough to show that the cycle classes of those $\frac{1}{2}(d_{\mmu}+d_{\nu_\bullet})$-dimensional irreducible components of $\Gr_{\mmu}\times_\Gr\Gr_{\nu_\bullet}$ in the \'etale cohomology of $\Gr_{\mmu}\times\Gr_{\nu_\bullet}$ are linearly independent. But this follows from the usual decomposition theorem for the convolution map \eqref{conv m2} (which is semismall).

Given the injectivity and the definition of the composition of correspondences, to prove the commutativity of the diagram, it is enough to note and inside $\Gr_{\mmu}\times\Gr_{\nu_\bullet}\times\Gr_{\la_\bullet}$ the intersections of cycles from $\on{H}(\Gr_{\mmu},\Gr_{\nu_\bullet})$ and $\on{H}(\Gr_{\nu_\bullet},\Gr_{\la_\bullet})$ are proper, and then to apply the dimension estimate \eqref{E:dim est}.
\end{proof}

Now we define a $\bQ$-linear additive category $\Sat_G^{0}$ and define the \emph{motivic Satake category} $\Sat_G^m$ as the idempotent completion of $\Sat_G^{0}$. We note that idea of defining $\Sat_G^0$ already appeared in \cite[\S 4.3]{FKK}. 
\begin{itemize}
\item Objects are pairs $(X,f)$, where $X$ is a disjoint union of perfect smooth projective schemes of the form $\Gr_\mmu$ for $\mmu$ being a sequence of minuscule coweights, and $f$ is a $\bZ$-valued locally constant function on $X$. If $X=\Gr_\mmu$, then $f$ is given by an integer $m$. In this case we denote the pair by $(\Gr_\mmu,m)$. If in addition $m=0$, sometimes we simply denote $(\Gr_\mmu,0)$ by $\Gr_\mmu$.
\item \[\Hom_{\Sat_G^0}((\Gr_{\mmu},m),(\Gr_{\nu_\bullet},n))=\left\{\begin{array}{ll}0 & 2(m-n)\neq d_\mmu-d_{\nu_\bullet}\\
\on{H}(\Gr_{\mmu},\Gr_{\nu_\bullet}) & 2(m-n)=d_\mmu-d_{\nu_\bullet} \end{array}\right.\] 
If we write $(X,f)=\sqcup_i (X_i,m_i)$ and $(Y,g)=\sqcup (Y_j,n_j)$ as disjoint unions of connected objects, then we define
$$\Hom_{\Sat_G^0}((X,f),(Y,g))=\bigoplus_{ij}\Hom_{\Sat_G^0}((X_i,m_i),(Y_j,n_j)).$$
\end{itemize}
Lemma \ref{L:faithful realization} guarantees that this is a well-defined category, and the natural functor 
$$\Sat_G^0\to \on{Mot}_k^*: (X,f)\mapsto \oplus h(X_i)(m_i)$$  is a faithful (but not full) additive functor.
In the sequel, given an irreducible component $Z\subset \Gr_{\mmu}\times_\Gr\Gr_{\nu_\bullet}$ of dimension $\frac{1}{2}(d_{\mmu}+d_{\nu_\bullet})$, we will use $[Z]$ to denote the morphism in $\Sat^0_G$ induced by $Z$. Note that to make use of Lemma \ref{L:faithful realization}, we do need to include various ``Tate twists" of $(\Gr_\mmu,m)$ as objects in $\Sat^0_G$.

Next, we endow $\Sat_G^{0}$ with a monoidal structure (which a priori is different from the usual tensor product structure in $\on{Mot}_k^*$).
Note that the diagonal $\Delta: \Gr_{\mmu}\to \Gr_{\mmu}\times_{\Gr}\Gr_{\mmu}$ gives the identity morphism $1_{(\Gr_\mmu,m)}$ of $(\Gr_\mmu,m)$ in $\Sat_G^{0}$.
We define the tensor product on connected objects as
\[(\Gr_{\mmu},m)\star (\Gr_{\nu_\bullet},n):=(\Gr_{\mmu,\nu_\bullet},m+n),\]
and naturally extend to all objects by linearity.
To define the tensor product on morphisms, it is enough to define $1_{(\Gr_{\mmu},m)}\star [Z]$ and $[Z']\star 1_{(\Gr_{\nu_\bullet},n)}$, where $[Z]\in \Hom_{\Sat_G^{0}}((\Gr_{\nu_\bullet},n'), (\Gr_{\nu'_\bullet},n'))$ (resp. $[Z'] \in \Hom_{\Sat_G^{0}}((\Gr_\mmu,m), (\Gr_{\mu'_\bullet},m'))$) is represented by an irreducible component. We define
$$1_{(\Gr_{\mmu},m)}\star [Z]:=[\Delta\tilde\times Z],\quad [Z']\star 1_{(\Gr_{\nu_\bullet},n)}:=[Z'\tilde\times\Delta],$$
where we use the fact that $Z$ is $K$-invariant to form the twisted product defined via \eqref{E-conv Sch}.

\begin{lem}
\begin{enumerate}
\item The natural (not full) embedding $\Sat_G^0\subset\on{Mot}_k^*$ admits a canonical monoidal structure, i.e. there is a canonical isomorphism
\[h(\Gr_{\mmu,\nu_\bullet})(m+n)\cong h(\Gr_{\mmu})(n)\otimes h(\Gr_{\nu_\bullet})(n),\]
in $\on{Mot}^*_k$ compatible with the monoidal structures of both categories.
\item There is a unique commutativity constraint in $\Sat_G^{0}$ making the above monoidal functor symmetric. 
\item The functor $\Sat_G^0\to\on{Mot}^*_k$ extends to a symmetric monoidal functor 
$$\Sat_G^m\to \on{Mot}_k^*.$$ 
\end{enumerate}
\end{lem}
Since $\on{Mot}^*_k$ is the idempotent complete and by definition $\Sat_G^m$ is the idempotent completion of $\Sat_G^0$, the last statement follows from the other two.
By replacing (equivariant) cohomology in \cite[\S 2.3, \S 2.4]{Zh1} by (equivariant) chow rings (cf. \cite{EG}), one can construct the 
map $A^*(\Gr_{\mmu})\otimes_{\bQ}A^*(\Gr_{\nu_\bullet})\to A^*(\Gr_{\mmu,\nu_\bullet})$, and $A^*(\Gr_{\mmu,\nu_\bullet})\to A^*(\Gr_{\nu_\bullet,\mmu})$, where $A^*(-)$ denotes the rational Chow groups. Since all the varieties admit cellular decomposition, these maps give the candidates of the desired structures in (1) and (2).
By passing to the cohomology and using \cite[\S 2.3, \S 2.4]{Zh1}, one sees that these are indeed the required structures.
Instead of recalling all details, we give an example.

\begin{ex}
\label{R: comm via cycle}
Let $G=\GL_2$.
We consider $\Gr_{\omega_1,\omega_1}$, which classifies a chain of lattices $\{\La_2\subset\La_1\subset\La_0=\mO^2\}$ such that $\La_i/\La_{i+1}$ is $1$-dimensional.  It is isomorphic to the perfection of $\bP(\mO_{\bP^1}(-1)\oplus\mO_{\bP^1}(1))$, as mentioned in Example \ref{quasiminuSch} (see \cite[\S B.3]{Zh1} for a detailed discussion). Recall that it contains a $(-2)$-curve define by the condition $\La_2=\varpi\La_0$, denoted by $\Gr_{\omega_1,\omega_1}^{\omega_2}\subset \Gr_{\omega_1,\omega_1}$. (More precisely, in each case we consider the deperfection given by  $\bP(\mO_{\bP^1}(-1)\oplus\mO_{\bP^1}(1))$. Then the corresponding reduced closed subscheme of $\bP(\mO_{\bP^1}(-1)\oplus\mO_{\bP^1}(1))$ is a $(-2)$-curve.)

The isomorphism $h(\Gr_{\omega_1,\omega_1})\cong h(\Gr_{\omega_1})\otimes h(\Gr_{\omega_1})$ can be realized by a cycle 
$$Z_1+Z_2\subset \Gr_{\omega_1,\omega_1}\times \Gr_{\omega_1}\times \Gr_{\omega_1},$$ 
where 
$$Z_1=\{(z,\pr_1(z),\pr_1(z))\mid z\in \Gr_{\omega_1,\omega_1}\}, \quad Z_2=\{ (z, \pr_1(z), y)\mid z\in \Gr_{\omega_1,\omega_1}^{\omega_2}, y\in \Gr_{\omega_1}\}.$$ 
The automorphism of $\Gr_{\omega_1,\omega_1}$ in $\Sat_G^0$ 
\[c':\Gr_{\omega_1,\omega_1}= \Gr_{\omega_1}\star \Gr_{\omega_1}\stackrel{c}{\cong} \Gr_{\omega_1}\star \Gr_{\omega_1}=\Gr_{\omega_1,\omega_1},\]
induced by the commutativity constraint $c$, is given by 
\begin{equation}
\label{E:comm via cycle}
c'=[\Delta]+[\Gr_{\omega_1,\omega_1}^{\omega_2}\times \Gr_{\omega_1,\omega_1}^{\omega_2}],
\end{equation}
where the diagonal $\Delta$ and $\Gr_{\omega_1,\omega_1}^{\omega_2}\times \Gr_{\omega_1,\omega_1}^{\omega_2}$ are the irreducible components of $\Gr_{\omega_1,\omega_1}\times_{\Gr}\Gr_{\omega_1,\omega_1}$, giving a basis of
$\Hom_{\Sat_G^{0}}(\Gr_{\omega_1,\omega_1},\Gr_{\omega_1,\omega_1})$.
\end{ex}

Let us explain the relation between $\Sat_G^m$ and the usual Satake category. 
Let $\Sat^T_G\subset \on{P}_K(\Gr)$ denote the full subcategory spanned by $\IC_\mu(i),\quad \mu\in\xcoch(T)/W,i\in\bZ$, where $\IC_\mu$ is the intersection cohomology sheaf on $\Gr_\mu$ (whose restriction to the Schubert cell is $\Ql[d_\mu]$). By  \cite[Lemma 5.5.14]{Zh2} (whose argument works in mixed characteristic without change), one can bootstrap the geometric Satake to obtain the following commutative diagram
\[
\begin{CD}
\Sat^T_G@>\cong>>\on{Rep}_{\Ql}(\hat G^T)\\
@VVV@VVV\\
\Sat_{G\otimes\mO_L}@>\cong>>\on{Rep}_{\Ql}(\hat G).
\end{CD}
\]
Here 
\begin{itemize}
\item The group $\hat{G}^T$ is the semi-direct product of $\hat{G}$ with $\bG_m$ with the action of $\bG_m$ on $\hat G$ defined as follows (see \cite[(5.5.10)]{Zh2} for details): Let $\hat G_\ad$ denote the adjoint quotient of $\hat G$, which acts on $\hat G$ by conjugation. Then the composition $\bG_m\to \hat G\to\hat G_\ad$, where the first map is the cocharacter in Remark \ref{R: birth of the dual group} (2), admits a square root $\rho_\ad:\bG_m\to \hat G_\ad$ which induces an action of $\bG_m$ on $\hat G$. Note that $\hat G^T$ is equipped with a Borel subgroup $\hat B^T:=\hat B\rtimes\bG_m$ and a maximal torus $\hat T^T=\hat T\rtimes \bG_m=\hat T\times \bG_m$.
\item The bottom equivalence is as in Theorem \ref{intro:geomSat}.
\item The left vertical functors are the natural pullback functor along $\Gr\otimes\bar k \to \Gr$ and the right vertical functor is the restriction functor along $\hat G\to \hat G^T$.
\end{itemize}

By the semismallness of the convolution map \eqref{conv m2} and the decomposition theorem, there is a fully faithful monoidal functor 
$$\Sat_G^0\otimes\Ql\to \Sat_G^T,\quad (\Gr_{\mmu},m)\mapsto m_*\Ql_{\Gr_{\mmu}}[d_{\mmu}](m)\in \on{P}_K(\Gr),$$  which by \cite[Lemma 2.16]{Zh1} extends to an equivalence 
$$\Sat^m_G\otimes\Ql\cong \Sat^T_G.$$ 
Therefore, after a Koszul sign change of the commutativity constraint as in \cite[(2.4.4)]{Zh1}, we have the equivalence of symmetric monoidal categories
\begin{equation*}
\label{E:geom sat Tate}
\Sat^m_G\otimes\Ql\cong \on{Rep}(\hat{G}^T),
\end{equation*}
which we regard as the motivic geometric Satake.

\begin{rmk}
\label{R:Tate twist}

As already mentioned above, we do need to introduce the Tate twist to define $\Sat_G^0$ as a subcategory of $\on{Mot}^*_k$. Therefore the Tannakian group for $\Sat_G^m\otimes\Ql$ is not the original Langlands dual group $\hat{G}$, but the modified one $\hat G^T$. As being realized in recent years, the group $\hat G^T$ might be the more correct object to use in the formulation of the Langlands correspondence (e.g. see \cite{BG}). In certain cases, one can choose an isomorphism $\hat{G}^T\simeq \hat{G}\times \bG_m$. For example, if $G=\GL_n$, such an isomorphism can be given by 
\begin{equation}
\label{E: dual group GLn}
\hat{G}^T:=\GL_n\rtimes\bG_m\to \GL_n\times \bG_m,\ (A,t)\mapsto (A\on{diag}\{t^{n-1},t^{n-2},\ldots,1\}, t).
\end{equation}
We refer to \cite[\S 5.5]{Zh2} for more detailed discussions of different versions of Langlands dual groups. 

The equivalence induces an isomorphism of $\bZ[v,v^{-1}]$-algebras 
$$K(\Sat^m_G)\cong K(\on{Rep}(\hat{G}^T)),$$ 
where the inclusion of $\bZ[v,v^{-1}]$ to the left hand side is induced by the inclusion of the full subcategory $\{(\Gr_0,m), m\in\bZ\}$ into $\Sat_G^m$ and the inclusion of $\bZ[v,v^{-1}]$ to the right hand side is induced by the inclusion of the full subcategory $\on{Rep}(\bG_m)\subset \on{Rep}(\hat G^T)$. Then specializing to $v=q$ gives a \emph{canonical} Satake isomorphism, which unlike \eqref{Sat isom}, is independent of any choice. In the case of $\GL_2$, this isomorphism together with \eqref{E: dual group GLn} induces the normalized Satake isomorphism $\Sat'^{cl}$ in Remark \ref{R: norm Sat}.
\end{rmk}

\begin{ex}\label{intersection by geom Sat}
Here is the toy model of our following applications of the geometric Satake. 

We consider $G=\GL_2$. Let $\Gr_{\omega_1,\omega_1^*}$ (resp. $\Gr_{\omega_1^*,\omega_1}$) be the moduli classifying a chain of lattices
\begin{equation}\label{Chain 2}
\{\La_2\supset\La_1\subset \La_0=\mO^2\}, \quad \quad \mbox{resp. } \quad \{\La'_2\subset \La'_1\supset \La_0=\mO_2\}
\end{equation}
such that both $\La_0/\La_{1}$ and $\La_2/\La_1$ (resp. $\La'_1/\La_0$ and $\La'_1/\La'_2$) are one-dimensional over $k$.
Note that there are the canonical isomorphisms 
\begin{equation}
\label{E: isomisom1}
\Gr_{\omega_1,\omega_1^*}\cong\Gr_{\omega_1,\omega_1}, \quad \La_2\supset\La_1\subset \La_0\mapsto \varpi\La_2\subset\La_1\subset\La_0,
\end{equation}
and
\begin{equation}
\label{E: isomisom2}
\Gr_{\omega_1^*,\omega_1}\cong\Gr_{\omega_1,\omega_1}, \quad \La'_2\subset\La'_1\supset \La_0\mapsto \varpi\La'_2\subset\varpi\La'_1\subset\La_0,
\end{equation}
so both $\Gr_{\omega_1,\omega_1^*}$ and $\Gr_{\omega_1^*,\omega_1}$ are isomorphic to the perfection of $\bP(\mO_{\bP^1}(-1)\oplus\mO_{\bP^1}(1))$. Note that 
$$\Gr_{\omega_1,\omega_1*}^0:=\Gr_{\omega_1,\omega_1^*}\times_{\Gr}\Gr_0,\quad\quad \mbox{resp. }\quad\Gr_{\omega_1^*,\omega_2}^0:=\Gr_{\omega_1^*,\omega_1}\times_{\Gr}\Gr_0$$
is the closed subscheme of $\Gr_{\omega_1,\omega_1^*}$ (resp. $\Gr_{\omega_1^*,\omega_1}$) classifying those chains in \eqref{Chain 2} with $\La_2=\La_0$ (resp. $\La'_2=\La_0$) and therefore is isomorphic to the perfection of $\bP^1$, which as mentioned in Example \ref{quasiminuSch} is a $(-2)$-curve on the surface. 

We consider the composition of the following map in $\Sat_G^{0}$ given by
\[\Gr_0\xrightarrow{\Gr_{\omega_1,\omega_1^*}^0} \Gr_{\omega_1,\omega_1^*}=\Gr_{\omega_1}\star\Gr_{\omega_1^*}\stackrel{c}{\cong}\Gr_{\omega_1^*}\star\Gr_{\omega_1}=\Gr_{\omega_1^*,\omega_1}\xrightarrow{\Gr_{\omega_1^*,\omega_1}^0}\Gr_0,\]
where we regard $\Gr_{\omega_1,\omega_1^*}^0\in \Hom(\Gr_0, \Gr_{\omega_1,\omega_1^*})$ and $\Gr_{\omega_1^*,\omega_1}^0\in \Hom(\Gr_{\omega_1^*,\omega_1},\Gr_0)$ and $c$ is the commutativity constraint. One can see that this is given by multiplication by $2$, by combining the isomorphism \eqref{E: isomisom1} and \eqref{E: isomisom2}, the commutativity constraint \eqref{E:comm via cycle}, and the fact that $\Gr_{\omega_1,\omega_1}^{\omega_2}$ is a $(-2)$-curve in $\Gr_{\omega_1,\omega_1}$. 

On the other hand, according to the geometric Satake, this map can be computed by the following map 
\[\mathbf{1}\to \on{Std}\otimes \on{Std}^*\cong \on{Std}^*\otimes \on{Std} \to \mathbf{1}\]
in the category of finite dimensional representations of $\GL_2\times \bG_m$, which is also given by multiplying $\dim \on{Std}=2$. To summarize, the intersection numbers between certain cycles in the (convolution) affine Grassmannian can be calculated via some representation theory of the Langlands dual group. 
\end{ex}

Finally, let us briefly mention how to generalize the above construction from $\GL_n$ to a general reductive group $G$ over $\mO$. For a general reductive group, there are not enough minuscule coweights. For example, if $G=E_8$, there is no non-zero minuscule coweight.
So in general, we need to define $\Sat^0_G$ as the category with objects being disjoint unions of $(\Gr_{\mmu},m)$ where $\mmu=(\mu_1,\ldots,\mu_r)$ with $\mu_i$ minuscule or quasi-minuscule (see, for example, \cite[\S 2.2.2]{Zh1} for the notion of quasi-minuscule coweights), and with morphisms between $(\Gr_{\mmu},m)$ and $(\Gr_{\nnu},n)$ given by the same formula as before. 

We need to show that this is a well-defined category (i.e. morphisms can be composed) which admits a faithful embedding into $\on{Mot}_k^*$.
If $\mu$ is quasi-minuscule, $\Gr_\mu$ is not perfectly smooth, but is the perfection a projective cone over some partial flag variety $G/P$, which admits an explicit resolution of singularities by blowing up the singular point of the cone (e.g. see \cite[Lemma 2.12]{Zh1}). Let $\pi:\wGr_\mu\to \Gr_\mu$ denote this ``resolution", which is $K$-equivariant from the construction. It is not difficult to write down the idempotent $p$, which can be represented by a $K$-invariant cycle in $\wGr_\mu\times\wGr_\mu$, such that the \'etale realization of $(\wGr,p,0)$ gives the intersection cohomology of $\Gr_\mu$ (which is a direct summand of $\on{H}^*(\wGr_\mu)$ by the decomposition theorem). 
 Then there is a lemma similar to Lemma \ref{L:faithful realization}, with obvious modifications. Repeating the previous constructions then gives the motivic Satake category $\Sat^m_G$.

\section{The categorical trace construction}
In recent years, the notion of categorical center/trace has played important roles in representation theory. 
We refer to \cite{BN,BFO,Lu3,Lu4,Lu5,Lu6} for the applications to the theory of character sheaves and to \cite{BKV} for the applications to the stable Bernstein center. In the remaining part of this article, we explain a different perspective and application of these categorical constructions, based on a joint work with Liang Xiao \cite{XZ}.

\subsection{The categorical trace}
\label{S: cat trace}
To motivate the definition, let us first briefly recall the notion of the trace (or sometimes called the cocenter) of an algebra.
Let $E$ be a base commutative ring and $A$ an $E$-algebra. A trace function from $A$ to an $E$-module $V$ is an $E$-linear map $f: A\to V$ such that $f(ab)=f(ba)$.
The trace of $A$ is an $E$-module $\Tr(A)$ equipped with a trace function $\tr: A\to \Tr(A)$, such that every trace function $f:A\to V$ uniquely factors as $A\xrightarrow{\tr} \Tr(A)\xrightarrow{\bar{f}}V$ for some linear map $\bar{f}:\Tr(A)\to V$. Clearly, if we define $\Tr(A)$ as the quotient of $A$ by the $E$-submodule spanned by elements of the form $ab-ba,\ a,b\in A$ and let $\tr$ be the quotient map, then $(\Tr(A),\tr)$ is the trace of $A$. Alternatively, we can express $\Tr(A)$ as the colimit of the following diagram in the category of $E$-modules
\begin{equation}
\label{E:HH}
\Tr(A)=\on{colim}(A\otimes A\substack{\longrightarrow\\[-1em] \longrightarrow}  A)
\end{equation}
where the two maps $A\otimes A\to A$ are given $a\otimes b\mapsto ab$ and $a\otimes b\mapsto ba$. Note that this diagram is just the first two terms of the bar complex of $A$ that computes the Hochschild homology of $A$.

Now we move to the categorical setting and introduce the notion of categorical trace of a monoidal category. 
The basic idea is as follows: monoidal categories are algebra objects in the symmetric monoidal $2$-category of certain categories; in general one should define the trace of an algebra object in a symmetric monoidal (higher) category as the colimit of its bar complex \eqref{E:HH}. Thanks to the robust higher category theory, this idea works in a quite general setting.  For example, we refer to \cite{BN} for excellent discussions.

For our later applications, it is enough to introduce this notion at the level of ordinary $E$-linear categories and it is important to construct the categorical trace explicitly. On the other hand, we need to the notion of the twisted categorical trace of
a monoidal endofunctor. So let us spread out the definition in detail. We will work in the $2$-category of essentially small $E$-linear categories.  
Let $(\mC,\otimes,\mathbf{1})$ be an $E$-linear monoidal category and  $\sigma: \mC\to \mC$ a monoidal endofunctor. For every $n$, let
$\mC^{\otimes n}$ denote the $E$-linear category with objects $(X_i)_{i=1,\ldots,n}$, and morphisms $\Hom_{\mC^{\otimes n}}((X_i)_i, (Y_i)_i)=\otimes_i \Hom_\mC(X_i,Y_i)$. Then the left $\sigma$-twisted categorical trace $\Tr_\sigma(\mC)$ is defined to be the $2$-colimit of the diagram
\[
\mC^{\otimes 3}\ \substack{\longrightarrow\\[-1em] \longrightarrow\\[-1em] \longrightarrow} \ \mC^{\otimes 2}\ \substack{\longrightarrow\\[-1em] \longrightarrow} \ \mC,
\]
where the functors $\mC^{\otimes 2}\to \mC$ are given by $(X,Y)\mapsto X\otimes Y$ and $Y\otimes \sigma X$ respectively, and $\mC^{\otimes 3}\to \mC^{\otimes 2}$ are given by $(X,Y,Z)\mapsto (X,Y\otimes Z), (Y, Z\otimes \sigma X)$ and $(X\otimes Y,Z)$ respectively.

More concretely, let $\mD$ be a (plain) $E$-linear category.
A left $\sigma$-twisted \emph{trace functor} from $\mC$ to $\mD$ is a functor $F:\mC\to\mD$, together with a family of canonical isomorphisms (functorial in each argument)
\[\al_{X,Y}: F(X\otimes Y)\cong F(Y\otimes \sigma X)\]
such that the following diagram is commutative
\begin{equation}
\label{E:cond trace}
\xymatrix{
F((X\otimes Y)\otimes Z)\ar^{\al_{X\otimes Y,Z}}[r]\ar_\cong[d] & F(Z\otimes \sigma (X\otimes Y))\ar^{\cong}[r]&  F((Z\otimes \sigma X)\otimes \sigma Y)  \\
F((X\otimes (Y\otimes Z))\ar^{\al_{X,Y\otimes Z}}[r] & F((Y\otimes Z)\otimes \sigma X)\ar^{\cong}[r]&  F(Y\otimes (Z\otimes \sigma X)). \ar_{\al_{Y,(Z\otimes\sigma X)}}[u]
}
\end{equation}
Note that it follows that the following two diagrams are commutative
\begin{equation}
\label{E:cond trace cor}
\xymatrix{
F(\mathbf{1}\otimes X)\ar_{\cong}[d]\ar^{\al_{\mathbf{1},X}}[r]& F(X\otimes \sigma\mathbf{1})\ar^\cong[d]&  F(X\otimes Y)\ar^{\al_{Y,\sigma X}\al_{X,Y}}[r]\ar_\cong[d]& F(\sigma X\otimes \sigma Y)\ar^\cong[d]\\
 F(X)\ar^-\cong[r]& F(X\otimes \mathbf{1})&  F((X\otimes Y)\otimes\mathbf{1})\ar^{\al_{X\otimes Y,\mathbf{1}}} [r]& F(\mathbf{1}\otimes\sigma(X\otimes Y)).
}\end{equation}
Let $\on{Func}^{\tr_\sigma}(\mC,\mD)$ denote the category of left $\sigma$-twisted trace functors.  Then $\Tr_\sigma(\mC)$ if exists, is the unique  (up to a unique equivalence) category such that for any (plain) $E$-linear category $\mD$,
\[\on{Func}^{\tr_\sigma}(\mC,\mD)\cong \on{Func}(\Tr_\sigma(\mC),\mD),\]
where $\on{Func}(-,-)$ denote the category of $E$-linear functors between $E$-linear categories.

For our application, we need an explicit description of hom spaces of $\Tr_\sigma(\mC)$. Usually, describing the hom spaces of a colimit diagram of categories is difficult. However, in our case this is possible under a mild assumption on $\mC$.
\begin{prop}
\label{L:trace cat}
Assume that $\mC$ is a small $E$-linear category, and that every object in $\mC$ admits a left dual object. Then $(\Tr_\sigma(\mC),\tr)$ exists. If $\mC$ is additive, so is $\Tr_\sigma(\mC)$.
\end{prop}
Recall that in a monoidal category, the left dual of an object $V$ is an object $V^*$ equipped with morphisms 
$$\coev_V:\mathbf{1}\to V\otimes V^*,\quad \ev_V: V^*\otimes V\to \mathbf{1}$$ 
such that both of the following compositions are the identity map
\[V\cong \mathbf{1}\otimes V\xrightarrow{\coev_V\otimes\id_{V}}(V\otimes V^*)\otimes V\cong V\otimes (V^*\otimes V)\xrightarrow{\id_{V}\otimes\ev_V}V\otimes\mathbf{1}\cong V\]
\[V^*\cong V^*\otimes\mathbf{1}\xrightarrow{\id_{V^*}\otimes\coev_V} V^*\otimes (V\otimes V^*)\cong (V^*\otimes V)\otimes V^*\xrightarrow{\ev_V\otimes\id_{V^*}}\mathbf{1}\otimes V^*\cong V^*.\]
The triple $(V^*,\coev_V,\ev_V)$ is unique up to a unique isomorphism.

\begin{proof}
We give a detailed proof since it will serve as a motivation for the construction of $S$-operators in the next subsection. To simplify notations, we suppress the associativity constraint in all formulas and diagrams appearing below.

We first define the category $\Tr'_\sigma(\mC)$  whose objects are the same as those of $\mC$. To avoid confusion, an object $X\in\mC$ will be denoted by $\widetilde{X}$ when it is regarded as an object of $\Tr'_\sigma(\mC)$. We define the space of morphisms as 
\begin{equation*}
\label{E:trace hom}
\Hom_{\Tr'_\sigma(\mC)}(\widetilde{X},\widetilde{Y}):=\big(\bigoplus_{(V,W)\in \mC\times \mC}\Hom(X, V\otimes W)\otimes_E \Hom(W\otimes \sigma V, Y)\big)/\sim.
\end{equation*}
where the equivalence relation is generated as follows: if we denote by $S_{u,v}$ the element in $\Hom_{\Tr'_\sigma(\mC)}(\widetilde{X},\widetilde{Y})$ corresponding to $u: X\to V\otimes W$ and $v:W\otimes \sigma V\to Y$, then $S_{u,v}=S_{u',v'}$ if there is the following commutative diagram
\[\xymatrix{
X\ar^-{u}[r] \ar_{u'}[dr]& V\otimes W\ar[d] & W\otimes \sigma V\ar[d]\ar^-{v}[r] & Y\\
& V'\otimes W' & W'\otimes \sigma V'\ar_{v'}[ur] &
},\]
where the vertical maps are induced by an element in $\Hom_\mC(V,V')\otimes_E\Hom_{\mC}(W,W')$. In particular, it follows from the two commutative diagrams 
\begin{equation}
\label{E:frac}
\xymatrix{
&V\otimes V^*\otimes X\ar^{\id_V\otimes u^*}[d]& V^*\otimes X\otimes \sigma V\ar^{v_u}[dr]\ar_{u^*\otimes \id_{\sigma V}}[d]&\\
X\ar^{\coev_V\otimes\id_X}[ur]\ar^{u}[r]\ar_{u_v}[dr]&V\otimes W\ar^{\id_V\otimes v^*}[d]&W\otimes \sigma V\ar^{v}[r]\ar_{v^*\otimes\id_{\sigma V}}[d]& Y\\
&V\otimes Y\otimes \sigma V^*& Y\otimes \sigma V^*\otimes \sigma V\ar_{\id_Y\otimes\ev_{\sigma V}}[ur]&
}
\end{equation}
that $S_{\coev_V\otimes\id_X, v_u}=S_{u,v}=S_{u_v,\id_Y\otimes\ev_{\sigma V}}$.  
Here, $V^*$ is the left dual of $V$, $u^*: V^*\otimes X\to W$ dual to $u$ and $v^*:W\to Y\otimes \sigma V^*$ dual to $v$, and
$u_v= (\id_V\otimes v^*)\circ u$ and $v_u= v\circ(u^*\otimes \id_{\sigma V})$.

We also need to explain the composition of morphisms in $\Tr'_\sigma(\mC)$. Given $u_1: X\to V_1\otimes W_1$ and $v_1: W_1\otimes \sigma V_1\to Y$, and given $u_2: Y\to V_2\otimes W_2$ and $v_2: W_2\otimes \sigma V_2\to Z$, let $S_{u_1,v_1}\in \Hom_{\Tr'_\sigma(\mC)}(\widetilde{X},\widetilde{Y})$ and $S_{u_2,v_2}\in \Hom_{\Tr'_\sigma(\mC)}(\widetilde{Y},\widetilde{Z})$ be the corresponding elements. We define $S_{u_2,v_2}\circ S_{u_1,v_1}$ as $S_{u,v}$, where 
\[u: X\xrightarrow{u_1} V_1\otimes W_1\xrightarrow{\id_{V_1}\otimes\coev_{V_2}\otimes\id_{W_1}}V_1\otimes V_2\otimes V^*_2\otimes W_1,\] 
\[v: V^*_2\otimes W_1\otimes \sigma V_1\otimes \sigma V_2\xrightarrow{\id_{V^*_2}\otimes v_1\otimes\id_{\sigma V_2}}V^*_2\otimes Y\otimes \sigma V_2\xrightarrow{(v_2)_{u_2}} Z.\]

One checks immediately that the composition is independent of the choice of $(u_1,v_1)$.  On the other hand, using \eqref{E:frac}, it is easy to see that $S_{u,v}=S_{u',v'}$, where
\[
u': X\xrightarrow{(u_1)_{v_1}} V_1\otimes Y\otimes \sigma V^*_1\xrightarrow{\id_{V_1}\otimes u_2\otimes \id_{\sigma V^*_1}} V_1\otimes V_2\otimes W_2\otimes \sigma V^*_1, 
\]
\[
v': W_2\otimes \sigma V^*_1 \otimes \sigma V_1\otimes \sigma V_2 \xrightarrow{\id_{W_2}\otimes\ev_{\sigma V_1}\otimes\id_{\sigma V_2}}W_2\otimes \sigma V_2\xrightarrow{v_2} Z.
\]

Using this expression, one checks that the composition is also independent of the choice of $(u_2,v_2)$. 
In addition, one checks immediately  the identity map of $\widetilde{X}$ can be represented by the canonical isomorphisms $X\cong  \mathbf{1}\otimes X$ and $X\otimes \sigma \mathbf{1}\cong X\otimes\mathbf{1}\cong X$.

We have defined $\Tr'_\sigma(\mC)$ as a category.  Given $X$ and $Y$, there is a canonical morphism
\[\al_{X,Y}:\widetilde{X\otimes Y}\to \widetilde{Y\otimes \sigma X}\]
given by $S_{u,v}$, where both $u$ and $v$ are the identity map. 
In particular, for every $X$, there is a canonical morphism
$$\ga_{X}:\widetilde{X}\cong \widetilde{X\otimes\mathbf{1}}\xrightarrow{\al_{X,\mathbf{1}}} \widetilde{\mathbf{1}\otimes \sigma X}\cong \widetilde{\sigma X}.$$
One checks from the construction that the collection $\{\al_{X,Y}\}$ satisfy the commutative diagram \eqref{E:cond trace}, and therefore \eqref{E:cond trace cor} also holds.
We define $\Tr_\sigma(\mC)$ as the localization of $\Tr'_\sigma(\mC)$ with respect to the multiplicative system $\{\ga^n_X,\ X\in \mC, n\in \bZ_{\geq 0}\}$ (so that $\ga_X$ becomes an isomorphism in $\Tr_\sigma(\mC)$). It follows that all $\al_{X,Y}$ become isomorphisms the $\Tr_\sigma(\mC)$.

Clearly, there is a $\sigma$-twisted trace functor $\tr_\sigma: \mC\to \Tr_\sigma (\mC)$ sending $X$ to $\widetilde{X}$ and $f:X\to Y$ to $S_{u,v}$ where $u: X\cong \mathbf{1}\otimes X$ and $v: X\otimes\sigma\mathbf{1}\cong X\xrightarrow{f} Y$. In addition, if $F: \mC\to \mD$ is a $\sigma$-twisted trace functor, then it induces a functor $\Tr_\sigma(\mC)\to \mD$ sending $\widetilde X$ to $F(X)$ and $S_{u,v}$ to $F(v)\circ \alpha_{V,W}\circ F(u)$. 

Finally, notice that from the construction, there is a canonical isomorphism $\widetilde{X}\oplus\widetilde{Y}\cong\widetilde{X\oplus Y}$ in $\Tr'_\sigma(\mC)$. If follows that if $\mC$ is additive, so is $\Tr'_\sigma(\mC)$ and $\Tr_\sigma(\mC)$. The proposition is proven.
\end{proof}

\begin{rmk}
The idea of defining the composition $S_{u_2,v_2}\circ S_{u_1,v_1}$ can be explained by the following diagram
\begin{tiny}
\begin{equation*}
\label{E: comp in tr}
\xymatrix{
                     &                                                                                          &V_1\otimes V_2\otimes V^*_2\otimes W_1\ar@{-->}[r]&V_2\otimes V_2^*\otimes W_1\otimes \sigma V_1\ar[dr]\ar@{-->}[r]& V_2^*\otimes W_1\otimes \sigma V_1\otimes \sigma V_2\ar^-{\id_{V^*_2}\otimes v_1\otimes\id_{\sigma V_2}}[dr] &                                                    &\\
 X\ar^-{u_1}[r]&V_1\otimes W_1\ar^-{\id\otimes\coev_{V_2}\otimes\id}[ur]\ar@{-->}[r]& W_1\otimes \sigma V_1\ar[ur]\ar[r]                      &                 Y\ar[r]                                                   & V_2\otimes V_2^*\otimes Y    \ar@{-->}[r]                              & V_2^*\otimes Y\otimes \sigma V_2 \ar^-{(v_2)_{u_2}}[r]& Z. \\                                            
}
\end{equation*}
\end{tiny}

Namely, we can first replace $S_{u_2,v_2}$ by $S_{\coev_{V_2}\otimes \id_Y,(v_2)_{u_2}}$. Then $S_{\coev_{V_2}\otimes \id_Y,(v_2)_{u_2}}\circ S_{u_1,v_1}$  has to be the one given in the proof since the triangle in the middle is commutative in $\mC$ and the dotted arrows should be isomorphisms in $\Tr_\sigma(\mC)$.
\end{rmk}

\begin{rmk}
Of course, there is a dual notion of a right trace functor, i.e. a functor $F:\mC\to \mD$ equipped with a family of functorial isomorphisms $\beta_{X,Y}: F(X\otimes Y)\cong F(\sigma Y\otimes X)$ satisfying a condition analogous to \eqref{E:cond trace}, and therefore the notion of the right $\sigma$-twisted trace category. The above construction and the following discussions have counterparts for these dual notions. 

To simply the exposition, in the sequel, we will call left trace functors by trace functors and the left $\sigma$-twisted trace category of $\mC$ simply by the $\sigma$-twisted trace category of $\mC$.
\end{rmk}

Here is a concrete example of the above construction. 

\begin{ex}
\label{E: trace of rep}
Let $H$ be a split reductive algebraic group over a field $E$ of characteristic zero, equipped with an automorphism $\sigma: H\to H$. Consider the $\sigma$-twisted conjugation of $H$ on itself given by the formula \eqref{E:twist conj},
which is equivalent to the usual conjugation of $H$ on the coset $H\sigma\subset H\rtimes\langle\sigma\rangle$.
Let $\frac{H}{c_\sigma H}$ denote the corresponding quotient stack. 
We regard the category $\on{Rep}_E(H)$ of finite dimensional representations of $H$ (over $E$) as the category of coherent sheaves on the classifying stack $\mathbf BH$. 
Then the pullback functor $\pi^*: \on{Rep}_E(H)\to \Coh(\frac{H}{c_\sigma H})$ has a $\sigma$-twisted trace functor structure, 
which induces
\[\Tr_\sigma(\on{Rep}_E(H))\cong \on{Coh}_{fr}^H(H\sigma)\subset \on{Coh}(\frac{H}{c_\sigma H}),\]
where $\Coh_{fr}^H(H\sigma)$ denotes the full subcategory of $\on{Coh}(\frac{H}{c_\sigma H})$ spanned by those $\widetilde V(:=\pi^*V)$. This can be seen as follows. 
Via descent, we may regard $\widetilde V$ as the trivial vector bundle $H\times V$ on $H$, with the diagonal $H$-equivariant structure. In addition,
for a representation $H\to \GL(V)$, let $\sigma V$ denote the $\sigma$-twist of $V$, i.e. the representation of $H$ defined by $H\xrightarrow{\sigma^{-1}}H\to\GL(V)$. Let us formally denote the identity map $V\to \sigma V$ by $\sigma$, so that $h\cdot \sigma=\sigma\cdot \sigma^{-1}(h):V\to\sigma V$.
Then for a morphism $S_{u,v}$ in $\Tr'_\sigma(\on{Rep}_E(H))$, given by $u:X\to V\otimes W$ and $v:W\otimes \sigma V\to Y$, we define a morphism $\widetilde{X}\to\widetilde{Y}$ in $\Coh_{fr}^H(H\sigma)$ whose restriction to the fiber over $h\in H$ is given by
\[
\widetilde{X}|_h=X\xrightarrow{u} V\otimes W\xrightarrow{h\sigma\otimes\id}\sigma V\otimes W\cong W\otimes \sigma V\xrightarrow{v} Y=\widetilde{Y}|_h.
\]
This gives a natural functor $\Tr_\sigma(\on{Rep}_E(H))\to \on{Coh}_{fr}^H(H\sigma)$. By definition, the functor is essentially surjective and the Peter-Weyl theorem implies that it is also fully faithful, and therefore is an equivalence.

Assume that $\sigma=\id$. Note that since $\on{Rep}_E(H)$ is symmetric monoidal, the identity functor is a trace functor, and therefore induces a functor $\Tr_\sigma(\on{Rep}_E(H))\cong \on{Coh}_{fr}^H(H\sigma)\to \on{Rep}_E(H)$. One checks that this is nothing but the restriction of sheaves on $\frac{H}{cH}$ the unit of $H$.
\end{ex}

We continue our general discussion of the categorical trace. Assume that $\mC$ is as in Proposition \ref{L:trace cat}.
Let us assume momentarily that $\sigma$ is the identity functor. In this case, we denote $\Tr_\sigma(\mC)$ by $\Tr(\mC)$. For every endomorphism $f: V\to V$ in $\mC$, we denote $\coev_f:\mathbf{1}\to V\otimes V^*$ to be the morphism given by $\mathbf{1}\xrightarrow{\coev_V}V\otimes V^*\xrightarrow{f\otimes\id_{V^*}}V\otimes V^*$. Then
there is an element
\[S_{\coev_f,\ev_V}\in \End_{\Tr(\mC)}\widetilde{\mathbf{1}}.\]
In particular, if $f=\id$, we denote it by
\begin{equation}
\label{E:abs S}
S_V:=S_{\coev_{V},\ev_{V}}\in \End_{\Tr(\mC)}\widetilde{\mathbf{1}}.
\end{equation}
This element depends only on the isomorphism classes of $V$. It follows from the construction that
\[S_{\mathbf{1}}=\id,\quad S_{V'}\cdot S_{V}=S_{V\otimes V'},\quad S_{V\oplus V'}=S_V+S_{V'}.\]
Thus if $\mC$ is additive, there is a canonical homomorphism 
\begin{equation}
\label{E: K to End}
K^{\oplus}(\mC)^{\on{op}}\to \End_{\Tr(\mC)}\widetilde{\mathbf{1}},
\end{equation}
where $K^{\oplus}(\mC)$ denotes the split Grothendieck ring of $\mC$: as an abelian group, it is generated by objects in $\mC$ modulo the relations $[X]+[Y]=[X\oplus Y]$, with the ring structure given by the tensor product.

\begin{rmk}
Note that if $\mC$ is symmetric monoidal, then the identity functor of $\mC$ has a natural trace functor structure provided by the commutativity constraint, and therefore factors as $\mC\to \Tr(\mC)\to \mC$. It is clear that the induced map $\End_{\Tr(\mC)}(\widetilde{\mathbf{1}})\to \End_\mC(\mathbf{1})$ sends $S_{\coev_f,\ev_V}$ to the usual trace $\tr_V(f)$ of $f$ and $S_V$ to $\on{rank}(V)$, the rank of $V$  (see, for example, \cite[(1.7.3)]{DM} for the definitions of these notions).
\end{rmk}

The above construction admits the following generalization. Let $X\in \mC$, $V\in \mC$, and let $c_{X,V}: X\otimes V\simeq V\otimes X$ be a given isomorphism. Then we have an element $S_{c_{X,V}}\in \End_{\Tr(\mC)}\widetilde{X}$ given by
\[X\xrightarrow{\coev_V\otimes\id_X}V\otimes V^*\otimes X,   \quad V^*\otimes X\otimes V\xrightarrow{\id_{V^*}\otimes c_{X,V}} V^*\otimes V\otimes X\xrightarrow{\ev_V\otimes \id_X} X.\]
If we define the centralizer category of $X$ in $\mC$ as the category $\mZ_{\mC}(X)$ whose objects are pairs
$\{V\in\mC, c_{X,V}: X\otimes V\simeq V\otimes X\}$, and whose morphisms are those morphisms $V\to V'$ in $\mC$ compatible with $c_{X,V}$ and $c_{X,V'}$, then $\mZ_{\mC}(X)$ is a monoidal category, and if $\mC$ is in addition additive, so is $\mZ_{\mC}(X)$. It follows that we have a morphism
\begin{equation}
\label{E: K to End2}
K^{\oplus}(\mZ_{\mC}(X))^{\on{op}}\to \End_{\Tr(\mC)}(\widetilde X)
\end{equation}
In particular, if $\mC$ is an additive braided monoidal category, there is a natural monoidal functor $\mC\to \mZ_\mC(X)$ for every $X$. Therefore, in this case we have 
$$K^{\oplus}(\mC)^{\on{op}}\to \End_{\Tr(\mC)}\widetilde{X}.$$

Now for general $\sigma$, let $\mC^\sigma$ denote the category of $\sigma$-equivariant objects in $\mC$. I.e. objects are pairs $(X,\phi)$ consisting of an object $X\in\mC$ and an isomorphism $\phi: \sigma X\simeq X$, and morphisms between $(X,\phi)$ and $(X',\phi')$ are those morphisms in $\mC$ compatible with $\phi$ and $\phi'$. Note that $\mC^\sigma$ is a natural monoidal category, and if every object in $\mC$ admits a left dual, so is every object in $\mC^\sigma$. 
We can apply the above discussions to $\mC^\sigma$.
The forgetful functor $\mC^\sigma\to\mC$ induces a functor
\[\Tr(\mC^\sigma)\to \Tr_\sigma(\mC).\]
In particular, we have
\begin{equation}
\label{E: K to End3}
K^{\oplus}(\mC^\sigma)^{\on{op}}\to \End_{\Tr_\sigma(\mC)}(\widetilde{\mathbf 1}),
\end{equation}
Concretely, for $(X,\phi: \sigma X\simeq X)\in\mC^\sigma$, the corresponding endomorphism of $\widetilde{\mathbf 1}$ in $\Tr_\sigma(\mC)$ is given by $S_{\coev_X, \ev_X(1\otimes\phi)}$.

Now, we assume that $E$ is an algebraically closed field, and that $\mC$ is semisimple abelian in which the unit object $\mathbf{1}$ is irreducible. 
Then, $\mC^\sigma$ is an abelian monoidal category. 

\begin{prop}
\label{L: K iso End}
Under the above assumption, the map \eqref{E: K to End3} induces an isomorphism
\[
(K(\mC^\sigma)^{\on{op}}\otimes E)/I\cong \End_{\Tr_\sigma(\mC)}(\widetilde{\mathbf 1}),
\]
where $I$ is the ideal of $K(\mC^\sigma)^{\on{op}}\otimes E$ generated by $[(X,a\phi)]-a[(X,\phi)]$, with $(X,\phi)$ irreducible in $\mC^\sigma$ and $a\in E^\times$.
\end{prop}
There is a similar result by Ostrik \cite[Corollary 2.16]{O}, in light of Remark \ref{R: cat center} below.
\begin{proof}
Since the result is not used in the sequel, we only sketch the proof. First, since $E$ is algebraically closed and $\mathbf{1}$ is irreducible, one shows that if $S_{u,v}\neq 0$ for some $u:\mathbf{1}\to V\otimes W$ and $v: W\otimes \sigma V\to \mathbf{1}$, then $S_{u,v}$ is a linear combination of $S_{\coev_{Y}, \ev_Y(1\otimes\psi)}$ with $(Y,\psi)$ indecomposable. 
Next, one shows that indecomposable objects in $\mC^\sigma$ are of the form $(Y,\psi)=(X\otimes W, \phi\otimes f)$, where $(X,\phi)$ is irreducible in $\mC^\sigma$, $W$ is a finite dimensional $E$-vector space and $f$ is an automorphism of $W$. Let $f^{\on{ss}}$ denote the semisimplification of $f$, and $(Y,\psi^{\on{ss}})=(X\otimes W,\phi\otimes f^{\on{ss}})$ the semisimplification of $(Y,\psi)$. 
Then one shows that $S_{\coev_{Y}, \ev_Y(1\otimes\psi)}=S_{\coev_{Y},\ev_Y(1\otimes\psi^{\on{ss}})}$. It follows that $\eqref{E: K to End3}$ induces the desired map which in addition is surjective. The injectivity is clear. 
\end{proof}

\begin{rmk}Recall that the traditional categorification/decategorification is a passage between monoidal categories and algebras via the K-ring construction. The above proposition suggests another passage via the categorical trace construction.
\end{rmk}

\begin{rmk}
\label{R: cat center}
There is also a notion of categorical center (or called Drinfeld center) of a monoidal category (as appearing in \cite{BFO, BN, Lu3}). Given $(\mC,\otimes,\mathbf{1})$, its categorical center $\mZ(\mC)$ is the category whose objects consist of $(X\in \mC, \{\al_{X,Y}: X\otimes Y\simeq Y\otimes X\}_{Y\in \mC})$, satisfying the natural compatibility conditions. In certain cases, there is an equivalence $\mZ(\mC)\to \Tr(\mC)$. However, in our applications, it is more natural to use the notion of categorical trace.
\end{rmk}
\quash{
\begin{rmk}
\label{R: cat tr infty cat}
We briefly discuss the notion of categorical trace in the $\infty$-categorical setting. Note that in higher category theory, we cannot construct $\Tr(\mC)$ in an elementary way as in Proposition \ref{L:trace cat} by specifying objects and morphisms. On the other hand, since
monoidal categories are categorical analogue of algebras, one should define the trace of $\mC$ as the categorical analogue of the Hochschild homology of $\mC$ as in \eqref{E:HH}. It turns out in symmetric monoidal $\infty$-categories, \eqref{E:HH} makes sense for algebra objects, and a (stable) presentable monoidal $\infty$-category can be regarded as an algebra object in the symmetric monoidal $\infty$-category of all (stable) presentable $\infty$-categories (via Lurie's tensor product). In this way, one can define categorical trace of $\mC$ in higher category theory. We refer to \cite{BN} for detailed construction. For example, in Example \ref{E: trace of rep}, if we replace the ordinary category $\on{Rep}_E(H)$ by its (cocompleted) $\infty$-categorical version, i.e. the stable $\infty$-category of quasi-coherent sheaves on $\bfB H$, and calculate its trace in the $\infty$-category of stable presentable $\infty$-categories, then the result would be the $\infty$-category of quasi-coherent sheaves on $\frac{H}{c_\sigma H}$. 
\end{rmk}}

\subsection{The categorical trace of finite Hecke categories}

In this subsection, we work in usual algebraic geometry (as opposed to perfect algebraic geometry). We study the category defined in \S \ref{SS: cat of corr} in the following situation. Let $G$ be an affine algebraic group over $k$ with an endomorphism $\sigma$, and let $K\subset G$ be a closed subgroup such that $\sigma$ restricts to an endomorphism of $K$. We assume that $\sigma$ induces a universal homeomorphism, so it induces an automorphism of the underlying \'etale topos of $G$. For example, if $G$ is defined over a finite field $\bF_q$, we can choose $\sigma$ to be the $q$-Frobenius. Another example is $\sigma=\id$.
We denote $\sigma$-twisted adjoint action as
\[\Ad_\sigma: G\times G\to G,\quad \Ad_\sigma(g_1)(g_2)=\sigma(g_1)g_2 g_1^{-1}.\]
Let
$$X=\frac{G}{\Ad_\sigma K},\quad Y=\frac{G}{\Ad_\sigma G},$$
be the quotient stacks of $G$ by the $\sigma$-twisted adjoint action by $K$ and by $G$ respectively.
Note that $f: X\to Y$ is proper if and only if $K$ is a parabolic subgroup of $G$. 

Let $C=X\times_YX$, and let $\on{D}^{C}(X)$ be the category constructed in \S \ref{SS: cat of corr} from the groupoid $C\rightrightarrows X$.
By Lemma \ref{L: fully faithful}, this is a full subcategory of the category $\on{D}(\frac{G}{\Ad_\sigma G})$ of $\Ad_\sigma G$-equivariant sheaves on $G$, which are familiar objects in representation theory.
So we did not construct anything new in this subsection. However, expressing $\Ad_\sigma G$-equivariant sheaves on $G$ as objects in $\on{D}^{C}(X)$ would help us gain new insights of many well-known constructions in representation theory, and will serve as a toy model for the construction in the next section when we move to the affine setting, where the formalism of $\on{D}^{C}(X)$ becomes essential.

As a warm-up, we first translate some well-known results/constructions using the new formulation.

\begin{ex} 
\label{Ex:springer}
Let $\sigma=\id$ and $K=B$.  By Lemma \ref{L: fully faithful}, this is a full subcategory of the category $\on{D}(\frac{G}{\Ad G})$ of $\Ad G$-equivariant sheaves  on $G$.

One can also consider the closed substack $\frac{B}{\Ad B}\subset \frac{G}{\Ad B}$ and the constant sheaf $\delta_{B}$ supported on it. Note that the base change of $\frac{B}{\Ad B}\to \frac{G}{\Ad G}$ along $G\to \frac{G}{\Ad G}$ is
$\widetilde G:=G\times_{\Ad,B}B\to G$, usually referred as the Grothendieck-Springer alteration. It follows from the classical Springer theory that 
$$\End_{\on{D}^{C}(X)}(\delta_{B})\cong E[W]$$ 
is isomorphic to the group algebra of the Weyl group $W$ of $G$. We will see shortly that this $W$-action on $\delta_{B}$ can be induced from \eqref{E: K to End}, applied to the current setting. More generally, for every element $w\in W$, one can regard the intersection cohomology sheaf on $\frac{\overline{BwB}}{\Ad B}$ as an object in $\on{D}^{C}(X)$. Then its idempotents will correspond to Lusztig's unipotent character sheaves under the embedding $\on{D}^{C}(X)\subset \on{D}(\frac{G}{\Ad G})$.

Let $U\subset B$ be the unipotent radical. One can also consider the closed substack $\frac{U}{\Ad B}\subset \frac{G}{\Ad B}$, and the constant sheaf $\delta_{U}$ supported on it. This corresponds to the Springer sheaf in $\on{D}(\frac{G}{\Ad G})$.
Again, it follows from the classical Springer theory that
$\End_{\on{D}^C(X)}(\delta_{U})\cong E[W]$.  
\end{ex}

\begin{ex}
\label{Ex:DL}
Let $k=\bF_q$, and let $K=B$ be the Borel subgroup.  Let $\sigma$ be the $q$-Frobenius. By Lang's theorem, this is a full subcategory of $\on{D}^C(X)=\on{D}(\on{pt}/G(\bF_q))$.
Consider the closed substack $\frac{B}{\Ad_\sigma B}\subset \frac{G}{\Ad_\sigma B}$. By Lang's theorem, the correspondence
\[\frac{B}{\Ad_\sigma B}\leftarrow\frac{B}{\Ad_\sigma B}\times_{\frac{G}{\Ad_\sigma B}}\frac{B}{\Ad_\sigma B}\rightarrow \frac{B}{\Ad_\sigma B}\]
can be identified with the correspondence between discrete (Deligne-Mumford) stacks
\begin{equation}
\label{E:Hk for unit}
B(\bF_q)\backslash \on{pt}\leftarrow B(\bF_q)\backslash G(\bF_q)/B(\bF_q)\rightarrow \on{pt}/B(\bF_q).
\end{equation}
Here $G(\bF_q)$ and $B(\bF_q)$ are regarded as discrete algebraic groups over $\bF_q$.
It follows that for every representation $\rho: B(\bF_q)\to \GL(V)$, there is a local system $\delta_\rho$ supported on $\frac{B}{\Ad_\sigma B}$ via the usual associated construction, which in turn defines an object in $\on{D}^C(X)$. If $\rho$ is the trivial representation, we also denote the corresponding object $\delta_\rho$ by $\widetilde{\delta_e}$. 
Let $H_W$ denote the $E$-module of $E$-valued functions on $B(\bF_q)\backslash G(\bF_q)/B(\bF_q)$, with the algebra structure given by the convolution (compare with \eqref{E:conv prod alge})
\[(f*g)(x)=\sum_{y\in G(\bF_q)/B(\bF_q)} f(y)g(y^{-1}x).\]
It is usually called the Iwahori-Hecke algebra attached to $W$,  with coefficient in $E$. Essentially by definition (and by Example \ref{Ex:examples of correspondences} (5)), we have
\begin{equation} 
\label{E:endo for unit}
\End(\widetilde{\delta_{e}})=H^{\on{op}}_W.
\end{equation}
Similarly, if $\rho=\on{Reg}$ is the regular representation of $B(\bF_q)$, then 
\begin{equation} 
\label{E:endo for reg}
\End(\delta_{\on{Reg}})=E[G(\bF_q)]^{\on{op}},
\end{equation} 
where $E[G(\bF_q)]$ is the group algebra of $G(\bF_q)$.

More generally, for every $w\in W$, the fibers of the map $\overleftarrow{c}: \overrightarrow{c}^{-1}(\frac{BwB}{\Ad_\sigma B})\to \frac{G}{\Ad_\sigma B}$ are just the usual Deligne-Lusztig variety $\on{DL}_w$ associated to $w$.
In addition, there is the isomorphism
\[  \frac{U}{\Ad_\sigma ((U \cap w^{-1}Uw)T^{w^{-1}\sigma})}\cong \frac{BwB}{\Ad_\sigma B},\quad u\mapsto u\dot{w}.\]
where $T^{w^{-1}\sigma}=\{t\in T\mid \sigma(t)=wtw^{-1}\}$ is a finite torus, and $\dot{w}$ is a lifting of $w$ to the normalizer of $T$ in $G$.
Therefore, for any character $\theta$ of this finite torus, there is a local system $\mL_{w,\theta}$ on  $\frac{BwB}{\Ad_\sigma B}$, which can be regarded as an object in $\on{D}^{C}(X)$ by extension to $\frac{G}{\Ad_\sigma B}$ by zero.
Its image under the full embedding $\on{D}^{C}(X)\subset \on{D}(\on{Rep}(G(\bF_q)))$ is the usual Deligne-Lusztig representation.
\end{ex}

\begin{rmk}
\label{R:springer}
In fact, one can check that in the above two examples, the idempotent completion of $\on{D}^C(X)$ is just $\on{D}(Y)$ if $\cha E=0$. 
\end{rmk}

Now, we explain a general construction of morphisms in $\on{D}^{C}(X)$, following some ideas from \cite[\S 6]{XZ}, which in turn generalizes V. Lafforgue's construction of $S$-operators (\cite[\S 6]{La}). We will heavily make use of the formalism of the cohomological correspondences, which is briefly reviewed in \S \ref{Sec:cohomological correspondence}.

We denote by $X_r$ the quotient of $G^r$ by $K^r$ given by the action 
\begin{equation*}
(k_1,\ldots,k_r)\cdot (g_1,\ldots,g_r)=(\sigma(k_1)g_1k_2^{-1},k_2g_2k_3^{-1},\ldots,k_rg_rk_1^{-1}),\quad k_i\in K, \ g_i\in G.
\end{equation*}
In particular $X_1=X=\frac{G}{\Ad_\sigma K}$. Sometimes, we also write $\frac{G\times^K G}{\Ad_\sigma K}$ for $X_2$. 
The map $m_i:G^r\to G^{r-1}$ of multiplying the $i$th and the $(i+1)$th factors induces 
\begin{equation} 
\label{E: mult}
m_i: X_r\to X_{r-1},
\end{equation}
which we still call the multiplication map.
In addition, the map $(g_1,\ldots, g_r)\mapsto (\sigma(g_r),g_1,\ldots, g_{r-1})$ induces 
\begin{equation}
\label{E: psigma}
\pFr: X_r\to X_r,  
\end{equation}
which we call the partial $\sigma$-map. Note that $(_\mathrm{p}\sigma)^r=\sigma$, which justifies the name.
The following equalities are clear.
\begin{equation*}
\label{E:mult vs pFr}
\pFr \cdot m_i=\left\{ \begin{array}{ll}m_{i-1}\cdot\pFr &  i>1\\ m_{r-1}\cdot (\pFr)^2 & i=1. \end{array}\right.
\end{equation*}

\begin{lem}
\label{L: decom of C}
There is a canonical isomorphism
\[C=X\times_YX\cong X_2\]
such that the left projection $\overleftarrow{c}:X\times_YX\to X$ to the first factor is identified with the partial $\sigma$-map \eqref{E: psigma} composed with the multiplication map \eqref{E: mult} and
that  the right projection $\overrightarrow{c}:X\times_YX\to X$ to the second factor is identified with the multiplication map \eqref{E: mult}.
\end{lem}
\begin{proof}
Note that $X\times_YX\cong K\backslash G\times^{K, \Ad_\sigma}G$, with the two projections 
$$\overleftarrow{c},\overrightarrow{c}:  K\backslash G\times^{K, \Ad_\sigma}G\to \frac{G}{\Ad_\sigma K}$$ 
induced by $G\times G\to G, \ (g_1,g_2)\mapsto \sigma(g_1)g_2g_1^{-1}$, and $(g_1,g_2)\mapsto g_2$.
Now the isomorphism $G\times G\cong G\times G,\ (g_1,g_2)\mapsto (g_2g_1^{-1}, g_1)$ induces
\[
 K\backslash G\times^{K, \Ad_\sigma}G \cong  X_2\]
under which, $\overleftarrow{c}$ and $\overrightarrow{c}$ become those maps described in the lemma.
\end{proof}

\begin{construction}
\label{Con: S}
(S-operators \`a la V. Lafforgue.)
Using the above observation, we will construct morphisms between objects in $\on{D}^C(X)$ that come from
\[\on{D}(K\backslash G/K)\to \on{D}(X)\to\on{D}^{C}(X),\]
where the first functor is the pullback along $\frac{G}{\Ad_\sigma K}\to K\backslash G/K$. For an object $\mF\in \on{D}(K\backslash G/K)$, let $\widetilde\mF$ denote its image in $\on{D}(X)$ (and in $\on{D}^{X_2}(X)$). For $\mF_1,\mF_2$, let
$\widetilde\mF_1\tilde\times\widetilde\mF_2$ denote the pullback of $\mF_1\boxtimes\mF_2$ along the projection map 
$$\frac{G\times^KG}{\Ad_\sigma K}\to K\backslash G/K\times K\backslash G/K.$$

Consider the following correspondence
\begin{equation*}
\label{E:Conv corr}
K\backslash G/K\times K\backslash G/K\xleftarrow{\pi} K\backslash G\times^KG/K\xrightarrow{m} K\backslash G/K,
\end{equation*}
where $m$ is induced by multiplication $G\times^K G\to G$ and $\pi$ is the natural projection. Note that $m$ is proper and $\pi$ is smooth (as it is a $K$-torsor).
Similar to the definition of the convolution product for the Satake category (see \eqref{Lus conv}), there is the convolution product on $\on{D}(K\backslash G/K)$:
\begin{equation}
\label{Lus conv2}
\mF_1\star \mF_2=m_!\pi^*(\mF_1\boxtimes\mF_2)=m_*\pi^!(\mF_1\boxtimes\mF_2)\langle-2\dim K\rangle.
\end{equation}
With this convolution product, $\on{D}(K\backslash G/K)$ is a monoidal category, with the unit object being the skyscraper sheaf supported on $K\backslash K/K$, denoted by $\delta_e$.

By adjunction, giving a morphism $u: \mF\to \mG_1\star \mG_2$ is equivalent to giving a cohomological correspondence
\begin{equation*}
\label{E: creation corr}
m^*\mF\to\pi^!(\mG_1\boxtimes\mG_2)\langle-2\dim K\rangle,
\end{equation*}
which, by abuse of notation, is still denoted by $u$. Dually, the datum of a morphism $v:\mG_1\star\mG_2\to \mF$ is the same as the datum of a cohomological correspondence
\begin{equation*}
\label{E: ann corr}
\pi^*(\mG_1\boxtimes\mG_2)\to m^!\mF,
\end{equation*}
still denoted by $v$. In addition, consider the morphism
\[\on{sw}(\id\times \sigma): K\backslash G/K\times K\backslash G/K\to K\backslash G/K\times K\backslash G/K,\]
where $\on{sw}: (K\backslash G/K)^2\to (K\backslash G/K)^2$ swaps the two factors. Then we have the pullback cohomological correspondence (see Example \ref{Ex:examples of correspondences} (2))
\begin{equation} 
\label{E: Frob corr}
\Gamma_{\on{sw}(\id\times\sigma)}^*: (K\backslash G/K\times K\backslash G/K, \mF_1\boxtimes \mF_2)\to (K\backslash G/K\times K\backslash G/K, \mF_2\boxtimes \sigma^*\mF_1).
\end{equation}

Now, given $u: \mF\to \mG_1\star \mG_2$, by applying the formalism of pullback of cohomological correspondences \eqref{ASS:smooth pullback correspondence}-\eqref{ASS:smooth pullback correspondence3} to the following commutative diagram (where the left square is Cartesian)
\[\xymatrix{
\frac{G}{\Ad_\sigma K}\ar[d]&\ar_-m[l]\frac{G\times^KG}{\Ad_\sigma K}\ar@{=}[r]\ar[d]&\frac{G\times^KG}{\Ad_\sigma K}\ar[d]\\
K\backslash G/K&\ar_-m[l] K\backslash G\times^KG/K\ar^-\pi[r]& K\backslash G/K\times K\backslash G/K,
}\]
we obtain a cohomological corresponding $\widetilde u: m^*\widetilde{\mF}\to \widetilde{\mG_1}\tilde\times\widetilde{\mG_2}$. Dually, a morphism $v: \mG_1\star\mG_2\to \mF$ induces $\widetilde v: \widetilde\mG_1\tilde\times\widetilde\mG_2\to m^!\widetilde\mF$. 
In addition, from the following Cartesian diagram
\[
\xymatrix{
K\backslash G/K\times K\backslash G/K\ar^-{\on{sw}(\id\times\sigma)}[rr]\ar[d]&& K\backslash G/K\times K\backslash G/K\ar[d]\\
\frac{G\times^KG}{\Ad_\sigma K}\ar^{\pFr}[rr]&& \frac{G\times^KG}{\Ad_\sigma K},
}\]
we see that the pullback of \eqref{E: Frob corr} gives $\Ga_{\pFr}^*: (\frac{G\times^KG}{\Ad_\sigma K}, \widetilde{\mF_1}\tilde\times\widetilde{\mF_2})\to (\frac{G\times^KG}{\Ad_\sigma K}, \widetilde{\mF_2}\tilde\times \sigma^*\widetilde{\mF_1})$.

Now, suppose we have
\[u: \mF_1\to \mG_1\star \mG_2,\quad v: \mG_2\star \sigma^*\mG_1\to \mF_2\]
we obtain a cohomological correspondence by composition
\[S_{u,v}:=\widetilde v \circ  \Ga_{{_p}\sigma}^* \circ \widetilde u: ( \frac{G}{\Ad_\sigma K},\widetilde{\mF_1})\to (\frac{G\times^KG}{\Ad_\sigma K},\widetilde{\mG_1}\tilde\times\widetilde{\mG_2})\to (\frac{G\times^KG}{\Ad_\sigma K}, \widetilde{\mG_2}\tilde\times\widetilde{\sigma^*\mG_1})\to (\frac{G}{\Ad_\sigma K},\widetilde{\mF_2}).\]
It follows from Lemma \ref{L: decom of C} that $S_{u,v}\in \Hom_{\on{D}^{C}(X)}(\widetilde{\mF_1},\widetilde{\mF_2})$. These operators are called $S$-operators.
\end{construction}

Here is a direct consequence of the above construction. 

\begin{cor}
\label{C:trace str on tilde}
The functor $\on{D}(K\backslash G/K)\to \on{D}(X)\to \on{D}^{C}(X)$ has a natural $\sigma$-twisted trace functor structure. Therefore, the composition $\on{D}(K\backslash G/K)\to \on{D}(\frac{G}{\Ad_\sigma K})\to \on{D}(\frac{G}{\Ad_\sigma G})$ (known as the horocycle transform) admits a natural $\sigma$-twisted trace functor structure.
\end{cor}
The second statement already appeared \cite{BFO,BN,Lu3}. 
\begin{proof}
There is a canonical isomorphism $\widetilde{\mF_1\star\mF_2}\cong\widetilde{\mF_2\star \sigma^*\mF_1}$ induced by
\[u=\id: \mF_1\star\mF_2=\mF_1\star\mF_2,\quad v=\id: \mF_2\star \sigma^*\mF_1=\mF_2\star \sigma^*\mF_1.\]
\end{proof}

If $\mF\in \on{D}_c^b(K\backslash G/K)$, then $\mF$ is dualizable with respect to the monoidal structure on $\on{D}(K\backslash G/K)$ given by the convolution product \eqref{Lus conv2}. Indeed, let
\[\mF^*=\iota^*\bD(\mF)\langle2\dim K\rangle,\]
where $\bD$ is the Verdier dual on $\on{D}_c^b(K\backslash G/K)$ and 
$$\iota: K\backslash G/K\to K\backslash G/K,\quad  g\mapsto g^{-1}.$$ 
Then $\mF^*$ is both the left and the right dual of $\mF$.
Namely, using the projection $\on{pt}\to \mathbf BK=K\backslash K/K$, we may regard the cohomological correspondences in Example \ref{Ex:examples of correspondences} (4) as cohomological correspondences
\[\coev'_{\mF}\in \on{Corr}_{K\backslash G/K}((K\backslash K/K, E), (K\backslash G/K\times K\backslash G/K, \mF\boxtimes \bD \mF)).\]
\[\ev'_{\mF}\in\on{Corr}_{K\backslash G/K}((K\backslash G/K\times K\backslash G/K, \bD\mF\langle2\dim K\rangle\boxtimes \mF), (K\backslash K/K,E)) \]
There is the following commutative diagram
\[
\xymatrix{
K\backslash K/K\ar[d]&\ar[l] K\backslash G/K\ar^-\Delta[r]\ar^{\Delta'}[d]& K\backslash G/K\times K\backslash G/K\ar^{\id\times\iota}[d]\\
K\backslash G/K&\ar[l] K\backslash G\times^KG/K\ar[r]& K\backslash G/K\times K\backslash G/K,
}\]
where $\Delta'$ is the map induced by $G\to G\times^KG,\ g\mapsto (g,g^{-1})$ so that the left commutative square is Cartesian. Note that the right commutative is also base changeable in the sense of Definition \ref{D: basechangeable}.
Then the pushforward of $\coev'_{\mF}$ and $\ev'_\mF$ along $\Delta'$ in the sense of \eqref{ASS:pushforward correspondence} gives
\begin{equation*}
\label{E:unandcoun}
\coev_{\mF}: \delta_e\to \mF\star \mF^*,\quad \ev_{\mF}:\mF^*\star\mF\to \delta_e
\end{equation*}
making $\mF^*$ the left dual of $\mF$. Similarly, there are morphisms $\coev'_{\mF}:\delta_e\to \mF^*\star \mF$ and $\ev'_{\mF}:\mF\star\mF'\to \delta_e$ making $\mF^*$ the right dual of $\mF$. 
From the construction, there is the particular element 
\begin{equation}
\label{E: SLa}
S_{\mF}:=S_{\coev_\mF,\ev_\mF}\in \End_{\on{D}^{X_2}(X)}(\widetilde{\delta_e}).
\end{equation}
This is the original $S$-operator constructed by V. Lafforgue.

It follows from Proposition \ref{L:trace cat} and Lemma \ref{C:trace str on tilde} that the functor $\on{D}_c^b(K\backslash G/K)\to \on{D}^{C}(X)$ factors through
\[
\Tr_\sigma(\on{D}_c^b(K\backslash G/K))\to \on{D}^{C}(X).
\]
The element $S_{u,v}$ defined in the proof of Proposition \ref{L:trace cat} goes exactly to the $S$-operator constructed in Construction \ref{Con: S}. In particular, the element \eqref{E:abs S} goes to \eqref{E: SLa}.

\begin{ex}
Now we specify the above general discussions to the particular cases.
We consider the case $G$ is reductive, $K=B$. Recall that $\on{D}_c^b(B\backslash G/B)$ is usually referred as the finite Hecke category.
If $E=\Ql$ and $k$ is a finite field,  for $w\in W$, let $\IC_w$ denote the intermediate extension of local system $\Ql[\ell(w)](\ell(w)/2)$ on $B\backslash BwB/B$ to its closure, where $\ell(w)=\dim BwB/B$ is the usual length function on $W$. (Here we fix a half Tate twist.)
Then the subcategory of $\on{D}_c^b(B\backslash G/B)$ formed by $\{\IC_w[d](\frac{d}{2}),\ w\in W, d\in \bZ\}$ form a semisimple monoidal subcategory. Its Grothendieck ring is isomorphic to
the finite Hecke algebra $\mathbb H_W$ over $\bZ[v,v^{-1}]$ (with the inclusion $\bZ[v,v^{-1}]\subset \mathbb H_W$ induced by the inclusion of $\{\on{IC}_0(i), \ i\in\bZ\}$ into the Hecke category). When $v=1$, $\mathbb H_W$ specializes to the group algebra of $W$, and when $v=\sqrt{q}$, it specialize to the Hecke algebra $H_W$ mentioned before.
It follows from Lemma \ref{L: K iso End} that
\begin{itemize}
\item In the case as in Example \ref{Ex:springer}, there is a natural action of $\bH^{\on{op}}_W$ on $\delta_B$. But it is clear that $S_{\IC_w(i)}=S_{\IC_w}$. Therefore, the action of $\bH^{\on{op}}_W$ factors through the action of $W$ on $\delta_B$. In fact, our construction is nothing but a reformation of Ginzburg's construction of the Springer action by correspondences.
\item In the case as in Example \ref{Ex:DL}, the map \eqref{E: K to End3} induces a canonical action of $H_W^{\on{op}}$ on $\widetilde{\delta_e}$. On the other hand,
$\End\widetilde{\delta_e}=H_W^{\on{op}}$ by \eqref{E:Hk for unit} . It follows that we obtain a morphism $H^{\on{op}}_W\to H^{\on{op}}_W$. Proposition \ref{P:S=T finite} below implies that this map is the identity.
\end{itemize}
\end{ex}
\begin{rmk}
The map \eqref{E: K to End2} (and its $\sigma$-twisted analogue) should allow one to deduce information of endomorphism rings of more general objects in $\on{D}^C(X)$ (e.g. the Deligne-Lusztig representations). 
\end{rmk}

Let $\mF\in \on{D}_c^b(B\backslash G/B)$. Assume that it is equipped with a canonical isomorphism $\sigma^*\mF\cong \mF$ (a Weil structure). Let $f_\mF$ denote the function on $B(\bF_q)\backslash G(\bF_q)/B(\bF_q)$ via Grothendieck's sheaf-to-function dictionary. It follows from \eqref{E:Hk for unit}  that this function gives an element in $\End(\widetilde{\delta_e})$, denoted by $S'_\mF$.
\begin{prop}
\label{P:S=T finite}
Under the map $\End_{\Tr_{\sigma}(\on{D}^b_c(B\backslash G/B))}(\widetilde{\mathbf{1}})\to \End (\widetilde{\delta}_e)$, the element $S_\mF$ maps to $S'_{\mF}$.
\end{prop}
\begin{proof}
We give a detailed proof of this proposition. 
A similar, but more complicated argument was used in \cite[\S 6.3]{XZ} to establish its affine analog (which will be stated as part of Theorem \ref{T: trace of geom Sat} below).

Let us denote $\mZ=m^{-1}(\frac{B}{\Ad_\sigma B})\subset \frac{G\times^BG}{\Ad_\sigma B}$. We consider the following diagram
\begin{small}
\begin{equation}\label{E: cal coh corr to local}
\xymatrix@C=15pt{
\on{pt}&\ar[l] G/B\ar^-{\Delta}[r]& G/B\times G/B\ar@{}[drr]|{\mathrm{(A)}}&& G/B\times G/B\ar_-{\on{sw}(\id\times \sigma)}[ll]\ar@{}[dr]|{\mathrm{(B)}}&\ar_-\Delta[l] G/B\ar[r]& \on{pt}\\
\on{pt}\ar[d]\ar@{=}[u]&\ar[l] \mZ_1\ar[r]\ar[d]\ar[u]& G\times G\ar[d]\ar[u]\ar@{}[drr]|{\mathrm{(C)}}&&\ar_-{\on{sw}(\id\times \sigma)}[ll]G\times G\ar[d]\ar[u]\ar@{}[dr]|{\mathrm{(D)}}&\ar[l]\mZ_2\ar[r]\ar[d]\ar[u]&\on{pt}\ar[d]\ar@{=}[u]\\
\frac{B}{\Ad_\sigma B} & \mZ\ar[l]\ar[r]& \frac{G\times^BG}{\Ad_\sigma B}&&\ar_{\pFr}[ll]\frac{G\times^BG}{\Ad_\sigma B}&\ar[l] \mZ\ar[r]& \frac{B}{\Ad_\sigma B},
}
\end{equation}
\end{small}
where 
\begin{itemize}
\item the two vertical maps in Square $(C)$ are natural projections, and the map $G\times G\to G\times G$ is given by $(g_1, g_2)\mapsto (\sigma(g_2), g_1)$ so that $(C)$ is Cartesian;
\item the variety $\mZ_2$ is defined so that the diagram $(D)$ is Cartesian, i.e. $\mZ_2=(g_1,g_2)\in G\times G,\ g_1g_2\in B$;
\item the left vertical map $G\times G\to G/B\times G/B$ in diagram $(A)$ is is given by $(g_1,g_2)\mapsto (g_1B, g_2^{-1}B)$, and the right vertical map $G\times G\to G/B\times G/B$ in diagram $(A)$ is given by $(g_1,g_2)\mapsto (g_1^{-1}B,g_2B)$,  so that $(A)$ and $(B)$ are Cartesian;
\item the variety $\mZ_1=G$, with the map $\mZ_1\to G\times G$ given by $g\mapsto (g,g^{-1})$ so that the composition $\mZ_1\to G\times G\to \frac{G\times^BG}{\Ad_\sigma B}$ factors through $\mZ$ and the resulting morphism $\mZ_1 \to \mZ$ is smooth, and with the map $\mZ_1\to G/B$ being the natural projection (and therefore is also smooth).
\end{itemize}
Taking the fiber products of each row, we obtain the commutative diagram
\begin{equation*}
\xymatrix{
\on{pt}&\ar[l] (G/B)(\bF_q)\ar[r] & \on{pt}\\
\on{pt}\ar@{=}[u]\ar[d]&\ar[l] \mW \ar[u]\ar[d]\ar[r] & \on{pt}\ar@{=}[u]\ar[d]\\
B(\bF_q)\backslash \on{pt}&\ar[l] B(\bF_q)\backslash G(\bF_q)/B(\bF_q)\ar[r]& \on{pt}/B(\bF_q),
}\end{equation*}
where $\mW:=\mZ_1\times_{G\times G, \on{sw}(\id\times \sigma)}\mZ_2\cong\{g\in G\mid \sigma(g)^{-1}g\in B\}$, which maps to $\mZ_1=G$ by sending $g$ to $\sigma(g)$ and maps to $\mZ_2$ by sending $g$ to $(\sigma(g)^{-1},g)$. The map $\mW\to (G/B)(\bF_q)$ sends $g$ to $gB$, so that the fibers of this map are isomorphic to $B$. The map $\mW\to B(\bF_q)\backslash G(\bF_q)/B(\bF_q)$ sends $g$ to its image in $B(\bF_q)\backslash (G/B)$ (which automatically lands in $B(\bF_q)\backslash G(\bF_q)/B(\bF_q)$). 

By definition, $S_{\mF}: (B(\bF_q)\backslash \on{pt},\widetilde{\delta_e})\to (\on{pt}/B(\bF_q),\widetilde{\delta_e})$ is a cohomological correspondence supported on the bottom row of the above diagram. On the other hand, there is the cohomological correspondence 
$$S'_\mF:=\on{ev}_{\mF}\circ \Ga_{\on{sw}(\id\times\sigma)}^* \circ\on{coev}_{\mF}: (\on{pt},E)\to(\on{pt},E)$$ supported on the top row, which by a trace formula of Braverman-Varshavsky (see \cite[Lemma A.2.22]{XZ}), is given by the function $f_\mF$ on $B(\bF_q)\backslash G(\bF_q)/B(\bF_q)$.

Since both maps $\mW\to (G/B)(\bF_q)$ and $\mW\to B(\bF_q)\backslash G(\bF_q)/B(\bF_q)$ are equidimensionally smooth (of dimension $\dim B$), we can pullback $S_\mF$ and $S'_\mF$ to the middle row using the formalism explained in \eqref{ASS:smooth pullback correspondence}-\eqref{ASS:smooth pullback correspondence3}. It is enough to show that their pullbacks are equal.

Note that both $S_{\mF}$ and $S'_{\mF}$ are composition of three cohomological correspondences.
Since commutative squares (A)-(D) are all Cartesian, Lemma \ref{AL:pullback compatible with composition} implies that smooth pullbacks commute with compositions of cohomological correspondences. Therefore, it remains to observe that in Diagram \eqref{E: cal coh corr to local}
\begin{itemize}
\item
the pullback of the cohomological correspondence $\widetilde{\coev_\mF}: (\frac{B}{\Ad_\sigma B}, \widetilde{\delta_e})\to (\frac{G\times^BG}{\Ad_\sigma B}, \widetilde{\mF}\tilde\times \widetilde{\mF^*})$ supported on $\mZ$ to the middle row  is equal to the pullback of $\on{un}_{\mF}: (\on{pt},E)\to (G/B\times G/B, \mF\boxtimes \bD\mF)$ supported on $G/B$, and similarly the dual statement holds for $\ev_\mF$ and $\on{coun}_{\mF}$; and
\item the pullback of $\Ga_{\pFr}^*: (\frac{G\times^BG}{\Ad_\sigma B}, \widetilde{\mF}\tilde\times\widetilde{\mF^*})\to (\frac{G\times^BG}{\Ad_\sigma B}, \widetilde{\mF^*}\tilde\times\widetilde{\mF})$ to the middle row is equal to the pullback of $\Ga_{\on{sw}(\id\times\sigma)}^*: (G/B\times G/B, \mF\boxtimes \bD \mF)\to (G/B\times G/B, \bD \mF\boxtimes \mF)$.
\end{itemize}
\end{proof}

\section{The categorical trace of the geometric Satake}
\label{S: cattrSat}

In this section, we move to the affine setting. While the basic idea remains the same, the technical details are more involved.

Let $F$ denote a local field with finite residue field $k=\bF_q$. The $q$-Frobenius is denoted by $\sigma$. 
Recall the notations from \S \ref{S: affGrass}, in particular \eqref{ramified Witt vector} and \eqref{discs}. We assume that $G$ is a connected reductive group over $\mO$. For simplicity, we write $K=L^+G$.

It is clear from the discussion in the previous section that to study the representation theory of $p$-adic groups, it is important to study the category of $\Ad LG$ or $\Ad_\sigma LG$-equivariant sheaves on $LG$. However, one of the difficulties in the affine setting is that the quotient $Y=\frac{LG}{\Ad LG}$ or $Y=\frac{LG}{\Ad_\sigma LG}$ is far from being algebraic so it is not clear how to make sense of the category of sheaves on it. Our approach (which should be equivalent to the approach outlined in \cite{Ga}) is to replace $\on{D}(Y)$ by the category $\on{D}^\Hk(X)$ introduced in \S \ref{SS: cat of corr}, where $X=\frac{LG}{\Ad K}$ or $X=\frac{LG}{\Ad_\sigma  K}$ and $\Hk=\Hk(X):=X\times_YX$\footnote{More precisely, one needs to define $\on{D}(Y)$ as the homotopy category of $\Shv^{\Hk}(X)$ as outlined in Remark \ref{R: DC enh} (2). But we do not need the more sophisticated version in this article.}.  As we shall see, although neither $X$ or $\Hk$ are algebraic stacks, they can be approximated by finite dimensional algebraic stacks so $\on{D}^{\Hk}(X)$ makes sense. Another advantage to replace $\on{D}(Y)$ by $\on{D}^\Hk(X)$ is that the latter is realized via cohomological correspondences and is related to global moduli spaces (Shimura varieties or moduli of global Shtukas). Let us also mention that in Fargues-Scholze's approach to the local Langlands correspondence for $p$-adic groups, they replace the sought-after category $\on{D}(\frac{LG}{\Ad_\sigma LG})$ by the category of sheaves on the moduli of $G$-bundles on the Fargues-Fontaine curve. It should relate to our category $\on{D}^\Hk(X)$ via the nearby cycle functors. However, it is not clear to the author how to access $\on{D}(\frac{LG}{\Ad LG})$ in Fargues-Scholze's approach.

Once the appropriate category is defined, we can take the categorical trace of the geometric Satake correspondence, which will be the key ingredient of our arithmetic applications in the next section. In this section, we will only discuss the case $Y=\frac{LG}{\Ad_\sigma LG}$. However, almost all constructions also apply to the case $Y=\frac{LG}{\Ad LG}$. See Remark \ref{R:affine springer} .

\subsection{Moduli of local Shtukas}

Since
$X=\frac{LG}{\Ad_\sigma K}$, $Y=\frac{LG}{\Ad_\sigma LG}$ and $X\times_YX$ are not algebraic stacks, to extend to previous discussions to the affine setting the first step is to make sense the category defined in \S \ref{SS: cat of corr} in this setting. It turns out 
the $\frac{LG}{\Ad_\sigma K}$ in this case has an interpretation as the moduli of local Shtukas.

Let $R$ be a perfect $k$-algebra.  If $\mE$ is a $G$-torsor on $D_R$, its pullback along the $q$-Frobenius $\sigma:D_R\to D_R$ is denoted by ${^\sigma}\mE$.
A local $r$-iterated $G$-Shtuka over $\Spec R$ is a sequence of $G$-torsors $\mE_1,\ldots, \mE_r$ on $D_{R}$, together with a chain of modifications
\[\mE_r\dashrightarrow \mE_{r-1}\dashrightarrow\cdots\dashrightarrow \mE_0:={^\sigma}\mE_r.\]
We say its singularities are bounded by $\mmu=(\mu_1,\ldots,\mu_r)$ if
\[\inv(\mE_i\dashrightarrow \mE_{i-1})\leq \mu_i.\]
We define the moduli of local $r$-iterated $G$-Shtukas $\on{Sht}^{r,\loc}$ as the prestack that assigns every $R$ the groupoid of local $r$-iterated $G$-Shtukas on $\Spec R$. It is the union (over $\mmu$) of closed sub-prestacks $\on{Sht}^{\loc}_{\mmu}$ consisting of those local Shtukas with singularities bounded by $\mmu=(\mu_1,\ldots,\mu_r)$.  In particular, if $r=1$, $\on{Sht}^{r,\loc}$ is denoted by $\Sht^\loc$, called the moduli of local $G$-Shtukas, which is the union of closed sub-prestacks $\Sht^\loc_\mu$.

Since every $G$-torsor on $W(R)$ can be trivialized \'etale locally on $R$ (e.g. see \cite[Lemma 1.3]{Zh1}), there is a natural isomorphism
\begin{equation}
\label{R:pre of Sht}
\Sht^\loc\cong \frac{LG}{\Ad_\sigma K},\quad \mbox{which restricts to the isomorphism } \Sht^\loc_\mu\cong \frac{\Gr_\mu^{(\infty)}}{\Ad_\sigma K},\ \ \ \mu\in \xcoch(T)/W
\end{equation}
Namely, let $\Sht^{\loc,\Box}$ denote the $K$-torsor over $\Sht^\loc$ classifying local $G$-Shtukas $(\beta:\mE\dashrightarrow {^\sigma}\mE)$ together with a trivialization $\epsilon: \mE\simeq \mE^0$. Then the composition $\mE^0\xrightarrow{\epsilon^{-1}}\mE\stackrel{\beta}{\dashrightarrow}{^\sigma}\mE\xrightarrow{\sigma(\epsilon)}\mE^0$ defines an element $g\in LG$ and induces the isomorphism $\Sht^{\loc,\Box}\cong LG$, under which the $K$-action on $\Sht^{\loc,\Box}$ is identified with the action $\Ad_\sigma$ on $LG$. In addition, if $\inv(\beta)\leq \mu$, then $g\in \Gr_{\mu}^{(\infty)}$. Taking the quotient gives the above isomorphism.

\begin{ex}\label{no singularity}
If $\mmu=\mu=0$, $\Sht^{\loc}_{0}=\mathbf B G(\mO)$ is the the classifying stack of the profinite group $G(\mO)$. 
\end{ex}

\begin{ex}\label{loc Sht and pdiv}
If $\mO=\bZ_p$, the category of $\GL_n$-Shtukas over $R$ with singularities bounded by the coweight $\omega_i$ (as defined in Example \ref{minu Sch}) is equivalent to the category of $p$-divisible groups of dimension $i$ and height $n$ over $R$. The functor sends a $p$-divisible group $\bX$ to its Dieudonn\'e module, and a theorem of Gabber's asserts that this is an equivalence. Note that  $\GL_n$-Shtukas coming from $p$-divisible groups are very special. Namely their singularities must be bounded by a minuscule coweight.
\end{ex}

\begin{rmk}
What we just defined are local Shtukas with singularities at the closed point $s\in D$. One can also define local Shtukas with singularities at the generic point $\eta\in D$, or even with singularities moving along $D$. 
In mixed characteristic, local Shtukas with singularities along $D$ are closely related the Breuil-Kisin modules (\cite{SW}).
\end{rmk}

\begin{rmk}
In equal characteristic, one can define global $G$-Shtukas, where $D$ is replaced by a global algebraic curve over a finite field. This is what Drinfeld originally invented, which vastly generalizes the notion of elliptic modules. However, currently it is not clear whether there exists analogue of global $G$-Shtukas in the number field setting. We refer to \cite{S} for some speculations.
\end{rmk}

The isomorphism \eqref{R:pre of Sht} clearly generalizes to an isomorphism 
$$\frac{LG\times^KLG\times^K\cdots\times^KLG}{\Ad_\sigma K}\cong \on{Sht}^{r,\loc},$$ under which the map analogous to \eqref{E: psigma} also admits a moduli interpretation, usually called the \emph{partial Frobenius map}.
Fix a sequence of dominant coweights $\mmu=(\mu_1,\ldots,\mu_r)$ as above and set $\mu_0=\sigma(\mu_r)$. Then we define the partial Frobenius map as
\begin{equation*}
\label{hpf}
\xymatrix@R=0pt{
F_\mmu\colon \on{Sht}_{(\mu_1,\ldots,\mu_r)}^{r,\loc} \ar[r] & \on{Sht}^{r,\loc}_{(\sigma(\mu_{r}),\mu_1,\ldots,\mu_{r-1})} \\ 
(\mE_r\stackrel{\beta_r}{\dashrightarrow} \mE_{r-1}\stackrel{\beta_{r-1}}{\dashrightarrow}\cdots\stackrel{\beta_2}{\dashrightarrow} \mE_1\stackrel{\beta_1}{\dashrightarrow} {^\sigma}\mE_r) \ar@{|->}[r] & (\mE_{r-1}\stackrel{\beta_2}{\dashrightarrow}\cdots\stackrel{\beta_r}{\dashrightarrow} \mE_1\stackrel{\beta_1}{\dashrightarrow} {^\sigma}\mE_r\stackrel{\sigma(\beta_r)}{\dashrightarrow} {^\sigma}\mE_{r-1}),
}
\end{equation*}
which is an isomorphism of prestacks (as functors over perfect $k$-algebras).

We define the following correspondence of prestacks 
\begin{equation}
\label{E;Hk Shtuka}
\xymatrix{&\Hk^{s,t}(\Sht^\loc)\ar_{\overleftarrow{h}^\loc}[dl]\ar^{\overrightarrow{h}^\loc}[dr]&\\
\Sht^{s,\loc}&&\Sht^{t,\loc},
}
\end{equation}
where $\Hk^{s,t}(\Sht^\loc)$ is the prestack classifying, for each perfect $k$-algebra $R$, the following commutative diagram of modifications of $G$-torsors over $D_R$:
\begin{equation*}\label{qi for diff type}
\xymatrix@C=35pt{
\mE'_s \ar@{-->}[r]^{\beta'_s} \ar@{-->}[d]_\beta & \cdots  \ar@{-->}[r]^{\beta'_2} & \calE'_1 \ar@{-->}[r]^{\beta'_1} & {}^\sigma\mE'_s \ar@{-->}[d]^{\sigma(\beta)}\\
\mE_t \ar@{-->}[r]^{\beta_t} & \cdots  \ar@{-->}[r]^{\beta_2} & \calE_1 \ar@{-->}[r]^{\beta_1} & {}^\sigma\mE_t,
}
\end{equation*}
such that the top row (resp. bottom row) defines an $R$-point of $\Sht^{s,\loc}$ (resp. $\Sht^{t,\loc}$).
We call such a diagram a \emph{Hecke correspondence} from the top row to the bottom row. When $s=t=1$, we denote $\Hk^{s,t}(\Sht^\loc)$ simply by $\Hk(\Sht^\loc)$.
Then similar to \eqref{R:pre of Sht}, there is a canonical isomorphism 
$$\Hk(\Sht^\loc)\cong \frac{LG}{\Ad_\sigma K}\times_{\frac{LG}{\Ad_\sigma LG}} \frac{LG}{\Ad_\sigma K},$$
with $\overleftarrow{h}^\loc$ and $\overrightarrow{h}^\loc$ being natural projections.
For two sequences of dominant coweights $\la_\bullet$ and $\mmu$, let 
\[\Sht^{\loc}_{\la_\bullet\mid\mmu}:=(\overleftarrow{h}^\loc)^{-1}(\Sht^\loc_{\la_\bullet})\cap (\overrightarrow{h}^\loc)^{-1}(\Sht^\loc_{\mmu})\]
and sometimes denote the restriction of $\overleftarrow{h}^\loc$ and $\overrightarrow{h}^\loc$ to $\Sht^{\loc}_{\la_\bullet\mid\mmu}$ by $\overleftarrow{h}^\loc_{\la_\bullet}$ and $\overrightarrow{h}^\loc_\mmu$.
If we further require $\beta$ to have relative position $\leq \nu$, we obtain a closed sub-prestack $\Sht_{\la_\bullet|\mmu}^{\nu,\loc}$ of $\Sht_{\la_\bullet|\mmu}^{\loc}$.

\begin{rmk}
\label{R:fiber of hloc}
Note that fibers of the morphism $\overleftarrow{h}^\loc: \Hk(\Sht^\loc)\to\Sht^\loc$ are just affine Grassmannians. For every $\mu$, the fibers
of $\overleftarrow{h}^\loc: (\overrightarrow{h}^\loc)^{-1}(\Sht^{\loc}_\mu)\to \Sht^\loc$ are affine Deligne-Lusztig varieties (e.g. see \cite[\S 3.1]{Zh1} for the definition).
\end{rmk}

As we already see from Example \ref{no singularity}, $\Sht^\loc_\mu$ is not algebraic so the general construction in \S \ref{SS: cat of corr} does not apply in the current situation directly. In order to define the category in this setting, we need to replace the moduli spaces by their finite dimensional approximations. 
We sketch the construction here and refer to  \cite[\S 5]{XZ} for details. A group theoretical description of the constructions is given in Remark \ref{R: grp theoretical Sht}.

Recall that $K$ acts on the convolution Grassmannians. Given a sequence $\mmu=(\mu_1,\ldots,\mu_r)$ of dominant coweights, an integer $m$ is called $\mmu$-large if it is $\geq \langle\sum_i\mu_i,\al_h\rangle$, where $\al_h$ is the highest root of $G$.
It is easy to see that given a pair of non-negative integers $(m,n)$ ($(m,n)=(\infty,\infty)$ allowed), if $m-n$ is $\mmu$-large
the action $K$ on $\Gr_\mmu^{(n)}$ factors through the action of $K_{m}$ (cf. \cite[Lemma 3.1.7]{XZ}). In this case, the fpqc quotient stack 
$\Hk^{\loc(m)}_\mmu:=[K_m\backslash \Gr_\mmu]$ is called the $m$-restricted local Hecke stack. 
By writing
\[
[K_m\backslash\Gr_\mmu]\cong [K_m\backslash\Gr^{(n)}_\mmu/K_n]\xrightarrow{t_\leftone\times t_\rightone} \bfB K_n\times \bfB K_m,
\]
we see that over $\Hk_{\mmu}^{\loc(m)}$, the first map $t_\rightone$ defines the $K_m$-torsor 
\begin{equation*}\label{torsor2}
\Gr_\mmu \to [K_m\backslash\Gr_\mmu],
\end{equation*}
and the second map $t_\leftone$ defines the $K_n$-torsor
\begin{equation*}\label{torsor1}
[K_m\backslash\Gr_{\mmu}^{(n)}]\to  [K_m\backslash\Gr_\mmu].
\end{equation*}
We write $\calE_{\rightone}|_{D_{m}}$ and $\calE_{\leftone}|_{D_n}$ for these two \emph{canonical} torsors. 

For $m\geq n$, let 
$$\res^m_n: \bfB K_m\to \bfB K_n$$ 
denote the map induced by the quotient map $\pi_{m,n}:K_m\to K_n$. Clearly, for $m\leq m'\leq m''$,
\begin{equation}
\label{E:ex and q compatibility}
\res_{m}^{m'} \circ \res_{m'}^{m''} = \res_{m}^{m''},
\end{equation}
where we allow $m''$ to be $\infty$.
Then for $m\leq m'$ and $n\leq n'$ ($(m',n')=(\infty,\infty)$ allowed) such that both $m'-n'$ and $m-n$ are $\mmu$-large, there
is a restriction map
$$\res^{m'}_m: \Hk^{\loc(m')}_{\mmu}\to \Hk^{\loc(m)}_{\mmu}$$ such that
the following diagram is commutative and the right square is Cartesian.
\begin{equation*}
\label{E:restriction and torsor commutative}
\xymatrix{
\bfB K_{n'} \ar[d]^{\res^{n'}_n} &
\Hk_\mmu^{\loc(m')} \ar[l]_-{t_\leftone} \ar[r]^-{t_\rightone} \ar[d]^{\res^{m'}_{m}} &  \bfB K_{m'} \ar[d]^{\res^{m'}_{m}}
\\ 
\bfB K_n & \Hk_\mmu^{\loc(m)} \ar[l]_-{t_\leftone} \ar[r]^-{t_\rightone} &  \bfB K_m.
}
\end{equation*}
In addition, these restrictions maps  satisfy the natural compatibility condition \eqref{E:ex and q compatibility}.

Now fix a sequence $\mmu$ of dominant coweights and choose a pair of integers $(m,n)$ such that $m-n$ is $\mmu$-large.
For a perfect $k$-algebra $R$, an $(m,n)$-restricted  local iterated $G$-Shtuka over $\Spec R$ with singularities bounded by $\mmu$ consists of
\begin{itemize}
\item an $R$-point of $ \Hk^{\loc(m)}_\mmu$,  and 
\item an isomorphism of $K_n$-torsors over $\Spec R$
$$\psi:{^\sigma} (\mE_\leftone|_{D_{n}}) \simeq(\mE_\rightone|_{D_{m}})|_{D_{n}}.$$ 
\end{itemize}
We define $\Sht^{\loc(m,n)}_\mmu$ as the prestack that classifies all $(m,n)$-restricted local $G$-Shtukas.  In other words, $\Sht_\mmu^{\loc(m,n)}$ is defined as the Cartesian product of $\Hk_\mmu^{\loc(m)}$ with a Frobenius graph
\begin{equation}
\label{E:mS restricted as a Cartesian product}
\xymatrix@C=90pt{
\Sht_\mmu^{\loc(m,n)} \ar[r]^{\varphi^{\loc(m,n)}} \ar[d]
&  \Hk_\mmu^{\loc(m)} \ar[d]^-{ t_\leftone\times \res^{m}_{n}\circ t_\rightone }
\\ \bfB K_n \ar[r]^{1\times \sigma} & \bfB K_n \times \bfB K_n,
}
\end{equation}
In particular, $\Sht_\mmu^{\loc(m,0)}=\Hk_\mmu^{\loc(m)}$, and $\Sht_\mmu^{\loc(\infty,\infty)}=\Sht_\mmu^\loc$.
Note that $\varphi^{\loc(m,n)}$ is a perfectly smooth morphism of relative dimension $n \dim G$. 
 
 Let $(m',n')$ and $(m,n)$ be two pairs of non-negative integers such that $m\leq m'$, $n\leq n'$ are both $m'-n'$ and $m-n$ are $\mmu$-large ($(m',n')=(\infty,\infty)$ allowed). 
There is the \emph{restriction morphism}
\begin{equation}\label{trunc}
\res^{m',n'}_{m,n}: \Sht^{\loc(m',n')}_\mmu\longto \Sht_\mmu^{\loc(m,n)},
\end{equation}
which is perfectly smooth of relative dimension $(m-n)-(m'-n')$, compatible with \eqref{E:mS restricted as a Cartesian product}, and satisfying
\begin{equation}
\label{true-to-trun}
\res^{m_2,n_2}_{m_1,n_1}\circ\res^{m_3,n_3}_{m_2,n_2}=\res^{m_3,n_3}_{m_1,n_1}.
\end{equation} 
 
\begin{ex}
\label{Ex:mStau mn}
When $\mu=0$ or more generally $\mu=\tau\in\xcoch(Z_G)$, we have 
$$\Sht^{\loc(n,n)}_\tau\cong \bfB G(\mO/\varpi^n).$$ 
The restriction map $\res_{n,n}: \Sht_\tau^\loc\cong\bfB G(\mO)\to \Sht_\tau^{\loc(m,n)}\cong\bfB G(\mO/\varpi^n)$ is induced by the natural map $G(\mO)\to G(\mO/\varpi^n)$.

\end{ex}

\begin{rmk}
\label{Ex:Sht and GZip}
As mentioned in Example \ref{loc Sht and pdiv}, moduli of local shtukas can be regarded as a generalization of the moduli of $p$-divisible groups with $G$-structures. It is natural to expect that the moduli of restricted local shtukas can be regarded as a generalization of moduli of truncated Barsotti-Tate groups with $G$-structures. It turns out that they are indeed closely related but the relation is slightly more complicated than the unrestricted case. As explained in \cite[Lemma 5.3.6]{XZ}, if $\mu$ is minuscule, and $(m,n)=(2,1)$,
there is a natural perfectly smooth map
\begin{equation}
\label{E:Sht-Zip}
\Sht_\mu^{\loc(2,1)}\to G\on{-Zip}^\pf_\mu.
\end{equation}
Here, $G\on{-Zip}^\pf_\mu$ is the perfection of the moduli of $F$-zips with $G$-structures as defined in \cite{PWZ}, which is homeomorphic to the moduli of $1$-truncated Barsotti-Tate groups with $G$-structure. 

We describe the morphism \eqref{E:Sht-Zip} in  the case when $G=\GL_n$ and $\mu=\omega_i$. In this case $G\on{-Zip}_{\omega_i}(R)$ classifies quadruples $(M, C^\bullet M, D_\bullet M, \phi_0,\phi_1)$, where $M$ is a finite projective $R$-module of rank $n$, $M=C^0M\supset C^1M\supset C^2M=0$ is a decreasing filtration and $D_{-1}M=0\subset D_0M\subset D_1M$ is an increasing filtration, such that $\phi_0:\sigma^*\on{gr}_C^0M\cong \on{gr}_0^DM$ is an isomorphism of finite projective $R$-modules of rank $i$ and $\phi_1: \sigma^*\on{gr}_C^1M\cong \on{gr}_1^DM$ is an isomorphism of finite projective $R$-modules of rank $n-i$. The relation between $F$-zips of the above type and $1$-truncated Barsotti-Tate groups of height $n$ and dimension $i$ was explained in \cite[\S 9.3]{PWZ}.

Note that giving an $R$-point of $\Sht_{\omega_i}^{\loc(2,1)}$, there is an honest map of finite projective $R$-modules 
$$\mE_\leftone|_{R}\xrightarrow{\beta} \mE_\rightone|_{R}\cong \sigma^*(\mE_\leftone|_{R}),$$ whose cokernel is a finite projective $R$-module of rank $i$.
Then \eqref{E:Sht-Zip} sends an $R$-point of $\Sht_{\omega_i}^{\loc(2,1)}$ to the $R$-point of $G\on{-Zip}_{\omega_i}$ represented by $M=\mE_\rightone|_{R}$, $C^1M=\on{Im} \beta$, and $D_0M=\ker(\sigma^*\beta)$. The isomorphism $\phi_1: \sigma^*(\on{gr}_C^1M)=\sigma^*\on{Im}\beta\cong \on{gr}^D_1M=\frac{M}{\ker\sigma^*\beta}$ is the canonical one. 
The canonical isomorphism $\phi_0: \sigma^*(\on{gr}_C^0M)=\sigma^*\frac{M}{\on{Im}\beta}\cong \ker(\sigma^*\beta)=\on{gr}^0_DM$ can be deduced from \cite[(1.3.2)]{Zh1}.
\end{rmk}

There exist Hecke correspondences between the various moduli spaces of restricted local iterated shtukas that are compatible with the correspondence \eqref{E;Hk Shtuka} via the restriction morphism \eqref{trunc}. In addition, these Hecke correspondences can be composed, and the compositions are associative. We will summarize these structures in Proposition \ref{P:pushforward is local}, which will allow us to define the category $\on{P}^{\Hk}(\Sht^\loc_{\bar k})$ in the next subsection.

It will be convenient to introduce the following terminology.
Given a sequence of dominant coweights $\mmu=(\mu_1,\ldots,\mu_t)$, a quadruple of non-negative integers $(m_1,n_1,m_2,n_2)$ is said to be $\mmu$-\emph{acceptable} if 
\begin{enumerate}
\item
$m_1-m_2=n_1-n_2$ are $\mu_t$-large (or equivalently $\sigma(\mu_t)$-large), 
\item 
$m_2-n_1$ is $(\mu_1, \dots, \mu_{t-1})$-large.
\end{enumerate} 
In particular, $m_1-n_1$ is $\mmu$-large. We regard $(\infty,\infty,\infty,\infty)$ to be $\mmu$-acceptable for any $\mmu$.

The following proposition summarizes a large part of \cite[\S 5.3]{XZ}.
For simplicity, we only consider moduli of restricted local Shtukas whose singularities are bounded by a single coweight.
\begin{proposition}
\label{P:pushforward is local} 
\begin{enumerate}
\item Let $\mu_1$, $\mu_2$ be two dominant coweights, and $\nu$ a dominant coweight.
Let $(m_1,n_1,m_2,n_2)$ be a quadruple that is $(\mu_1+\nu, \nu)$-acceptable and $(\mu_2+\nu,\nu)$-acceptable.
Then there exists the Hecke correspondence
\begin{equation}\label{deompf1}
\Sht_{\mu_1}^{\loc(m_1,n_1)}\xleftarrow{\overleftarrow{h}_{\mu_1}^{\loc(m_1,n_1)}} \Sht_{\mu_1\mid \mu_2}^{\nu,\loc(m_1,n_1)}\xrightarrow{\overrightarrow{h}_{\mu_2}^{\loc(m_2,n_2)}} \Sht_{\mu_2}^{\loc(m_{2},n_{2})}
\end{equation}
where $\Sht_{\mu_1|\mu_2}^{\nu,\loc(m_1,n_1)}$ is an algebraic stack independent of the choice of $(m_1,n_1)$.

Moreover, if $(m'_1,n'_1,m'_2,n'_2)\geq (m_1,n_1,m_2,n_2)$ is another quadruples satisfying the same conditions (the case $(m'_1,n'_1,m'_2,n'_2)=(\infty,\infty,\infty,\infty)$ being allowed).
Then the following diagram is commutative with the left diagram Cartesian
\begin{equation*}
\label{E:cycle construction diagram full}
\xymatrix@C=17pt{
\Sht_{\mu_1}^{\loc(m'_1,n'_1)}\ar_{\res^{m'_1,n'_1}_{m_1,n_1}}[d]
&\Sht_{\mu_1|\mu_2}^{\nu,\loc(m'_1,n'_1)}\ar_{\res}[d]\ar[l]\ar[r]& \Sht_{\mu_2}^{\loc(m'_2,n'_2)}\ar^{\res^{m'_2,n'_2}_{m_2,n_2}}[d]
\\
\Sht_{\mu_1}^{\loc(m_1,n_1)}&\Sht_{\mu_1|\mu_2}^{\nu,\loc(m_1,n_1)}\ar[l]\ar[r]& \Sht_{\mu_2}^{\loc(m_2,n_2)}.
}
\end{equation*}

\item Let $\mu_1,\mu_2,\mu_3,\nu_1,\nu_2$ be dominant coweights and let $(m_1,n_1,m_2,n_2,m_3,n_3)$ be a sextuple such that every quadruple $(m_i,n_i,m_{i+1},n_{i+1})$ is $(\mu_i+\nu_i,\nu_i)$-acceptable and $(\mu_{i+1}+\nu_i,\nu_i)$-acceptable. Then there is a natural perfectly proper morphism
\begin{equation*}
\label{E:composition restricted}
\on{Comp}_{\mu_1, \mu_2, \mu_3}^{\mathrm{loc}(m_1,n_1)}: \Sht_{\mu_1|\mu_2}^{\nu_1,\loc(m_1,n_1)} \times_{\Sht_{\mu_2}^{\loc(m_2,n_2)}} \Sht_{\mu_2|\mu_3}^{\nu_2,\loc(m_2,n_2)} \longto \Sht_{\mu_1|\mu_3}^{\nu_1+\nu_2,\loc (m_1,n_1)}.
\end{equation*}

Moreover, if $(m'_1,n'_1,m'_2,n'_2,m'_3,n'_3)\geq (m_1,n_1,m_2,n_2,m_3,n_3)$ is another sextuple satisfying the same conditions ($(m'_1,n'_1,m'_2,n'_2,m'_3,n'_3)=(\infty,\infty,\infty,\infty,\infty,\infty)$ being allowed), then the following diagram is commutative.
\begin{equation*}
\label{E:projection localization truncated} 
\xymatrix@C=20pt{
&&\Sht_{\mu_1|\mu_3}^{\nu_1+\nu_2,\loc(m'_1,n'_1)} 
\ar[ddd]\ar[dll]\ar[dr]
\\
\Sht^{\loc(m'_1,n'_1)}_{\mu_1} \ar[d]^{\res^{m'_1,n'_1}_{m_1,n_1}} &\ar[l]\ar[d] \Sht_{\mu_1\mid\mu_2}^{\nu_1,\loc(m'_1,n'_1)} \times_{\Sht_{\mu_2}^\loc(m'_2,n'_2)} \Sht_{\mu_2\mid\mu_3}^{\nu_2,\loc(m'_2,n'_2)} \ar[l] \ar[rr] \ar[ur]_-{\mathrm{Comp}^{\loc(m'_1,n'_1)}}&& \Sht^{\loc(m_3,n_3)}_{\mu_3} \ar[d]_{\res^{m'_3,n'_3}_{m_3,n_3}}
\\
\Sht_{\mu_1}^{\loc(m_1,n_1)}
&\ar[l]\Sht_{\mu_1|\mu_2}^{\nu_1,\loc(m_1,n_1)} \times_{\Sht_{\mu_2}^{\loc(m_2,n_2)}} \Sht_{\mu_2|\mu_3}^{\nu_2,\loc(m_2,n_2)} \ar[rd]^-{\mathrm{Comp}^{\loc(m_1,n_1)}} \ar[rr]
&&\Sht_{\mu_3}^{\loc(m_3,n_3)}
\\
&& \Sht_{\mu_1|\mu_3}^{\nu_1+\nu_2, \loc(m_1,n_1)}\ar[ull]\ar[ru].
}
\end{equation*}
In addition, the middle trapezoid is Cartesian.

\item Let $\mu_1, \mu_2, \mu_3, \mu_4, \nu_1, \nu_2, \nu_3$ be dominant coweights and $(m_1, n_1, m_2,n_2, m_3,n_3, m_4,n_4)$ be an octuple of non-negative integers such that $(m_i,n_i,m_{i+1},n_{i+1})$ is $(\mu_i+\nu_i,\nu_i)$-acceptable and $(\mu_{i+1}+\nu_i,\nu_i)$-acceptable. Then
\[\on{Comp}_{\mu_1,\mu_3, \mu_4}^{\mathrm{loc}(m_1,n_1)}\circ (\on{Comp}_{\mu_1,\mu_2, \mu_3}^{\mathrm{loc}(m_1,n_1)} \times \id)=\on{Comp}_{\mu_1,\mu_2, \mu_4}^{\mathrm{loc}(m_1,n_1)}\circ (\id \times \on{Comp}_{\mu_2,\mu_3, \mu_4}^{\mathrm{loc}(m_2,n_2)}).\]
\end{enumerate}
\end{proposition}

\begin{remark}
\label{R: grp theoretical Sht}
We give a group theoretical explanation of the moduli of restricted local $G$-Shtukas.
Given an $(m,n)$-restricted local $G$-Shtuka, by trivialization both $\calE_\rightone|_{D_{m}}$ and $\calE_\leftone|_{D_{n}}$, the isomorphism $\psi$ defines an element in $K_n$.
Note that $\Gr_\mmu^{(n)}$ is exactly the moduli space that classifies trivializations of both $\calE_\rightone|_{D_{m}}$ and $\calE_\leftone|_{D_{n}}$ on $\Hk^{\loc(m)}_\mmu$. It follows that
\[
\Sht_{\mmu}^{\loc(m,n)}\cong \frac{\Gr^{(n)}_{\mmu}\times^{K_n,\sigma} K_n}{\Ad K_m},
\]
where the right hand side denotes the quotient of $\Gr_{\mmu}^{(n)}\times K_n$ by $K_m\times K_n$ with the action given by 
$$(k_m,k_n)\cdot (x,k)=(k_mxk_n^{-1},\sigma(k_n)kk_m^{-1}), \quad (k_m,k_n)\in K_m\times K_n, \ (x,k)\in \Gr_\mmu^{(n)}\times K_n.$$
In particular, the Frobenius automorphism of $\Gr_{\mu}^{(n)}$ then induces an isomorphism
\[\sigma: \Sht_\mu^{\loc(m,n)}\cong \frac{\Gr_{\mu}^{(n)}}{\Ad_\sigma K_m}.\]
(In equal characteristic, these spaces admit canonical deperfection, and  the Frobenius endomorphism of $\Gr_{\mu}^{(n)}$ induces a universal homeomorphism between $\Sht_\mu^{\loc(m,n)}$ and $\frac{\Gr_{\mu}^{(n)}}{\Ad_\sigma K_m}$.)
Therefore, later on when we define the category of perverse sheaves on $\Sht^\loc_{\bar k}$, one can use either $\Sht_{\mmu}^{\loc(m,n)}$ or $\frac{\Gr^{(n)}_{\mmu}\times^{K_n,\sigma} K_n}{\Ad K_m}$. However, although the latter appears to be simpler to describe (group theoretically), the former space is need to relate to Shimura varieties (or moduli of global Shtukas). There is a similar group theoretical description of $\Sht_{\mu_1\mid \mu_2}^{\nu,\loc(m_1,n_1)}$, for which we refer to \cite[Remark 5.3.15]{XZ} for details.
\end{remark}

\subsection{The category $\on{P}^\Hk(\Sht^\loc_{\bar k})$}
After the preparation of the last subsection, it is possible to define the category $\on{D}^\Hk(\Sht^\loc_{\bar k})$, which in light of Lemma \ref{L: fully faithful}, can be thought as the replacement of (a subcategory of) the category of $\Ad_\sigma LG$-equivariant sheaves on $LG$. In fact, we will be mainly focusing on the subcategory $\on{P}^{\Hk}(\Sht^\loc_{\bar k})\subset \on{D}^\Hk(\Sht^\loc_{\bar k})$, which will be the receiver of our categorical trace of the geometric Satake.

First, it is easy to define 
the category $\mathrm{P}(\Sht_{\bar k}^\loc)$ of perverse sheaves on the moduli of local shtukas (base changed to $\bar k$). 
Let $\mmu$ is a sequence of dominant coweights. For two pairs $(m',n'), (m,n)$ of non-negative integers, satisfying $m \leq m', n\leq n'$ and both $m'-n'$ and $m-n$ are $\mmu$-large ($m'\neq \infty$), the functor
$$\Res^{m',n'}_{m,n}=(\res^{m',n'}_{m,n})^![((m'-n')-(m-n))\dim G]:\on{P}(\Sht^{\loc(m,n)}_{\mmu})\to \on{P}(\Sht^{\loc(m',n')}_{\mmu})$$ 
of the shifted $!$-pullback of sheaves is perverse exact. By \eqref{true-to-trun}, there is a canonical isomorphism of functors
\[\Res^{m',n}_{m,n}\circ\Res^{m',n'}_{m',n}\cong \Res^{m',n'}_{m,n}.\]
The functor $\Res^{m',n}_{m,n}$ induces an equivalence of categories if $m\geq 1$. On the other hand, $\res^{m',n'}_{m',n}$ is a torsor under $\Ker\big( \Aut(\calE_\rightone|_{D_{n'}}) \to \Aut(\calE_\rightone|_{D_{n}}) \big)$, which is again the perfection of a unipotent (non-constant) group scheme if $n\geq 1$, therefore, $\Res^{m',n'}_{m',n}$ is fully faithful.

Then we can define the category of perverse sheaves on $\Sht^\loc$  as (a direct sum of) filtered limits
\begin{equation}
\label{E:definition of P(Sloc)}
\mathrm{P}(\Sht_{\bar k}^\loc): =\varinjlim_{(\mu, m,n)} \mathrm{P}(\Sht_\mu^{\loc(m, n)}),
\end{equation}
where the limit is taken over the triples $\{(\mu, m, n) \in \xcoch(T)/W \times \ZZ^2_{\geq 0} \mid m-n \mbox{ is }\mu\mbox{-large}\}$, with the product partial order, and 
where the connecting functor is given by the following composite of fully faithful functors
\[
\mathrm{P}(\Sht_{\mu}^{\loc(m, n)}) 
\xrightarrow{\Res^{m',n'}_{m,n}}
\mathrm{P}(\Sht_{\mu}^{\loc(m', n')})
\xrightarrow{i_{\mu,\mu',*}}
\mathrm{P}(\Sht_{\mu'}^{\loc(m' n')}), 
\]
and where $i_{\mu, \mu'}: \Sht_{\mu}^{\loc(m', n')} \to \Sht_{\mu'}^{\loc(m', n')}$ is the natural closed embedding. 
Note that these connecting functors satisfy natural compatibility conditions given by proper or smooth base change so the limit indeed makes sense.
By replacing $\on{P}(\Sht_{\mu}^{\loc(m, n)})$ by $\on{D}(\Sht_{\mu}^{\loc(m, n)})$ in the above colimit, we also define $\on{D}(\Sht_{\bar k}^\loc)$, which is equipped with a perverse $t$-structure.

\begin{rmk}
\label{R: sheaves on Shtloc}
According to \eqref{R:pre of Sht} and Remark \ref{R: grp theoretical Sht}, we can interpret the above definition in a more group theoretical language. Namely, we write
\[\frac{LG}{\Ad_\sigma K}=\underrightarrow\lim_{\mu}\underleftarrow\lim_n \frac{\Gr_\mu^{(n)}}{\Ad_\sigma K},\]
and therefore define
\[\on{P}(\Sht^\loc_{\bar k})=\underrightarrow\lim_{\mu,m,n}\on{P}(\frac{\Gr^{(n)}_\mu}{\Ad_\sigma K_m}).\]
\end{rmk}

For each dominant coweight $\mu$ and a pair $(m,n)$ such that $m-n$ is $\mu$-large, we have a natural pullback functor
\begin{equation*}
\label{E:functor P(Hk) to P(S)fin}
\Phi^{\loc(m,n)}:=\Res^{m,n}_{m,0}: \on{P}_{K\otimes \bar k}(\Gr_\mu)\cong \on{P}(\Hk_\mu^{\loc(m)})\to \on{P}(\Sht_\mu^{\loc(m,n)}),
\end{equation*}
which commutes the above connecting morphism by \eqref{true-to-trun} and the proper smooth base change. 
Taking the colimit, we obtain a well-defined functor
\begin{equation}
\label{E:functor P(Hk) to P(S)}
\Phi^{\loc}: \mathrm{P}_{K\otimes\bar k}(\Gr\otimes\bar k) \to \mathrm{P}(\Sht_{\bar k}^\loc).
\end{equation}

\begin{rmk}
\label{R: more object}
Note that there are many objects in $\mathrm{P}(\Sht_{\bar k}^\loc)$ that do not come from $\on{P}_{K\otimes \bar k}(\Gr\otimes\bar k)$ under the above functor $\Phi^{\loc}$. For example, as mentioned in Example \ref{Ex:mStau mn}, $\Sht^{\loc(n,n)}_0=\bfB G(\mO/\varpi^n)$. Therefore, every representation $\rho$ of the finite group $G(\mO/\varpi^n)$ defines a local system on $\Sht^{\loc(n,n)}_0$, and therefore an object  in $\mathrm{P}(\Sht_{\bar k}^\loc)$. This object does not lie in the essential image of $\Phi^\loc$ (as soon as $n>0$). More generally, for every $\mu$, and every representation $\rho$ of $G(\mO)$, there is an object $\mF_{\mu,\rho}$ in $\mathrm{P}(\Sht_{\bar k}^\loc)$ coming from the nearby cycle functors construction \cite{GL}. 
\end{rmk}

We now define the category $\rmP^{\Hk}(\Sht^\loc_{\bar k})$. It has the same objects as $\rmP(\Sht_{\bar k}^\loc)$. And we just need to explain the space of morphisms and compositions of them as follows.
An object $\mF$ of $\rmP(\Sht_{\bar k}^\loc)$ (or equivalently of $\rmP^\mathrm{Hk}(\Sht^\loc_{\bar k})$) is called connected if it comes from $\rmP(\Sht_\mu^{\loc(m,n)})$ for some $(\mu,m, n)$ in the limit \eqref{E:definition of P(Sloc)}. In this case, we denote by $\mF_\mu^{(m,n)}\in \rmP(\Sht_\mu^{\loc(m,n)})$ the unique (up to a unique isomorphism) object that represents $\mF$. Since every object of $\rmP(\Sht_{\bar k}^\loc)$ is a direct sum of connected objects, it is enough to define the space  of morphisms between two connected objects $\calF_1, \calF_2$, and to extend the definition of the morphisms between general objects by linearity.

Recall that for a septuple $(\mu_1,\mu_2,\nu,m_1,n_1,m_2,n_2)$ such that  $(m_1,n_1,m_2,n_2)$ is $(\mu_1+\nu,\nu)$-acceptable and $(\mu_2+\nu,\nu)$-acceptable,
there is the correspondence  \eqref{deompf1} from Proposition \ref{P:pushforward is local}. Given another septuple $(\mu'_1,\mu'_2,\nu',m'_1,n'_1,m'_2,n'_2)\geq (\mu_1,\mu_2,\nu,m_1,n_1,m_2,n_2)$ satisfying similar acceptability conditions, there is the following diagram
\begin{equation*}
\label{E: loc Hecke corr changing level}
\begin{CD}
\Sht_{\mu_1}^{\loc(m_1,n_1)}@<\res^{m'_1,n'_1}_{m_1,n_1}<< \Sht_{\mu_1}^{\loc(m'_1,n'_1)}@>i_{\mu_1,\mu'_1}>> \Sht_{\mu'_1}^{\loc(m'_1,n'_1)}\\
@AAA@AAA@AAA\\
\Sht_{\mu_1\mid\mu_2}^{\nu,\loc(m_1,n_1)}@<\res<<\Sht_{\mu_1\mid\mu_2}^{\nu,\loc(m'_1,n'_1)}@>i>>\Sht_{\mu'_1\mid\mu'_2}^{\nu',\loc(m'_1,n'_1)}\\
@VVV@VVV@VVV\\
\Sht_{\mu_2}^{\loc(m_2,n_2)}@<\res^{m'_2,n'_2}_{m_2,n_2}<<\Sht_{\mu_2}^{\loc(m'_2,n'_2)}@>i_{\mu_2,\mu'_2}>>\Sht_{\mu'_2}^{\loc(m'_2,n'_2)}
\end{CD}
\end{equation*}
where arrows to the left are equidimensional, perfectly smooth, of the same relative dimension, and arrows to the right are closed embeddings. Note that the upper left square is Cartesian by Proposition \ref{P:pushforward is local} and all arrows in the upper right square are perfectly proper.

Now, for two connected objects $\mF_1,\mF_2$ of $\on{P}^\Hk(\Sht^\loc_{\bar k})$, we define $\mathrm{Mor}_{\rmP^\Hk(\Sht^\loc_{\bar k})} (\calF_1, \calF_2)$ as
\begin{equation*}
\label{E: hom space of PcorrSht}
\varinjlim_{(\mu_1,\mu_2,\nu,m_1,n_1,m_2,n_2)} \mathrm{Corr}_{\Sht_{\mu_1|\mu_2}^{\nu, \loc(m_1,n_1)}} \big( (\Sht_{\mu_1}^{\loc(m_1,n_1)}, \calF_{1,\mu_1}^{(m_1,n_1)}),\,( \Sht_{\mu_2}^{\loc(m_2,n_2)},\calF_{2,\mu_2}^{(m_2,n_2)}) \big),
\end{equation*}
where the colimit is taken over all (partially ordered) sextuples $(\mu_1,\mu_2,\nu,m_1,n_1,m_2,n_2)$ satisfying the acceptability conditions as above and such that $\mF_{i,\mu_i}^{(m_i,n_i)}$ represents $\mF_i$, and where the connecting map of colimit is given by $i_!\circ\res^![((m'-n')-(m-n))\dim G]$, where $i_!$ denotes the pushforward of cohomological correspondence along $i$ as in \eqref{ASS:pushforward correspondence} and where $\res^!$ is the pullback of cohomological correspondences as in \eqref{ASS:smooth pullback correspondence2}. As explained in \cite[\S 5.4.1]{XZ}, these connecting maps compose and the colimit is well-defined.

Next, we explain the composition of morphisms. By linearity, it suffices to define this on the connected objects and we may assume that the given morphisms $c_1: \calF_1 \to \calF_2$ and $c_2: \calF_2 \to \calF_3$ are realized as correspondences (as opposed to linear combinations of correspondences)
\begin{align*}
c^{(\mu_1,\mu_2,\nu_1,m_1,n_1,m_2,n_2)}_1 \in \mathrm{Corr}_{\Sht_{\mu_1|\mu_2}^{\nu_1,\loc(m_1,n_1)}}\big( (\Sht_{\mu_1}^{\loc(m_1, n_1)}, \calF_{1,\mu_1}^{(m_1,n_1)}),\ (\Sht_{\mu_2}^{\loc(m_2, n_2)}, \calF_{2,\mu_2}^{(m_2,n_2)}) \big)&, \quad\textrm{and}\\
c^{(\mu_2,\mu_3,\nu_2,m_2,n_2,m_3,n_3)}_2 \in \mathrm{Corr}_{\Sht_{\mu_2|\mu_3}^{\nu_2,\loc(m_2,n_2)}}\big( (\Sht_{\mu_2}^{\loc(m_2, n_2)},\ \calF_{2,\mu_2}^{(m_2,n_2)}) \,\ (\Sht_{\mu_3}^{\loc(m_3, n_3)}, \calF_{3,\mu_3}^{(m_3,n_3)})\big)&,
\end{align*}
where the septuples satisfy the conditions as above.
According to the formalism of the composition of cohomological correspondences (see Definition~\ref{AD:correspondences}), the composition $c^{(\mu_2,\mu_3,\nu_2,m_2,n_2,m_3,n_3)}_2 \circ 
c^{(\mu_1,\mu_2,\nu_1,m_1,n_1,m_2,n_2)}_1$ is supported on the following correspondence
\[
\Sht_{\mu_1}^{\loc(m_1,n_1)} \leftarrow \Sht_{\mu_1|\mu_2}^{\nu_1,\loc(m_1, n_1)} \times_{ \Sht_{\mu_2}^{\loc(m_2,n_2)}}
\Sht_{\mu_2|\mu_3}^{\nu_2,\loc(m_2, n_2)}  \to  \Sht_{\mu_3}^{\loc(m_3,n_3)}.
\]

We then define the composition $c_2\circ c_1$ to be the cohomological correspondence from $\calF_1$ to $\calF_3$ so that $(c_2\circ c_1)^{(\mu_1,\mu_3,\nu_1+\nu_2,m_1,n_1,m_3,n_3)}$ is given by the push forward of $c^{(\mu_2,\mu_3,\nu_2,m_2,n_2,m_3,n_3)}_2 \circ 
c^{(\mu_1,\mu_2,\nu_1,m_1,n_1,m_2,n_2)}_1$ along the perfectly proper morphism
\[
\Sht_{\mu_1|\mu_2}^{\nu_1,\loc(m_1, n_1)} \times_{ \Sht_{\mu_2}^{\loc(m_2,n_2)}}
\Sht_{\mu_2|\mu_3}^{\nu_2,\loc(m_2, n_2)} \xrightarrow{\mathrm{Comp}^{\loc(m_1,n_1)}}\Sht_{\mu_1|\mu_3}^{\nu_1+\nu_2,\loc(m_1, n_1)},
\]
in the sense of \eqref{ASS:pushforward correspondence}. Using Proposition \ref{P:pushforward is local} and several lemmas in \S \ref{Sec:cohomological correspondence}, one can show that $c_2\circ c_1$ is well-defined, and satisfies the natural associativity law $(c_3\circ c_2) \circ c_1 \cong c_3 \circ (c_2 \circ c_1)$. The construction of $\on{P}^{\Hk}(\Sht^\loc_{\bar k})$ now is complete.

We describe the endomorphism ring of some objects in $\on{P}^{\on{Hk}}(\Sht^\loc_{\bar k})$. 
For every $n>0$, let $\on{reg}_n$ denote the regular representation of $G(\mO/\varpi^n)$. It defines a local system
$\delta_{\on{reg}_n}^{(n,n)}$ on $\on{Sht}^{\loc(n,n)}_{0}=\bfB G(\mO/\varpi^n)$ (see Remark \ref{R: more object}). Let $\delta_{\on{reg}_n}$ be the corresponding object in $\on{P}^{\Hk}(\Sht^\loc_{\bar k})$.
Note that $\delta_{\on{reg}_0}$ is contained in the image of the functor \eqref{E:functor P(Hk) to P(S)}. We also denote it by $\widetilde{\delta_{e}}$.
The following statement is the affine analogue of \eqref{E:endo for unit} and \eqref{E:endo for reg}.
\begin{prop}
\label{P: endo of unit}
There is a canonical isomorphism
$$\End_{\on{P}^{\Hk}(\Sht^\loc_{\bar k})}(\delta_{\on{reg}_n})\cong C_c(K(n)\backslash G(F)/K(n),E)^{\on{op}},$$ where $K(n)=K^{(n)}(k)$ is the $n$th principal congruence subgroup, and  $ C_c(K(n)\backslash G(F)/K(n),E)$ denotes the corresponding Hecke algebra.
In particular,  $\End_{\on{P}^{\Hk}(\Sht^\loc_{\bar k})}(\widetilde{\delta_{e}})$ is isomorphic to the $E$-valued spherical Hecke algebra $H_G\otimes E=C_c(G(\mO)\backslash G(F)/G(\mO),E)$.
\end{prop}

\begin{rmk}
Exactly the same definition gives $\on{D}^\Hk(\Sht^{\loc}_{\bar k})$. In addition,
if $\on{Iw}\subset K$ is an Iwahori subgroup, then one can similarly define the moduli of local Shtukas with Iwahori level structure, denoted by $\Sht^{\loc}_{\on{Iw}}$. Similarly, one can define the Hecke correspondence $\Hk(\Sht^{\loc}_{\on{Iw}})$, and the category $\on{D}^\Hk(\Sht^{\loc}_{\on{Iw},\bar k})$. As mentioned in Remark \ref{R:springer}, if $\cha E=0$, up to idempotent completion it is equivalent to $\on{D}^\Hk(\Sht^{\loc}_{\bar k})$.
\end{rmk}

\subsection{The categorical trace of the geometric Satake}
Now we are ready to state one of the key ingredients for our arithmetic applications in the next section. Let $\hat G$ be the Langlands dual group of $G$, defined over $\Ql$, as constructed as a Tannakian group via the geometric Satake (Theorem \ref{intro:geomSat} and Remark \ref{R: birth of the dual group}). It is equipped with an action of the geometric $q$-Frobenius $\sigma$ of $k$.

Recall from Example \ref{E: trace of rep} that the pullback functor $\on{Rep}(\hat G)\to \Coh(\frac{\hat G}{c_\sigma\hat G})$ identifies $\Coh_{fr}^{\hat G}(\hat G\sigma)$ with the $\sigma$-twisted categorical trace of $\on{Rep}(\hat G)$, where
$\Coh_{fr}^{\hat G}(\hat G\sigma)$ is the full subcategory of $\Coh(\frac{\hat G}{c_\sigma\hat G})$ spanned by those $\widetilde{V}$ (the pullback of algebraic representations $V$ of $G$). 

Given all the preparations, the following theorem now becomes natural.
\begin{thm}
\label{T: trace of geom Sat}
There is a functor
\[S: \Coh_{fr}^{\hat G}(\hat G\sigma)\to \on{P}^{\Hk}(\Sht^\loc_{\bar k})\]
making the following diagram commutative
\[
\begin{CD}
\on{Rep}(\hat G)@>\cong>>\on{P}_{K\otimes \bar k}(\Gr\otimes \bar k)\\
@VVV@VVV\\
\Coh_{fr}^{\hat G}(\hat G\sigma)@>>> \on{P}^{\Hk}(\Sht^\loc_{\bar k})
\end{CD}
\]
In particular, $S(\mO)=\widetilde{\delta_e}$, where $\mO$ denotes the structure sheaf of $\frac{\hat G}{c_\sigma \hat G}$.
 
In addition, the map
\[S: \mathbf J=\Ql[\hat G\sigma]^{\hat G}=\Gamma([\hat G\sigma/\hat G],\mO)\to \End(S(\mO))\cong H_G\otimes \Ql\]
coincides with the Satake isomorphism \eqref{Sat isom}.
\end{thm}

\begin{rmk}
The functor $S$ should be fully faithful. This should follow from the study of irreducible components of certain affine Deligne-Lusztig varieties as in \cite[\S 4]{XZ}. It would be better to have a conceptual proof of this fact.
\end{rmk}

\begin{rmk}
\label{R:affine springer}
Most discussions in this section also apply to the setting
$X=\frac{LG}{\Ad K}$ and $Y=\frac{LG}{\Ad LG}$, where $K\subset LG$ acts on $LG$ by usual conjugation. E.g. we have the Hecke correspondence
\[\frac{LG}{\Ad K}\xleftarrow{\overleftarrow{h}}\Hk:= X\times_YX\xrightarrow{\overrightarrow{h}}\frac{LG}{\Ad K},\]
and the fibers of $\overleftarrow{h}: \overrightarrow{h}^{-1}(\frac{\Gr_\mu^{(\infty)}}{\Ad K})\to \frac{LG}{\Ad K}$ are so-called generalized (or group version) affine Springer fibers (compare to Remark \ref{R:fiber of hloc}).  
If $m-n$ is $\mu$-large, we have the space
$\frac{\Gr_\mu^{(n)}}{\Ad K_m}$, and similarly the Hecke stack between them. One can define $\on{D}^{\Hk}(\frac{LG}{\Ad K})$ similarly. Theorem \ref{T: trace of geom Sat} has a counterpart in this case, giving the affine Springer action.
\end{rmk}

\section{Applications to Shimura varieties}
\label{S:app to Sh}
We will construct cohomological correspondences between mod $p$ fibers of \emph{different} Shimura varieties, following \cite[\S 7]{XZ}. We have to assume in this section some familiarity with the basic theory of mod $p$ geometry of Shimura varieties. Readers can refer to \cite[\S 7.1, \S 7.2]{XZ} for some necessary background needed in the following discussions.

\subsection{Cohomological correspondences between mod $p$ fibers of Shimura varieties}

Let $(G,X)$ be a (weak) Shimura datum. We fix an open compact subgroup $K\subset G(\bA_f)$, and let
\[\Sh_K(G,X)=G(\bQ)\backslash X\times G(\bA_f)/K\]
denote the corresponding Shimura variety.  The Shimura datum defines a conjugacy class of one-parameter subgroups $\mu$ of $G_\bC$, usually called the Shimura cocharacter.
Let $(\hat G,\hat B,\hat T)$ be the Langlands dual group of $G$ (see Remark \ref{R: birth of the dual group}). Then $\mu$ determines an irreducible representation of $\hat G$ of highest weight $\mu$.
Recall that $\Sh_K(G,X)$ has a canonical model defined over the reflex field $E=E(G,X)\subset \bC$. The representation $V_\mu$ extends canonically to a representation of $\hat G\rtimes \Gal(\overline\bQ/E)$.
If  $(G,X)$ is of Hodge type, one can interpret $\Sh_K(G,X)$ as moduli of abelian varieties equipped with additional structures.

Now let $p>2$ be an unramified prime for $(G,X,K)$, by which we mean that $K_p\subset G(\bQ_p)$ is a hyperspecial open compact subgroup. 
In addition, we assume that $(G,X)$ is of Hodge type. Under these assumptions, $\Sh_K(G,X)$ admits a canonical integral model $\mathscr S_K(G,X)$ over $\mO_{E,(v)}$ by the work of Kisin (\cite{Ki}) and Vasiu (\cite{Vasiu}), where $v$ is a place of $E$ over $p$ with $\FF_v$ its residue field. We denote by $\overline{\mathscr S}$ its mod $p$ fiber, and $\Sh_\mu$ the perfection of $\overline{\mathscr S}$. Kisin's construction also gives rise to a $G$-Shtuka $\mE_1\dashrightarrow \mE_0\cong {^\sigma}\mE_1$ with singularities bounded by $\mu$ over $\Sh_\mu$. Namely,
\[\mE_0:=\mE_{\on{cris}},\]
where $\mE_{\on{cris}}$ is the usual crystalline $G$-torsor on $\Sh_\mu$ (\cite[Corollary 7.1.14]{XZ}).
We thus obtain a morphism of prestacks
\[\loc_p: \Sh_{\mu}\to \Sht_{\mu}^\loc.\]
In the Siegel case, $\loc_p$ is the perfection of the morphism sending an abelian variety to its underlying $p$-divisible group. 

According to the Serre-Tate theory, this morphism (before perfection) is formally \'etale. However, formal \'etaleness is not a good notion for morphisms between perfect schemes, so we will not make use of it. In addition, $\Sht_\mu^\loc$ is not an algebraic stack so it is difficult to work with it directly.  Instead, we will compose it with the morphism from $\Sht_{\mu}^\loc$ to the moduli of restricted local Shtukas. Let $(m,n)$ be a pair of non-negative integers such that $m-n$ is $\mu$-large. We set
\[\loc_p(m,n):\Sh_{\mu}\xrightarrow{\loc_p} \Sht_{\mu}^\loc\xrightarrow{\res_{m,n}} \Sht_{\mu}^{\loc(m,n)}.\]

The following result (\cite[Proposition 7.2.4]{XZ}) allows us to study the special fibers of Shimura varieties via the local Shtukas defined earlier.

\begin{prop}\label{smoothness}
The morphism $\loc_p(m,n)$ is perfectly smooth.
\end{prop}
\begin{rmk}
By virtual of Remark \ref{Ex:Sht and GZip}, this proposition refines the well-known perfectly smooth morphism $\Sh_\mu\to G\on{-zip}^\pf_\mu$ (e.g. see \cite{Zha}), which can be used to define the Ekedahl-Oort stratification on Shimura varieties.
\end{rmk}

Now we discuss cohomological correspondences between mod $p$ fibers of Shimura varieties.

Let $(G_1,X_1)$ and $(G_2,X_2)$ be two Shimura data. 
Let $\theta: G_1\otimes\bA_f\simeq G_2\otimes\bA_f$ be an isomorphism. We fix an open compact subgroup $K_i\subset G_i(\bA_f)$ which is sufficiently small such that $\theta K_1=K_2$, and
we assume that $K_{1,p}$ (and therefore $K_{2,p}$) is hyperspecial. 
Let $\underline G_i$ be the integral model of $G_{i,\bQ_p}$ over $\bZ_{p}$ determined by $K_{i,p}$. We have the isomorphism $\theta: \underline G_1\simeq \underline G_2$, which allows us to identify their Langlands dual group $(\hat G,\hat B,\hat T)$, on which the Frobenius $\sigma_p$ acts by an automorphism. We fix an isomorphism $\iota: \bC\simeq \overline\bQ_p$.
Let $\{\mu_i\}$ denote the conjugacy class of Hodge cocharacters of $G_{i\bC}$, determined by $X_i$, which via the isomorphism $\iota$, can be regarded as a dominant character of $(\hat G,\hat B,\hat T)$. 
Let $V_{\mu_1}$ and $V_{\mu_2}$ denote the corresponding highest weight representations of $\hat G$, and let $\widetilde{V_{\mu}}$ and $\widetilde{V_{\mu'}}$ denote the corresponding vector bundles on $\frac{\hat G}{c_{\sigma_p}\hat G}$, as in Example \ref{E: trace of rep}.

Let $E_i\subset \bC$ be the reflex field of $(G_i,X_i)$. Let $v$ be a place of $E=E_1E_2$ lying over $p$, determined by the isomorphism $\iota$. 
Let $\Sh_{\mu_i,K_i}$ denote the mod $p$ fiber of the canonical integral model for $\Sh_{K_i}(G_i,X_i)$, base changed to $k_{v}$. 
We assume that there exists a perfect ind-scheme $\Sh_{\mu_1|\mu_2}$ fitting into the following commutative diagram, with both squares Cartesian,
\begin{equation}
\label{E:global to local cartesian1}
\xymatrix{
\Sh_{\mu_1,K_1} \ar_{\loc_p}[d] & \ar_{\overleftarrow{h}_{\mu_1}}[l] \Sh_{\mu_1|\mu_2} \ar^{\overrightarrow{h}_{\mu_2}}[r] \ar[d] & \Sh_{\mu_2,K_2} \ar^{\loc_p}[d] \\
\Sht_{\mu_1}^\loc & \ar_{\overleftarrow{h}_{\mu_1}^\loc}[l] \Sht_{\mu_1|\mu_2}^{\loc} \ar^{\overrightarrow{h}_{\mu_2}^\loc}[r]& \Sht_{\mu_2}^\loc.
}
\end{equation}
Note that nonemtpiness of $\Sht_{\mu_1\mid\mu_2}^\loc$ is equivalent to the following condition
\begin{equation}
\label{E: exotic p}
\Hom_{\Coh^{\hat G}_{fr}(\hat G\sigma_p)}(\widetilde{V_{\mu_1}},\widetilde{V_{\mu_2}})\neq 0.
\end{equation}
But even $\Sht_{\mu_1\mid\mu_2}^\loc$ is non-empty, there are some restrictions on $(G_i,X_i)$ to guarantee the existence of \eqref{E:global to local cartesian1}.  Namely, one can define $\Sh_{\mu_1\mid\mu_2}$ making the left square in \eqref{E:global to local cartesian1} Cartesian. Then $\Sh_{\mu_1\mid\mu_2}$ is represented by an ind-perfect scheme since $\overleftarrow{h}^\loc_{\mu_1}$ is ind-proper (see Remark \ref{R:fiber of hloc}). But it is in general not true that there exists a morphism from $\Sh_{\mu_1\mid\mu_2}$ to $\Sh_{\mu_2}$ making the right square commutative (and also Cartesian). We do not discuss the general conditions needed for the existence of \eqref{E:global to local cartesian1} here. Instead, we mention the following two extremal cases and refer to \cite{XZ} for more discussions.

\begin{thm}
\label{T: corr. Shimura}
The diagram \eqref{E:global to local cartesian1} exists in the following cases. Assume that both $(G_i,X_i)$ are of Hodge type and $p>2$.
\begin{enumerate}
\item
When $(G_1,X_1)=(G_2,X_2)$ and $\theta$ is the identity map, $\Sh_{\mu_1\mid \mu_2}$ exists and is the perfection of the mod $p$ fiber of a natural integral model of the $p$-power Hecke correspondences of $\Sh_K(G_1,X_1)$.
\item
When there is an inner twist $\Psi: G_1\to G_2$ which is equal to $\Ad_h \circ \theta$ for some $h\in G_2(\overline{\bA}_f)$ at all finite places, and which realizes $G_{2\bR}$ as the compact modulo center inner form of $G_{1\bR}$ (so in particular $\Sh_{\mu_2}$ is a finite set, usually called a Shimura set), then $\Sh_{\mu_1\mid \mu_2}$ exists and can be regarded as the Rapoport-Zink uniformization of the basic locus of $\Sh_{\mu_1}$.
\end{enumerate}
\end{thm}

\begin{rmk}
As just mentioned, if $(G,X)=(G_1,X_1)=(G_2,X_2)$, $\Sh_{\mu\mid\mu}$ is the mod $p$ fiber of a natural integral model of the Hecke correspondence
\begin{equation*}
\label{E: pHecke corr}
\xymatrix@C=-20pt{
&G(\bQ)\backslash X\times G(\bA_f^p)/K^p\times G(\bQ_p)\times^{K_p}G(\bQ_p)/K_p\ar^{\overleftarrow{h}}[ld]\ar_{\overrightarrow{h}}[rd]&\\
G(\bQ)\backslash X\times G(\bA_f)/K && G(\bQ)\backslash X\times G(\bA_f)/K.}
\end{equation*}
If $(G_1,X_1)\neq (G_2,X_2)$ but \eqref{E: exotic p} holds, then $\Sh_{\mu_1\mid \mu_2}$ can be regarded as ``exotic Hecke correspondences" between mod $p$ fibers of \emph{different} Shimura varieties. These correspondences cannot be lifted to characteristic zero, and give a large class of characteristic $p$ cycles on Shimura varieties. (See the next subsection.) 
\end{rmk}

Combining \eqref{E:global to local cartesian1} with Proposition \ref{P:pushforward is local}, we have the following commutative diagram with the left square Cartesian
\[
\xymatrix{
\Sh_{\mu_1,K_1} \ar_{\loc_p(m_1,n_1)}[d] && \ar_{\overleftarrow{h}_{\mu_1}}[ll] \Sh_{\mu_1|\mu_2} \ar^{\overrightarrow{h}_{\mu_2}}[rr] \ar[d] && \Sh_{\mu_2,K_2} \ar^{\loc_p}[d] \\
\Sht_{\mu_1}^{\loc(m_1,n_1)} && \ar_{\overleftarrow{h}_{\mu_1}^{\loc(m_1,n_1)}}[ll] \Sht_{\mu_1|\mu_2}^{\loc(m_1.n_1)} \ar[rr]&& \Sht_{\mu_2}^{\loc(m_2,n_2)}.
}
\]

Now, by pulling back morphisms between $(\Sht_{\mu_1}^{\loc(m_1,n_1)},\mF_1)$ and $(\Sht_{\mu_1}^{\loc(m_2,n_2)},\mF_2)$ in $\on{P}^{\Hk}(\Sht_{\bar k}^\loc)$ along vertical maps in the above diagram in the sense of \eqref{ASS:smooth pullback correspondence2}, we obtain cohomological correspondences between $(\Sh_{\mu_1}, \loc_p(m_1,n_1)^!\mF_1)$ and $(\Sh_{\mu_2},\loc_p(m_2,n_2)^!\mF_2)$. Taking the cohomology with compact support and applying Lemma \ref{AL:pushforward compatible with composition}, we obtain the following theorem.  
\begin{thm}
\label{introT: spectral action}
\begin{enumerate}
\item
Let $(G_i,X_i), i=1,2,3$ be a collection of Shimura data, together with the isomorphism $\theta_{ij}:G_i(\bA_f)\cong G_{j}(\bA_f)$ satisfying the natural cocycle condition. We fix a common level structure $K$. Let $p>2$ be an unramified prime and fix an isomorphism $\iota: \bC\simeq \overline\bQ_p$.
In addition, assume that for each pair $(G_i,X_i),(G_j,X_j)$, the Cartesian diagram \eqref{E:global to local cartesian1} exists. Let $\{\mu_i\}$ denote the conjugacy class of Shimura cocharacters of $(G_i,X_i)$, $V_i=V_{\mu_i}$ the corresponding highest weight representation of $\hat G$, and $\widetilde{V_i}$ the corresponding vector bundle on $\frac{\hat G}{c_{\sigma_p}\hat G}$.  Denote $d_i=\dim \Sh_K(G_i,X_i)$. Let $\mH^p$ denote the prime-to-$p$ Hecke algebra.
Let $v\mid p$ be the place of the compositum of all reflex fields determined by $\iota$.
Then there is a natural action of $\mathbf J$ on $\on{H}^*_c(\Sh_{\mu_i,\overline\bF_v},\Ql(d_i/2))$, and a
canonical $\mathbf J$-equivariant map 
\[S_{\theta_{ij}}:\Hom(\widetilde{V_i},\widetilde{V_j})\to \Hom_{\mathbf J\otimes\mH^p}\big(\on{H}_c^{*+d_i}(\Sh_{\mu_i,\overline\bF_v},\Ql(d_i/2)),\on{H}_c^{*+d_j}(\Sh_{\mu_j,\overline\bF_v},\Ql(d_j/2))\big),\] 
which is compatible with the natural compositions on both sides. In particular, there is a natural action of the algebra $\End(\widetilde{V_{\mu}})$ on $\on{H}_c^{*}(\Sh_{\mu, \overline \FF_v},\Ql)$.

\item When $\Sh_{\mu,\overline \FF_v}$ is a Shimura set, the action of $\mathbf J=\End(\widetilde{V_{\mu}})$ on 
$$\on{H}_c^{*}(\Sh_{\mu, \overline \FF_v},\Ql)\cong C_c(G(\bQ)\backslash G(\bA_f)/K, \Ql)$$ 
coincides with the usual Hecke algebra action under the Satake isomorphism \eqref{Sat isom}.
\end{enumerate}
\end{thm}

\begin{rmk}
(1) The above theorem in particular gives some Jacquet-Langlands transfer in geometric way, as first studied by David Helm (\cite{He}) in some special cases. 

(2) The action of $\mathbf J\subset \End(\widetilde{V_{\mu}})$ on $\on{H}_c^{*}(\Sh_{\mu},\Ql)$ in the theorem is the Shimura variety analogue of V. Lafforgue's $S$-operators. However, the action of the larger algebra  $\End(\widetilde{V_{\mu}})$ on $\on{H}_c^{*}(\Sh_{\mu},\Ql)$ is new, even in the function field case (where instead of Shimura varieties one considers the moduli spaces of shtukas). 
The following conjecture is the analogue of V. Lafforgue's $S=T$ theorem.
\begin{conjecture}
\label{introconj: S=T}
Under the Satake isomorphism,
the action of $\mathbf J$ in the above theorem coincides with the usual Hecke algebra action on $\on{H}_c^{*}(\Sh_{\mu, \overline \FF_v},\Ql)$ under the Satake isomorphism \eqref{Sat isom}. 
\end{conjecture}

(3) By \cite[\S 6.3]{XZ2}, the above $S=T$ conjecture implies the congruence relation conjecture (known as the Blasius-Rogawski conjecture) for Shimura varieties of Hodge type.
 \end{rmk}

\begin{rmk}
\label{R: coh of Shimura}
The following is a (heuristic) conjectural description of the cohomology of $\overline{\scrS}$, inspired by Drinfeld's interpretation of V. Lafforgue's work:
There exists a graded $\Ql$-vector space $A^*$ (depending only on $G\otimes\bA_f$), equipped with an $\mH_K$-action, and a commuting action of $\Ql[\hat{G}\sigma_p]$, such that 
$$\on{H}_c^{*+d}(\overline{\scrS}_{\overline\bF_v}, \Ql(d/2)) = (A^*\otimes V_{\mu})^{\hat G}.$$
It is clear that Theorem \ref{introT: spectral action} is implied by this conjecture. In fact, there is similar conjectural description of the cohomology of $\Sh_K(G,X)$, which we plan to discuss in details in another ocassion. 
\end{rmk}

\subsection{Tate cycles on some mod $p$ fibers of Shimura varieties}
We apply the above machinery to verify ``generic" cases of Tate conjecture for the mod $p$ fibers of many Shimura varieties.

Let $(G,X,K)$ be as in the previous subsection and let $p>2$ be an unramified prime. 
Let $(\hat G, \hat B,\hat T)$ be the Langlands dual group of $G_{\bQ_p}$, equipped with a Borel and a maximal torus as in Remark \ref{R: birth of the dual group}. We fix an isomorphism $\iota: \overline\bQ_p\simeq \bC$ as before, and regard the Shimura cocharacter $\mu$ as a dominant weight of $(\hat G,\hat B,\hat T)$. Let $v\mid p$ be the place of $E$ determined by $\iota$ and let $\overline{\scrS}$ be the mod $p$ fiber of the canonical integral model of $\Sh_K(G,X)$ over $\mO_{E,(v)}$ as before.

For $\la\in\xch(\hat T)$ and a representation $V$ of $\hat G$, let $V(\la)$ denote the $\la$-weight subspace of $V$ (with respect to $\hat T$).
We define the following lattice
\[\Lambda^\Tate_p=\Big\{\la\in \xch(\hat T)\; \Big|\; \sum_{i=0}^{m-1} \sigma_p^i(\la)\in \xcoch(Z_G)\Big\}\subset \xch(\hat T).\]
Here, $Z_G$ denotes the center of $G$, and $\xcoch(Z_G)$ denotes the coweight lattice of $Z_G$, which is a natural subgroup of $\xch(\hat T)$. For a representation $V$ of $\hat G$, we define the following space
\[V^\Tate=\bigoplus_{\la\in\Lambda^\Tate_p} V(\la).\]

We are in particular interested in the condition $V^\Tate_\mu\neq 0$, where $\mu$ is the Shimura cocharacter of $G$.
As explained in the introduction of \cite{XZ}, under the conjectural description of the cohomology of the Shimura varieties given in Remark \ref{R: coh of Shimura}, for a Hecke module $\pi_f$ whose Satake parameter at $p$ is generic enough, certain multiple $a(\pi_f)$ of
the dimension of this vector space should be equal to the dimension of the space of Tate classes in the $\pi_f$-component of the middle dimensional compactly-supported cohomology of $\overline{\scrS}$. 
In addition, this space is usually large. For example, in the case $G$ is an odd unitary (similitude) group of signature $(i,n-i)$ over a quadratic imaginary field, the dimension of this space at an inert prime is $\begin{pmatrix}\frac{n+1}{2}\\ i\end{pmatrix}$.

We need one more notation. For a (not necessarily irreducible) algebraic variety $Z$ of dimension $d$ over an algebraically closed field, let ${\rm H}^{\rm BM}_{2d}(Z)(-d)$ denote the $(-d)$-Tate twist of the top degree Borel--Moore homology, which is the vector space spanned by the irreducible components of $Z$. Now let $X$ be
a smooth variety of dimension $d+r$ defined over a finite field $\FF_q$, and let $Z \subseteq X_{\overline \FF_q}$ be a (not necessarily irreducible) projective subvariety of dimension $d$. There is the cycle class map
\[\on{cl}: {\rm H}^{\rm BM}_{2d}(Z)(-d)\to \bigcup_{j\geq 1}{\rm H}_{\et,c}^{2r}(X_{\overline \FF_q}, \bQ_\ell(r))^{\sigma_q^j}.\]

Our main theorem of \cite{XZ} is as follows. There is a natural Newton stratification on $\overline{\mathscr S}$. 
Let $\overline{\mathscr S}_b$ denote the basic Newton stratum, on which $\mH_K$ acts by (cohomological) correspondences (given by $S$-operators constructed in the previous subsection.)
\begin{thm}
\label{T:main theorem}
Assume that $(G,X)$ is of Hodge type such that the center $Z_G$ of $G$ is connected. Let $K\subset G(\bA_f)$ be an open compact subgroup.  Write $2d=\dim\Sh_K(G,X)$.
Let $p>2$ be an unramified prime for $(G,X,K)$ such that $r:=V_{\mu^*}^\Tate\neq 0$. 
Then:
\begin{enumerate}
\item The basic Newton stratum $\overline{\mathscr S}_b$ of $\overline{\scrS}$ is pure of dimension $d$. In particular, $d$ is always an integer. In addition,
there is an $\mH_K$-equivariant isomorphism
$${\rm H}^{\rm BM}_{2d}(\overline{\mathscr S}_{b, \overline \FF_v})(-d)\cong C(G'(\bQ)\backslash G'(\bA_f)/K)^{\oplus r}.$$
Here $G'$ is an inner form of $G$ that is isomorphic to $G$ over $\bA_f$ and is compact modulo center at infinity, and $C(G'(\bQ)\backslash G'(\bA_f)/K)$ denotes the space of functions on the finite set $G'(\bQ)\backslash G'(\bA_f)/K$.

\item 

Let $\pi_f$ be an irreducible module of $\calH_K$, and let 
$${\rm H}^{\rm BM}_{2d}(\overline{\mathscr S}_{b, \overline \FF_v})[\pi_f]=\Hom_{\mH_K}(\pi_f,{\rm H}^{\rm BM}_{2d}(\overline{\mathscr S}_{b, \overline \FF_v})(-d))\otimes \pi_f$$ be the $\pi_f$-isotypical component. Then the cycle class map
\[\on{cl}:{\rm H}^{\rm BM}_{2d}(\overline{\mathscr S}_{b, \overline \FF_v})(-d)\to {\rm H}^{2d}_{\et,c}\big(\overline{\mathscr S}_{\overline \FF_v},\overline\bQ_\ell(d)\big)\]
restricted to ${\rm H}^{\rm BM}_{2d}(\overline{\mathscr S}_{b, \overline \FF_v})[\pi_f]$ is injective if the Satake parameter of $\pi_{f,p}$ (the component of $\pi_f$ at $p$) is $V_\mu$-general.

\item Assume that $\Sh_K(G,X)$ is a quaternionic Shimura variety or a Kottwitz arithmetic variety. Then the $\pi_f$-isotypical component of $\on{cl}$ is surjective to $T^d(\pi,\Ql)\otimes\pi_f$ if the Satake parameter of $\pi_{f,p}$ is strongly $V_\mu$-general. In particular, the Tate conjecture holds for these $\pi_f$.
\end{enumerate}
\end{thm}
\begin{rmk}
(1) The assumption that $Z_G$ is connected is not essential, but simplifies the formulation.

(2) For a representation $V$ of $\hat G$, the definitions of ``$V$-general" and ``strongly $V$-general" Satake parameters were given in \cite[Definition 1.4.2]{XZ}. Regular semisimple elements in $\hat G\sigma_p$ are always $V$-general, but not the converse. See \cite[Remark 1.4.3]{XZ}.

(3) Some special cases of the theorem were originally proved in \cite{HTX,TX}.

\end{rmk}

The proof of this theorem relies on several different ingredients. Via the Rapoport-Zink uniformization of the basic locus of a Shimura variety, Part (1) can be reduced a question about irreducible components of certain affine Deligne-Lusztig varieties, which was studied in \cite[\S 4]{XZ}.
The most difficult part is (2), which we proved by calculating the intersection numbers among all $d$-dimensional cycles in $\overline\scrS_b$. By Part (1), the intersection numbers of $d$-dimensional cycles in $\overline{\scrS}_b$ can be encoded in an $r\times r$-matrix with entries in the spherical Hecke algebra at $p$. In general, it seems hopeless to calculate this matrix directly and explicitly. 
However, similar to the toy model explained in Example \ref{intersection by geom Sat}, this matrix can be understood as the composition of certain morphisms in $\Coh^{\hat G}_{fr}(\hat G\sigma_p)$. Namely, first we realize $G'(\bQ)\backslash G'(\bA)/K$ as a Shimura set with $\tau$ its Shimura cocharacter\footnote{There is a subtlety regarding the choice of $\tau$ which we ignore here. See \cite[Remark 7.4.4]{XZ}.}.
Then using Theorem \ref{introT: spectral action}, this matrix can be calculated as
\[ \Hom_{\Coh^{\hat G}_{fr}(\hat G\sigma_p)}(\widetilde{V_{\tau}},\widetilde{V_{\mu}})\otimes  \Hom_{\Coh^{\hat G}_{fr}(\hat G\sigma_p)}(\widetilde{V_{\mu}},\widetilde{V_{\tau}})\to  \Hom_{\Coh^{\hat G}_{fr}(\hat G\sigma_p)}(\widetilde{V_{\tau}},\widetilde{V_{\tau}})=\mathbf{J}.\]
To proof Part (2), one needs to determine when this pairing is non-degenerate. This itself is an interesting question in representation theory, whose solution relies on the study of the Chevellay's restriction map for vector-valued functions. The determinant of this matrix was calculated in \cite{XZ2}, giving (2). Part (3) was proved by comparison two trace formulas, the Lefschetz trace formula for $G$ and the Arthur-Selberg trace formula for $G'$. We refer to \cite[\S 2]{XZ} for details.

\begin{ex}
Let $G=\on{GU}(1,2r)$ be the unitary similitude group of $(2r+1)$-variables associated to an imaginary quadratic extension $E/\bQ$, whose signature is $(1,2r)$ at infinity. It is equipped with a standard Shimura datum. We fix a level $K\subset G(\bA_f)$ and let $\Sh_K(G,X)$ be the corresponding Shimura variety. In particular if $r=1$, this is the Picard modular surface. Let $p$ be an unramified inert prime.  In this case $\overline{\scrS}_b$ is a union of certain Deligne-Lusztig varieties, parametrized by $G'(\bQ)\backslash G'(\bA_f)/K$, where $G'=\on{GU}(0,2r+1)$ that is isomorphic to $G$ at all finite places. The intersection pattern of these cycles inside  $\overline{\scrS}_b$ were given in \cite{VW}. However, the intersection numbers between these cycles are much harder to compute. In fact we do not know how to compute them directly for general $r$, except applying Theorem \ref{introT: spectral action} to this case. (The case $r=1$ can be handled directly.)

We have $\hat G=\GL_{2r+1}\times \bG_m$ on which $\sigma_p$ acts as $(A,c)\mapsto (J(A^T)^{-1}J,c)$, where $J$ is the anti-diagonal matrix with all entries along the anti-diagonal being $1$. The representation $V_\mu$ is the standard representation of $\GL_{2r+1}$ on which $\bG_m$ acts via homotheties. One checks that $\dim V_\mu^{\Tate}=1$ (which is consistent with the above mentioned parameterization of irreducible components of $\overline{\scrS}_b$ by $G'(\bQ)\backslash G'(\bA_f)/K$). 
We can choose $\tau=((1,\ldots,1),1)$, and think $G'(\bQ)\backslash G'(\bA_f)/K$ as a Shimura set with Shimura cocharacter $\tau$. 
Here we identify the weight lattice of $\hat G$ as $\bZ^{2r+1}\times\bZ$ as usual. Then $\Hom_{\Coh^{\hat G}_{fr}(\hat G\sigma_p)}(\widetilde{V_{\tau}},\widetilde{V_{\mu}})$ is a free rank one module over $\mathbf J$. A generator $\mathbf a_{\on{in}}$ induces a homomorphism $\mathbf J\times\mH^p$-equivariant
$$S(\mathbf a_{\on{in}}): C(G'(\bQ)\backslash G'(\bA_f)/K)\to \on{H}_c^{2r}(\overline{\scrS}_{\overline{\bF}_p},\Ql(r)),$$  
realizing the cycle class map of $\overline\scrS_b$ (up to a multiple). The module  $\Hom_{\Coh^{\hat G}_{fr}(\hat G\sigma_p)}(\widetilde{V_{\mu}},\widetilde{V_{\tau}})$ is also free of rank one over $\mathbf J$. For a chosen generator $\mathbf a_\out$, the composition
\[S(\bba_\out)\circ S(\bba_{\on{in}})=S(\bba_\out\circ\bba_{\on{in}})\]
calculates the intersection matrix of those cycles from the irreducible components of $\overline\scrS_b$. 

The element $h:=\bba_\out\circ\bba_{\on{in}} \in \mathbf J$ was explicitly computed in \cite[Example 6.4.2]{XZ2} (up to obvious modification). Under the Satake isomorphism $\mathbf J\cong H_{G_{\bQ_p}}\otimes \Ql$, it can also be written explicitly as a combination of Hecke operators (\cite{XZ3})
\begin{equation}
\label{E:Hecke operator in Tpj}
h=p^{r(r+1)}\sum_{i=0}^{r} (-1)^i(2i+1)p^{(i-r)(r+i+1)}\sum_{j=0}^{r-i} \begin{bmatrix} 2r+1-2j \\ r-i- j\end{bmatrix}_{v=-p}T_{p,j}.
\end{equation}
Here, $T_{p,j}=1_{K_p\la_j(p)K_p}$, with $\la_i=(1^i,0^{2r-2i+1},(-1)^i,0)$, and $\begin{bmatrix} n \\ m\end{bmatrix}_v$ is the
 $v$-analogue of the binomial coefficient
\[
[0]_v=1,\quad [n]_v=\frac{v^n-1}{v-1},\quad [n]_v!=[n]_v[n-1]_v\cdots[1]_v,\quad \begin{bmatrix} n \\ m\end{bmatrix}_v=\frac{[n]_v!}{[n-m]_v![m]_v!}.
\]
In other words, the intersection matrix of cycles in $\overline\scrS_b$ in this case is calculated by the Hecke operator \eqref{E:Hecke operator in Tpj}. We refer to \cite{XZ3} for details and further discussions.
\end{ex}

\section{Appendix: A category formed by correspondences}
Let $C\rightrightarrows X$ be a groupoid of spaces.
The main goal of the appendix (\S \ref{SS: cat of corr}) is to introduce a category $\on{D}^C(X)$, which captures a large part of information of $\on{D}(Y)$ when $C=X\times_YX$ for a proper morphism $X\to Y$.

We fix a perfect field $k$ and  write $\mathrm{pt}$ for $\Spec k$.
By a stack (resp. a perfect stack) $\mX$, we mean a stack over $\Aff_k$ (resp. $\Aff_k^\pf$) with fpqc topology whose diagonal is representable by an algebraic space (resp. by a perfect algebraic space) and such that there exists smooth (resp. perfectly smooth) surjective map from a scheme (resp. perfect scheme) $X\to \mX$. 
In the sequel, (perfect) stacks of (perfectly) of finite presentation over $k$ are also called spaces. We refer to \cite[Appendix A]{Zh1}, \cite{BS} and \cite[\S A.1]{XZ} for some discussions on the foundations of perfect algebraic geometry.

\subsection{Review of cohomological correspondence}
\label{Sec:cohomological correspondence}
Since we heavily make use of the formalism of cohomological correspondences, we briefly review some facts following \cite[\S A.2]{XZ}. We refer \emph{loc. cit.} for a more thorough treatment.

Let $E$ be a coefficient ring of such that $\cha k$ is invertible in $E$. 
If $X$ is a (perfectly) finitely presented scheme over $k$ and $E$ is finite, let $\on{D}(X)=\on{D}(X,E)$ denote the (unbounded) derived category of \'etale sheaves with $E$-coefficients. For a general space $X$ and possibly more general coefficient ring $E$ (e.g. $E=\Ql$), one can define $\on{D}(X)=\on{D}(X,E)$ via as certain limit (see \cite[\S A.1.15]{XZ} for the precise formulation).
Let $\on{D}^*_c(X)\subset \on{D}(X,E) (*=+,-,b,\emptyset)$ denote its constructible derived categories, and let $\on{P}(X)\subset\on{D}_c^b(X)$ denote the category of perverse sheaves. Let $\omega_X\in \on{D}_c^b(X)$ denote the dualizing sheaf on $X$, and let $\bD\mF=R\underline{\Hom}(\mF,\omega_X)$ denote the Verdier dual. By definition, there is canonical evaluation map
\begin{equation}
\label{E: VD counit}
 \ev_\mF: \Delta^*(\bD\mF\boxtimes\mF)=\bD\mF\otimes \mF\to \omega_X,
\end{equation}
and if $\mF\in\on{D}_c^b(X)$, there is the canonical coevaluation map by taking the Verdier dual
\begin{equation}
\label{E: VD unit}
\coev_\mF: E\to \Delta^!(\mF\boxtimes\bD\mF),\quad
\end{equation}

It is useful to introduce the following notions. 
\begin{dfn}
\label{D: basechangeable}
A commutative square of spaces 
\begin{equation}
\label{E: basechangeable}
\begin{CD}
D@>p>>C\\
@VbVV@VVaV\\
Y@>f>>X
\end{CD}
\end{equation}
is called \emph{base changeable} if the induced morphism $h:D\to C\times_XY$ is representable by (perfectly) proper algebraic spaces.
\end{dfn}

The following lemma is straightforward.
\begin{lem}
\label{L: criterion basechangeable}
\begin{enumerate}
\item A commutative diagram as \eqref{E: basechangeable} is base changeable if
\begin{itemize}
\item either $p$ is representable by (perfectly) proper algebraic spaces and $f$ is separated;
\item or $b$ is representable by (perfectly) proper algebraic spaces and $a$ is separated.
\end{itemize}
\item In the commutative diagram
\begin{equation}
\label{E: two base changeable}
\begin{CD}
E@>q>>D@>p>>C\\
@VcVV@VbVV@VVaV\\
Z@>g>>Y@>f>>X
\end{CD}
\end{equation}
if both inner squares are base changeable, so is the outer rectangle.
\end{enumerate}
\end{lem}

This name is justified by the following fact.

\begin{lem}
\label{L:proper base change}
Assume we have a base changeable commutative diagram of spaces
as \eqref{E: basechangeable}.
Then there is a natural base change homomorphism
\begin{equation}
\label{E: base change isom}
\on{BC}^*_!: a^*f_!\to p_!b^*,
\end{equation}
which is an isomorphism if \eqref{E: basechangeable} is Cartesian.
If $p$ and $f$ are separated, it fits into the commutative diagram 
\[\begin{CD}
a^*f_!@>>> p_!b^*\\
@VVV@VVV\\
a^*f_*@>>>p_*b^*,
\end{CD}\]
where the bottom arrow is the natural adjunction.

In addition, given \eqref{E: two base changeable} with both squares base changeable, the base change homomorphism $a^*(fg)_!\to (pq)_!c^*$ is equal to the composition 
$$a^*(fg)_!=a^*f_!g_!\xrightarrow{\on{BC}_!^*} f_!b^*g_!\xrightarrow{\on{BC}_!^*} f_!g_!c^*=(fg)_!c^*.$$ 
Similarly, $(fg)^*a_!\to c_!(pq)^*$ is equal to the composition
$$(fg)^*a_!=g^*f^*a_!\xrightarrow{\on{BC}_!^*} g^*b_!f^*\xrightarrow{\on{BC}_!^*} c_!g^*f^*=c_!(fg)^*.$$
\end{lem}

\begin{definition}
\label{AD:correspondences}
Let $(X_i,\mF_i)$ for $i=1,2$ be two pairs, where $X_i$ are spaces, and $\mF_i\in \on{D}(X_i,E)$. A \emph{cohomological correspondence} $(C,u):(X_1,\mF_1)\to (X_2,\mF_2)$ is a space $C\xrightarrow{c_1\times c_2} X_1\times  X_2$, together with a morphism $u:c_1^*\mF_1\to c_2^!\mF_2$. We call $C$ the \emph{support} of $(C,u)$. For simplicity, $(C,u)$ is sometimes denoted by $C$ or by $u$. There is an obvious notion of (iso)morphisms between two cohomological correspondences $(C,u)$ and $(C',u')$ from $(X_1,\mF_1)$ to $(X_2,\mF_2)$. The set of isomorphism classes of cohomological correspondences from $(X_1,\mF_1)$ to $(X_2,\mF_2)$ supported on $C$ is denoted by $\on{Corr}_C((X_1,\mF_1),(X_2,\mF_2))$ and by definition
\[\on{Corr}_C((X_1,\mF_1),(X_2,\mF_2))\cong \Hom_C(c_1^*\mF_1,c_2^!\mF_2).\]

If $(C,u):(X_1,\mF_1)\to (X_2,\mF_2)$ and $(D,v):(X_2,\mF_2)\to (X_3,\mF_3)$ are two cohomological correspondences, we define their composition to be $(C\times_{X_2}D,v\circ u):(X_1,\mF_1)\to (X_3,\mF_3)$, with $v\circ u$ being the following composition
\[p^*c_1^*\mF_1\xrightarrow{p^*u} p^*c_2^!\mF_2\xrightarrow{\on{BC}^{*!}} q^!d_1^*\mF_2\xrightarrow{q^!v} q^!d_2^!\mF_3,\]
where $\on{BC}^{*!}$ is the base change homomorphism as in \eqref{E: base change isom}, and $p, q$ are the projections from $C\times_{X_2}D$ to $C$ and to $D$ respectively.
\end{definition}

Here are examples.

\begin{ex}
\label{Ex:examples of correspondences}
\begin{enumerate}
\item For a morphism $f: X\to Y$ and $\mF \in \on{D}(X)$, there is a cohomological correspondence $(X\xleftarrow{\id}X\xrightarrow{f} Y, u:\mF\to f^!f_!\mF)$ from $(X,\mF)$ to $(Y,f_!\mF)$, called the \emph{pushforward correspondence}, and denoted by $(\Ga_f)_!$ for simplicity.

\item For a morphism $f: X\to Y$ and $\mF \in \on{D}(Y)$, there is a natural cohomological correspondence ($Y\xleftarrow{f} X\xrightarrow{\id} X, u:= \mathrm{id}:f^*\mF\to f^*\mF)$ from $(Y,\mF)$ to $(X,f^*\mF)$, called the \emph{pullback correspondence}, and denoted by $\Ga_f^*$ for simplicity.

\item If $(C,u): (X_1,\mF_1)\to (X_2,\mF_2)$ is a cohomological correspondence, with $\mF_i\in \on{D}_c^b(X_i)$, then $(C,\bD u): (X_2, \bD \mF_2)\to (X_1,\bD \mF_1)$ is a cohomological correspondence, called the \emph{dual correspondence}.

\item Let $(X,\mF)$ be a pair with $\mF\in \on{D}_c^b(X)$. Then \eqref{E: VD counit} and \eqref{E: VD unit} give a cohomological correspondence and its dual
$$ \ev_\mF\in \on{Corr}_{X}((X\times X, \bD\mF\boxtimes \mF),(\on{pt}, E)),\quad  \coev_\mF\in \on{Corr}_{X}((\on{pt}, E), (X\times X, \mF\boxtimes\bD\mF)).$$

\item Assume that $k=\bar k$. Let $X_1\leftarrow C\to X_2$ be a correspondence. Then 
\begin{equation*}
\label{E:corr and cycle}
\begin{split}
\on{Corr}_C((X_1,E[d_1]), (X_2, \omega_{X_2}[d_2]))&=\Hom_{\on{D}^b_c(C)}(E[d_1],\omega_C[d_2])\\
&=\on{H}^{\on{BM}}_{d_1-d_2}(C).
\end{split}
\end{equation*}
So if $2\dim C=d_1-d_2$, $\on{Corr}_C((X_1,E[d_1]), (X_2, \omega_{X_2}[d_2]))$ is identified with the top Borel-Moore homology of $C$. In particular,
If $X_1,X_2,C$ are finite sets, regarded as zero dimensional spaces over $k$,  $\on{Corr}_C((X_1,E),(X_2,E))$ can be identified with the space of $E$-valued functions on $C$. 
\end{enumerate}
\end{ex}

Now assume that we have the following commutative diagram
\[\begin{CD}
X_1@<c_1<< C@>c_2>> X_2\\
@Vf_1VV@VVfV@VVf_2V\\
Y_1@<d_1<<D@>d_2>>Y_2
\end{CD}\]
and let $(C, u:c_1^*\mF_1\to c_2^!\mF_2)$ be a cohomological correspondence from $(X_1,\mF_1)$ to $(X_2,\mF_2)$. 
Assume that the left commutative square is base changeable (Definition \ref{D: basechangeable}), then we have 
\begin{equation}
\label{ASS:pushforward correspondence}
d_1^*(f_1)_!\mF_1\xrightarrow{\on{BC}^*_!} f_!c_1^*\mF_1\stackrel{f_!u}{\to} f_!c_2^!\mF_2\to d_2^!(f_2)_!\mF_2,
\end{equation}
where the last map is the natural adjunction.
By abuse of notation, we still denote the above map by $f_!(u)$. Then $(D,f_!(u))$ is a cohomological correspondence from $(Y_1,(f_1)_!\mF_1)$ to $(Y_2,(f_2)_!\mF_2)$. We call it the \emph{pushforward of the cohomological correspondence}.

In particular, if $(Y_1\xleftarrow{d_1} D\xrightarrow{d_2} Y_2)$ is $(S=S=S)$ for some base scheme $S$, and $c_1$ is proper (so the left square is base changeable), we get a homomorphism of  sheaves
\[f_!(u): (f_1)_!\mF_1\to (f_2)_!\mF_2.\]
More specifically, if $S=\Spec k$ is algebraically closed, and $c_1$ is proper, we obtain the induced map on cohomology
$\on{H}_c(u): \on{H}_c^*(X_1,\mF_1)\to \on{H}_c^*(X_2,\mF_2)$.

\begin{rmk}
If $f_1$ is separated and $f$ is representable by perfectly proper algebraic spaces (so the left square is base changeable), then $f_!(u)$ factors as the composition of the natural map $(f_1)_!\mF_1\to (f_1)_*\mF_1$ and
\[d_1^*(f_1)_*\mF_1\to f_*c_1^*\mF_1\to f_!c_2^!\mF_2\to d_2^!(f_2)_!\mF_2.\]
The latter map is the pushforward of cohomological correspondences considered in SGA 5. However, in most situations we are considering, $d_1$ is separated and $c_1$ is representable by perfectly proper algebraic spaces.
\end{rmk}

The pushforward cohomological correspondence is compatible with composition of cohomological correspondences.
\begin{lem}
\label{AL:pushforward compatible with composition}
Suppose we have a commutative diagram
\[
\begin{CD}
X_1@<c_1<< C@>c_2>> X_2@<c'_1<< C'@>c'_2>> X_3\\
@Vf_1VV@VVfV@VVf_2V@VVf'V@VVf_3V\\
Y_1@<d_1<<D@>d_2>>Y_2@<d'_1<<D'@>d'_2>>Y_3
\end{CD}
\]
with the first and the third square base changeable, $c_1,  c'_1$ proper and $d_1,d'_1$ separated, and let $u: c_1^*\calF_1 \to c_2^!\calF_2$ and $v: c'^*_1\calF_2 \to c'^!_2 \calF_3$ be cohomological correspondences from $(X_1, \calF_1)$ to $(X_2, \calF_2)$ and  from $(X_2, \calF_2)$ to $(X_3, \calF_3)$, respectively.
Set $ \widetilde C := C \times_{X_2} C'$ and $\widetilde D: = D \times_{Y_2} D'$ and let $\tilde f: \widetilde C \to \widetilde D$ denote the naturally induced morphism.
Then the following statements hold.
\begin{enumerate}
\item The right square in the diagram 
\[\begin{CD}
X_1@<c_1<<C@<p_C<<\tilde C \\
@Vf_1VV@VVfV@VV\tilde fV\\
Y_1@<d_1<<D@<P_D<<\tilde D
\end{CD}\]
is base changeable. Note that this ensures that $\tilde f_!(v \circ u)$ is well defined by Lemma \ref{L: criterion basechangeable} (1).
\item We have an equality of cohomological correspondences from $(Y_1, (f_1)_!\calF_1)$ to $(Y_3, (f_3)_!\calF_3)$ supported on $\widetilde D$:
\[
\tilde f_!(v\circ u) = f'_!(v) \circ f_!(u).
\]
In particular, if $c_1$ and $c'_1$ are representably by (perfectly) proper algebraic spaces, then
$$\on{H}_c(v\circ u) = \on{H}_c(v) \circ \on{H}_c(u):\on{H}^*_c(X_1\otimes\bar k, \calF_1) \to \on{H}^*_c(X_3\otimes \bar k, \calF_3).$$ 
\end{enumerate}
\end{lem}

Next, assume that we have the following commutative diagram
\[\begin{CD}
Y_1@<d_1<< D@>d_2>> Y_2\\
@Vf_1VV@VVfV@VVf_2V\\
X_1@<c_1<<C@>c_2>>X_2
\end{CD}\]
 Given $u:c_1^*\mF_1\to c_2^!\mF_2$ so that $(C,u)$ is a cohomological correspondence from $(X_1,\mF_1)$ to $(X_2,\mF_2)$, we would like to define the ``pullback of" $(C,u)$.
There are at least three situations where this is possible.
\begin{enumerate}
\item[(1)] Assume that $f$ is (cohomologically) smooth, equidimensional, of relative dimension $d_f$, and let us choose a trace map $\Tr_f: R^{2d_f}f_!E(d)\to E$ (see \cite[Definition A.1.18]{XZ}).
Then we have a cohomological correspondence $(D, f^*(u))$ from $(Y_1, f_1^*\calF_1)$ to $(Y_2, f_2^!\calF_2\langle-2d_f\rangle)$ defined as follows:
\begin{equation}
\label{ASS:smooth pullback correspondence} 
f^*(u)\colon d_1^*f_1^*\mF_1\cong f^*c_1^*\mF_1\xrightarrow{f^*u} f^*c_2^!\mF_2\cong f^!c_2^!\mF_2\langle-2d_f\rangle \cong d_2^!f_2^!\mF_2\langle-2d_f\rangle.
\end{equation}
Here and below, $\langle d\rangle$ denotes the shift and twist $[d](d/2)$ as usual.
We call this the \emph{smooth pullback of the cohomological correspondence}. Note that this depends on a choice of $\Tr_f$.

\item[(2)]
If the left inner square of the above diagram is Cartesian, then we have
\begin{equation}
\label{ASS:smooth pullback correspondence2} 
f^!(u)\colon d_1^*f_1^!\mF_1\xrightarrow{\on{BC}^{*!}} f^!c_1^*\mF_1\xrightarrow{f^!u}f^!c_2^!\mF_2= d_2^!f_2^!\mF_2.
\end{equation}
If in addition $f_1$ is cohomologically smooth, we can choose a trace map $\Tr_{f_1}$ to identify $f^*_1\simeq f^!_1\langle-2d_{f_1}\rangle$ and let $\Tr_f$ be the base change of $\Tr_{f_1}$. Then \eqref{ASS:smooth pullback correspondence2}  and \eqref{ASS:smooth pullback correspondence} coincide. 

\item[(3)] Similarly, if the right inner square of the above diagram is Cartesian, then we can define
\begin{equation}
\label{ASS:smooth pullback correspondence3} 
f^*(u)\colon d_1^*f_1^*\mF_1= f^*c_1^*\mF_1\xrightarrow{f^*u} f^*c_2^!\mF_2\xrightarrow{\on{BC}^{*!}} d_2^! f_2^*\mF_2.
\end{equation}
If $f$ is cohomologically smooth, we can choose $\Tr_f$ as above and let $\Tr_{f_2}$ be the base change of $\Tr_f$ to identify $f_2^*\cong f_2^!\langle-2d_{f}\rangle$. Then
\eqref{ASS:smooth pullback correspondence3}  and \eqref{ASS:smooth pullback correspondence} coincide.
\end{enumerate}

Similar to pushforward cohomological correspondences, pullback of cohomological correspondences is also compatible with composition of cohomological correspondences in certain situations.
\begin{lem}
\label{AL:pullback compatible with composition}
Suppose we have a commutative diagram
\[
\begin{CD}
Y_1@<d_1<<D@>d_2>>Y_2@<d'_1<<D'@>d'_2>>Y_3\\
@Vf_1VV@VVfV@VVf_2V@VVf'V@VVf_3V\\
X_1@<c_1<< C@>c_2>> X_2@<c'_1<< C'@>c'_2>> X_3
\end{CD}
\]
with $f$ cohomologically smooth equidimensional of relative dimension $d_f$. Assume that the third inner square is Cartesian.
Set $ \widetilde C := C \times_{X_2} C'$ and $\widetilde D: = D \times_{Y_2} D'$, and let $\tilde f:\tilde D\to\tilde C$ be the natural map. Then $\tilde f$ is the base change of $f$ along $\tilde C\to C$ and therefore is cohomologically smooth. We choose $\Tr_f$ and let $\Tr_{\tilde f}$ be the base change of $\Tr_f$.
Let $u: c_1^*\calF_1 \to c_2^!\calF_2$ and $v: c'^*_1\calF_2 \to c'^!_2 \calF_3$ be cohomological correspondences from $(X_1, \calF_1)$ to $(X_2, \calF_2)$ and  from $(X_2, \calF_2)$ to $(X_3, \calF_3)$, respectively. Then cohomological correspondence $\tilde f^*(v\circ u): (Y_1, f_1^*\calF_1) \to (Y_3, f_3^!\calF_3\langle 2d_f\rangle )$ is equal to
\[
(Y_1, f_1^*\calF_1) \xrightarrow{f^*(u)} \big(Y_2, f_2^!\calF_2\langle 2d_f\rangle \big) \xrightarrow{f'^!(v)} \big(Y_3, f_3^!\calF_3\langle 2d_f\rangle \big).
\]
\end{lem}

\subsection{A category of correspondences}\label{SS: cat of corr}
Let $(\overleftarrow{c},\overrightarrow{c}): C\rightrightarrows X$ be a groupoid in spaces. In particular, there is the following diagram
\begin{equation*}
\label{E: diagram}
\xymatrix{
&& C\times_XC\ar_{\overleftarrow{p}}[ld]\ar^{\overrightarrow{p}}[dr]\ar^m[d]& &\\
&C\ar_{\overleftarrow{c}}[ld]\ar_{\overrightarrow{c}}[rd]& C\ar_{\overleftarrow{c}}[lld]\ar^{\overrightarrow{c}}[rrd]&C\ar^{\overleftarrow{c}}[ld]\ar^{\overrightarrow{c}}[rd]&\\
X&&X&&X
}
\end{equation*}
where $\overleftarrow{p},\overrightarrow{p}: C\times_XC\to C$ are the projections to the first and the second factor and $m: C\times_XC\to C$ is the multiplication map. There also exists the map
 $\delta: X\to C$ such that $(\overleftarrow{c},\overrightarrow{c})\circ \delta:X\to X\times X$ is the diagonal map. 

Now we assume that $m$ is proper and representable by an algebraic space, and $\delta$ is a closed embedding.
We define an $E$-linear category
$\on{D}^C(X)$ as follows: objects are sheaves on $X$; morphisms are 
\begin{equation}
\label{E:HomDC}
\Hom_{\on{D}^C(X)}(\mF_1,\mF_2)=\Hom_{\on{D}(C)}(\overleftarrow{c}^*\mF_1, \overrightarrow{c}^!\mF_2),
\end{equation}
of cohomological correspondences from $\mF_1$ to $\mF_2$ supported on $C$. 

The formalism of cohomological correspondences reviewed in the previous subsection allows one to define the composition law of the morphisms
\begin{eqnarray*}
&&\Hom_{\on{D}^C(X)}(\mF_1,\mF_2)\otimes\Hom_{\on{D}^C(X)}(\mF_2,\mF_3)\\
&=&\on{Corr}_C((X,\mF_1),(X,\mF_2))\otimes \on{Corr}_C((X,\mF_2),(X,\mF_3))\\
&\to& \on{Corr}_{C\times_XC}((X,\mF_1),(X,\mF_3)) \\
&\xrightarrow{m_!}& \on{Corr}_C((X,\mF_1),(X,\mF_3))\\
&=&\Hom_{\on{D}^C(X)}(\mF_1,\mF_3).
\end{eqnarray*}
It contains a full subcategory $\on{P}^C(X)$ consisting of those $\mF\in\on{D}^C(X)$ that are perverse on $X$.

The motivation to consider the above category is as follows. 
Let $f:X\to Y$ be a proper representable morphism and let $C=X\times_YX$, with $\overleftarrow{c}$ and $\overrightarrow{c}$ the two projections.
By its very definition, 
\begin{lem}
\label{L: fully faithful}
The functor $f_!: \on{D}(X)\to\on{D}(Y)$ factors as
\[\on{D}(X)\to \on{D}^C(X)\to \on{D}(Y)\]
and the second arrow is fully faithful. 
\end{lem}
\begin{proof}
The first statement is clear.
The second statement follows from the canonical isomorphism (by the proper base change)
\[\Hom(\overleftarrow{c}^*\mF_1,\overrightarrow{c}^!\mF_2)=\Hom(f_!\mF_1,f_!\mF_2)\]
which is compatible with the compositions.
\end{proof}

\begin{rmk}
\label{R: DC enh}
(1) 
One reason to make the above definitions is that in some cases even when the stack $Y$ (and therefore $\on{D}(Y)$) is ill-behaved, the category $\on{D}^C(X)$ can still be defined appropriately.

(2) 
The category $\on{D}^C(X)$ introduced as an $E$-linear category as above is sufficient for the applications in this article. However, let us indicate how to define it as an $\infty$-category, at least in the case where both $\overleftarrow{c}$ and $\overrightarrow{c}$ are (perfectly) proper morphisms of separated (perfectly) finite type schemes over a finite or algebraically closed field $k$. Let us denote $\Shv(-)$ the natural $\infty$-categorical enhancement of $\on{D}(-)$. If $X$ is separated and (perfectly) of finite type over a field of finite $\ell$-cohomological dimension, $\Shv(X,E)$ is compactly generated with the compact objects being constructible complexes, when $E=\bF_\ell,\bZ_\ell, \bQ_\ell$, etc.

Notice that  $\on{D}^C(X)$ is a monadic description of the colimit of the diagram in the $2$-category of $E$-linear categories
\[\on{D}(C\times_XC)\ \substack{\longrightarrow\\[-1em] \longrightarrow \\[-1em] \longrightarrow}\ \on{D}(C)\ \substack{\longrightarrow\\[-1em] \longrightarrow} \ \on{D}(X),\]
with the functors given by proper pushforwards.
Then it is clear that one can define an $\infty$-category $\Shv^C(X)$ as the colimit of
\[\cdots \substack{\longrightarrow\\[-1em] \longrightarrow \\[-1em] \longrightarrow\\[-1em] \longrightarrow}\ \Shv(C\times_XC)\ \substack{\longrightarrow\\[-1em] \longrightarrow \\[-1em] \longrightarrow}\ \Shv(C)\ \substack{\longrightarrow\\[-1em] \longrightarrow} \ \Shv(X),\]
with respect to proper push forwards. 
There is a monadic description of $\Shv^C(X)$ in terms of $\Shv(X)$ (by the Barr-Beck-Lurie theorem). In particular, the hom spaces of its homotopy category can be described in terms of cohomological correspondences: Let $\mF_1,\mF_2\in \Shv(X)$, considered as objects in $\Shv^C(X)$ via the natural functor $\Shv(X)\to \Shv^C(X)$. Then 
$$\Hom_{\on{h}\Shv^C(X)}(\mF_1,\mF_2)=R\Hom_{\on{D}(C)}(\overleftarrow{c}^*\mF_1,\overrightarrow{c}^!\mF_2).$$
If $C=X\times_YX$ for some proper surjective morphism $f:X\to Y$, then $f_!:\Shv(X)\to \Shv(Y)$ induces an equivalence of $\infty$-categories $\Shv^C(X)\simeq \Shv(Y)$.
\end{rmk}


\begin{thebibliography}{9999}


\bibitem[BL]{BL}A. Beauville and Y. Laszlo, {\it Conformal blocks and generalized theta functions}, Comm. Math. Phys. 164 (1994), no. 2, 385-419. 

\bibitem[BD]{BD}A. Beilinson, V. Drinfeld, {\it Quantization of Hitchin's integrable system and Hecke
eigensheaves}, Preprint, available at \url{www.math.uchicago.edu/~mitya/langlands}.

\bibitem[BN]{BN}D. Ben-Zvi, D. Nadler, {\it The character theory of a complex group}, arXiv:0904.1247.

\bibitem[BFO]{BFO}R. Bezrukavnikov, M. Finkelberg, V. Ostrik, {\it Character $D$-modules via Drinfeld center of Harish-Chandra bimodules}, Invent. Math. 188 (2012), no. 3, 589-620. 

\bibitem[BKV]{BKV}R. Bezrukavnikov, D. Kazhdan, Y. Varshavsky, {\it A categorical approach to the stable center conjecture}, Ast\'erisque 369 (2015), 27-97.

\bibitem[BS]{BS}B. Bhatt, P. Scholze, {\it Projectivity of the Witt vector affine Grassmannian}, Invent. Math. 209 (2017), no. 2, 329-423.

\bibitem[BG]{BG}K. Buzzard, T. Gee, {\it The conjectural connections between automorphic representations and Galois representations}, Proceedings of the LMS Durham Symposium 2011.

\bibitem[DM]{DM}P. Deligne, J. S. Milne, {\it Tannakian categories}, Hodge cycles, motives, and Shimura varieties, LNM 900 (1982), 101?228.

\bibitem[Dr]{Dr}V. G. Drinfeld, {\it Two-dimensional $\ell$-adic representations of the fundamental group of a
curve over a finite field and automorphic forms on $\GL(2)$}, Amer. J. Math. 105 (1983), 85-114.

\bibitem[EG]{EG}D. Edidin, W. Graham, {\it Equivariant Intersection Theory}, Invent. Math., 131 (1998) 595-634.

\bibitem[Fa]{Fa}G. Faltings, {\it Algebraic loop groups and moduli spaces of bundles}, J. Eur. Math. Soc. 5 (2003), no. 1, 41-68.

\bibitem[FKK]{FKK}B. Fontaine, J. Kamnitzer, and G. Kuperberg, {\it Buildings, spiders, and geometric Satake}, Compos. Math. 149 (2013), no. 11, 1871-1912. 

\bibitem[Ga]{Ga}D. Gaitsgory, {\it From geometric to function-theoretic Langlands (or how to invent shtukas)}, arXiv:1606.09608.

\bibitem[GL17]{GL}A. Genestier, V. Lafforgue, {\it Chtoucas restreints pour les groupes r\'eductifs et param\'etrisation de Langlands locale},  arXiv:1709.00978.

\bibitem[Gi95]{Gi1}V. Ginzburg, {\it Perverse sheaves on a loop group and Langlands' duality}, arXiv:math/9511007.

\bibitem[Gros]{Gr}B. Gross, {\it On the Satake isomorphism}, Galois representations in arithmetic algebraic geometry (Durham, 1996), London Math. Soc. Lecture Note Ser., vol. 254, Cambridge Univ. Press, Cambridge, 1998, pp. 223-237.

\bibitem[GI]{Illusie}
A. Grothendieck, r\'edig'e par L. Illusie, {\it Formule de Lefschetz}, expos\'e III, {\bf SGA} 5, Springer LNM 589 (1977), 73-137.

\bibitem[He]{He} D. Helm, {\it Towards a geometric Jacquet-Langlands correspondence for unitary Shimura varieties}, Duke Math. J. 155 (2010), 483-518.

\bibitem[HTX]{HTX} D. Helm, Y. Tian, and L. Xiao, {\it Tate cycles on some unitary Shimura varieties mod $p$},  Alegbra and Number Theory 11 (2017), 2213-2288. 

\bibitem[Ja]{Ja}U. Jannsen, {\it Motives, numerical equivalence, and semi-simplicity}, Invent. math. 107 (1992), no. 3, 447-452.

\bibitem[Ki]{Ki} M. Kisin, {\it Integral models for Shimura varieties of abelian type}, J. Amer. Math. Soc. 23 (2010), no. 4, 967-1012.

\bibitem[Kl]{Kl}S.L. Kleiman, {\it Motives}, algebraic geometry, Oslo 1970. Oort, F. (ed.), 53-82. Groningen, 1972.

\bibitem[La]{La} V. Lafforgue, {\it Chtoucas pour les groupes r\'eductifs et param\'etrisation de Langlands globale}, arXiv:1209.5352.

\bibitem[Lu1]{Lu}G. Lusztig, {\it Singularities, character formulas, and a $q$-analog of weight multiplicities}, In Analysis and topology on singular spaces, II, III (Luminy, 1981), 208--229, Ast\'erisque, 101-102, Soc. Math. France, Paris, 1983.

\bibitem[Lu2]{Lu2}G. Lusztig, {\it A bar operator for involutions in a Coxeter group}, Bull. Inst. Math. Acad. Sin. (N.S.) 7 (2012), 355-404.

\bibitem[Lu3]{Lu3}G .Lusztig, {\it Truncated convolution of character sheaves}, Bull. Inst. Math. Acad. Sin. (N.S.) 10 (2015), 1-72. 

\bibitem[Lu4]{Lu4}G. Lusztig, {\it Unipotent representations as a categorical centre}, Represent. Th. 19 (2015), 211-235.

\bibitem[Lu5]{Lu5}G. Lusztig, {\it Non-unipotent character sheaves as a categorical centre}, Bull. Inst. Math. Acad. Sin. (N.S.) 11 (2016), 603-731.

\bibitem[Lu6]{Lu6}G. Lusztig, {\it Non-unipotent representations and categorical centres}, Bull. Inst. Math. Acad. Sin. (N.S.) 12 (2017), 205-296.

\bibitem[LV]{LV}G. Lusztig, D. Vogan, {\it Hecke algebras and involutions in Weyl groups}, Bull. Inst. Math. Acad. Sinica (N.S.) 7(2012), 323-354.

\bibitem[LY]{LY}G. Lusztig, Z. Yun, {\it A $(-q)$-analog of weight multiplicities}, Journal of the Ramanujan Math. Soc., 28A (Special Issue-2013) 311-340.

\bibitem[MV]{MV}
I. Mirkovi\'c, K. Vilonen, 
{\it Geometric Langlands duality and representations of algebraic groups over commutative rings}, 
Ann. Math. 166 (2007), no. 1, 95-143.

\bibitem[Ng]{Ng}B.C. Ng\^o, {\it Le lemme fondamental pour les alg`ebres
de Lie}, Publ. Math. Inst. Hautes ƒtudes Sci. No. 111 (2010), 1-169.

\bibitem[Os]{O}V. Ostrik, {\it Pivotal fusion categories of rank $3$}, Mosc. Math. J.15 (2015) 373-396.

\bibitem[PWZ]{PWZ}R. Pink, T. Wedhorn, and P. Ziegler, {\it $F$-zips with additional structure}, Pacific J. Math. 274 (2015), 183-236.

\bibitem[Ri]{Ri}T. Richarz, {\it Affine Grassmannians and geometric Satake equivalences}, Int Math Res Notices, Volume 2016, Issue 12, pp. 3678-3716.

\bibitem[Sc]{S}P. Scholze, {\it $p$-adic geometry}, arXiv:1712.03708.

\bibitem[SW]{SW}P. Scholze, J. Weinstein, {\it $p$-adic geometry}, Berkeley lecture notes,
available at \url{http://www.math.uni-bonn.de/people/scholze/Berkeley.pdf}.

\bibitem[TX]{TX}
Y. Tian and L. Xiao, {\it Tate cycles on quaternionic Shimura varieties over finite fields}, arXiv:1410.2321.

\bibitem[Va]{Vasiu}A. Vasiu, {\it Good Reductions of Shimura Varieties of Hodge Type in Arbitrary Unramified Mixed Characteristic, Part I, II}, arXiv:0707.1668, 0712.1572.

\bibitem[VW]{VW}
I. Vollaard and T. Wedhorn,
The supersingular locus of the Shimura variety of $GU(1,n-1)$, II, {\it
Invent. Math.} {\bf 184} (2011), 591--627.

\bibitem[XZ1]{XZ}L. Xiao, X. Zhu, {\it Cycles on Shimura varieties via geometric Satake}, arXiv:1707.05700.

\bibitem[XZ2]{XZ2}L. Xiao, X. Zhu, {\it On vector-valued twisted conjugate invariant functions on a group}, arXiv:1802.05299.

\bibitem[XZ3]{XZ3}L. Xiao, X. Zhu, {\it Cycles on Shimura varieties via geometric Satake: Examples}, preprint.

\bibitem[Zha]{Zha} C. Zhang, {\it Ekedahl-Oort strata for good reductions of Shimura varieties of Hodge type}, arXiv:1312.4869.

\bibitem[Zh1]{Zh0}X. Zhu, {\it The Geometric Satake Correspondence for Ramified Groups}, Ann. Sci. \'Ec. Norm. Sup\'er., 48 (2015),
409-451.

\bibitem[Zh2]{Zh1}
X. Zhu, {\it Affine Grassmannians and the geometric Satake in mixed characteristic}, Ann. Math. 185 (2017), No. 2, 403-492.

\bibitem[Zh3]{Zh2}
X. Zhu, {\it An introduction to affine Grassmannians and the geometric Satake equivalence}, IAS/Park City Mathematics Series, 24. (2017) 59-154.


\end{thebibliography}
\end{document}